\documentclass[12pt,psfig]{article}
\usepackage{amsthm,amssymb,amsmath,amsfonts,amscd}
\usepackage{graphicx,epsfig,latexsym,hyperref,epstopdf}
\usepackage{mathtools}
\usepackage[space]{grffile}
\usepackage{calligra}
\usepackage{floatrow}
\usepackage{subfig,pgfplots}
\usepackage{circuitikz}
 \usepackage{listings}
\usepackage{caption}
\usepackage{cite,calc,xcolor,mathrsfs,dsfont}
\usepackage[nottoc,numbib]{tocbibind}
\usepackage{xcolor}
\usepackage{harmony}
\usepackage{float}

\newcommand{\LST}{\mathscr{L}}

\newcommand{\MKC}{\mathcal{M}^C_{k, \sigma}}

\newcommand{\Span}{{\rm span}}
\newcommand{\sign}{{\rm sign}}
\newcommand{\im}{{\rm im}}
\newcommand{\ad}{{\rm ad}}
\newcommand{\var}{{\rm var}}

\newcommand{\diag}{{\rm diag}}
\newcommand{\0}{\mathbf{0}}
\newcommand{\x}{\mathbf{x}}

\newcommand{\y}{\mathbf{y}}
\newcommand{\m}{\mathbf{m}}

\newcommand{\C}{C}
\newcommand{\ie}{{\em i.e.,} }
\newcommand{\eg}{{\em e.g.,} }

\newtheorem{thm}{Theorem}[section]

\newtheorem{lem}[thm]{Lemma}

\theoremstyle{definition}
\newtheorem{defn}[thm]{Definition}
\newtheorem{exm}[thm]{Example}
\newtheorem{rem}[thm]{Remark}
\newtheorem{notation}[thm]{Notation}
\numberwithin{equation}{section}
\def\Blem {\begin{lem}}
\def\Elem {\end{lem}}
\def\be {\begin{equation}}
\def\ee {\end{equation}}
\def\ba {\begin{eqnarray}}
\def\ea {\end{eqnarray}}
\def\bes {\begin{equation*}}
\def\ees {\end{equation*}}
\def\bas {\begin{eqnarray*}}
\def\eas {\end{eqnarray*}}
\def\bpr {\begin{proof}}
\def\epr {\end{proof}}
\begin{document}
\baselineskip=18pt
\renewcommand {\thefootnote}{ }

\pagestyle{empty}

\begin{center}
\leftline{}
\vspace{-0.500 in}
{\Large \bf 
Toral CW complexes and bifurcation control in Eulerian flows with multiple Hopf singularities
} \\ [0.3in]

{\large Majid Gazor\(^{\dag}\)\footnote{$^\dag\,$Corresponding author. Phone: (98-31) 33913634; Fax: (98-31) 33912602;
Email: mgazor@iut.ac.ir; Email: ahmad.shoghi@math.iut.ac.ir.} and Ahmad Shoghi }

\vspace{0.105in} {\small {\em Department of Mathematical Sciences, Isfahan University of Technology
\\[-0.5ex]
Isfahan 84156-83111, Iran }}

\today

\vspace{0.05in}

\noindent
\end{center}

\vspace{-0.10in}

\baselineskip=16pt

\:\:\:\:\ \ \rule{5.88in}{0.012in}

\begin{abstract}
We are concerned with bifurcation analysis and control of nonlinear Eulerian flows with non-resonant \(n\)-tuple Hopf singularity. The analysis is involved with {\it CW complex bifurcations} of flow-invariant {\it Clifford hypertori}, where we refer to these toral manifolds by {\it toral CW complexes}. We observe from primary to tertiary flow-invariant toral CW complex bifurcations for one-parametric systems associated with two most generic cases. In a particular case, a tertiary toral CW complex bifurcates from and resides outside a secondary toral CW complex. When the parameter varies, the secondary internal toral CW complex collapses with the origin. However, the tertiary external toral CW complex continues to live even after the secondary internal toral manifold disappears. Our analysis starts with a flow-invariant primary cell-decomposition of the state space. Each open cell admits a secondary cell-decomposition via a smooth flow-invariant foliation. Each leaf of the foliations is a minimal flow-invariant realization of the state space configuration for all Eulerian flows with \(n\)-tuple Hopf singularities. Permissible leaf-vector field preserving transformations are introduced via a Lie algebra structure for nonlinear vector fields on the leaf-manifold. Complete parametric leaf-normal form classification is provided for singular leaf-flows. {\it Leaf-bifurcation} analysis of {\it leaf-normal forms} are performed for three most leaf-generic cases associated with one to three unfolding bifurcation-parameters. Leaf-bifurcation varieties are derived. Leaf-bifurcations provides a venue for cell-bifurcation control of invariant toral CW complexes. The results are implemented and verified using {\sc Maple} for practical bifurcation control of such parametric nonlinear oscillators.
\end{abstract}



\vspace{0.10in} \noindent {\it 2010 Mathematics Subject Classification}:\, 34H20; 55U10; 57N15; 34C20.




\section{Introduction }
Some parts of the proofs are omitted here for briefness. Full version of this paper is available upon request from authors. 

Depending on their applications, oscillations can be beneficial or damaging. Thus, the amplitude-size control, signal synchronization and the oscillation frequency management address important characteristics of parametric oscillator systems. These have many different engineering control applications such as instrumentation in vibrators and robotics \cite{SlotineRobot,GazorShoghiEulNF}, sinusoidal and waveform generators \cite{BookSinusoidalOscillators}, signal processing and information transmission \cite{BookCommCircts}, multi-agent and geometric control systems in robotics and finally, a computer harmonic music design \cite{GazorShoghiMusic}. Many robotic designs such as parallel robots and vibrators need a system design to have a restricted movements within a geometric space (a manifold) as the realization of the state space configuration. Further, a synchronised and an organised harmonic joint movements within the state space configuration are essentially required for an efficient and harmonic operation; also see \cite{GazorShoghiEulNF}. Our proposed treatment of such problems involves an Eulerian type state-feedback controller system design; see \cite{GazorShoghiRobotic}. The Eulerian structure formulates the geometric state space configuration. However, the actual {\it realization} within the state space configuration is controlled by permissible families of flow-invariant leaf-manifolds. Permissible leaves represent the desired realizations of the state space configuration and are determined by integral foliations. Hence, our approach requires the study of integral foliations and the bifurcation control of the governing dynamics on the individual leaves of the foliations; also see \cite{HamziKangCenter05,HamziCharNormalF,Kang98,KangKrener,KangIEEE,GazorShoghiMusic,GazorShoghiRobotic,GazorShoghiEulNF,Langford2HopdSIADS,
MacKayPhysicaD,MacKay,MacKayOnTorus,MacKayCuckPhysicaD,AshwinHopfSN,AshwinN2Oscillators}.

\pagestyle{myheadings} \markright{{\footnotesize {\it M. Gazor and A. Shoghi \hspace{2.2in} {\it  Bifurcation control with multiple Hopf singularity}}}}

{\it Bifurcation} refers to a {\it qualitative type change} in the dynamics of a system when some parameters of the system slightly change around their critical values. Hence, a change in the number and/or the stability types of the equilibria, periodic and/or quasi-periodic orbits and flow-invariant manifolds is called a {\it bifurcation}. We are concerned with non-resonant \(n\)-tuple Hopf singularity with nonlinear Eulerian (and rotational) type coupling throughout this paper. Hopf bifurcation is an important venue to generate {\it limit cycles} from an equilibrium while a coupled \(n\)-tuple Hopf bifurcation for Eulerian flows generates a flow-invariant toral-manifold bifurcated from a steady-state solution. A periodic orbit as an \(\alpha\) or \(\omega\)-limit set is a {\it limit cycle} while an invariant torus in non-resonant cases consists of quasi-periodic solutions. A flow-invariant toral manifold here refers to a family of flow-invariant Clifford hypertori that are smoothly parameterized by open CW-cells of a CW complex. Thus, these toral manifolds are called by {\it toral CW complexes}. A toral CW complex refers to a flow-invariant connected topological space with a partition into hypertorus bundles over open CW-cells of a regular CW complex where they are topologically glued together; see Definition \ref{TorCWDef}. By considering a regular CW complex and its hypertorus bundle, we construct a {\it toral CW complex}. The toral CW complex is the quotient space of the hypertorus bundle over an appropriated constructed equivalence relation; see Lemma \ref{TS1}. This provides an actual description for what singular Eulerian flows in this paper experience. As the parameters slightly change, a parametric Eulerian flow with a multiple Hopf singularity may experience bifurcations of invariant toral CW complexes and, thus, we have a complex oscillating dynamics for the nearby orbits. The cell-bifurcation analysis of invariant toral CW complexes provide the information for our proposed parametric state-feedback controller design and suitable leaf-choices in its control applications. Yet, leaf-bifurcations are chosen and controlled by initial data and bifurcation controller parameters. This approach lays the ground for a desired {\it controlled realization} of the oscillating dynamics in the state space; see \cite{GazorShoghiMusic,GazorShoghiRobotic}.

We are concerned with hypernormalization, bifurcation analysis and control of {\it flows} of nonlinear \(n\)-tuple Hopf singularities given by
\ba\label{Eq1}
&\Theta+v(\x), \;\hbox{ where }\; v(\x):= \sum_{i=1}^{n}\Theta^i_{f_i}+ E_{g}, \; \Theta:=\sum_{i=1}^{n}\omega_i\Theta_{\0}^i, \; {\prod}_{i=1}^{n}\omega_i\neq0, f_i(0)=g(0)=0, &
\ea
\bas
& E_{\0}:= \sum^n_{i=1} x_i\frac{\partial}{\partial x_i}+y_i\frac{\partial}{\partial y_i}, \qquad E_g:= gE_{\0}, \qquad \Theta_{\0}^i:= -y_i\frac{\partial}{\partial x_i}+x_i\frac{\partial}{\partial y_i}, \qquad \Theta^i_{f_i}:= f_i\Theta^i_{\0}, &
\eas \(f_i, g\in \mathbb{R}[[\x]],\) \(\frac{\omega_i}{\omega_j}\notin\mathbb{Q}\) for \(1\leq i< j\leq n\), and \(\x:=(x_1, y_1, \ldots, x_n, y_n)\). We call \(E_g\) and \(\Theta^i_{f_i},\) an Eulerian and a rotational vector field, respectively. Any such vector field is associated with a formal autonomous differential system and vice versa. Thus, the terminology and notations of vectors, vector fields, formal flows, and differential systems are interchangeably used throughout this paper. We simply refer to the formal flows associated with \eqref{Eq1} by Eulerian flows. A bifurcation control of systems of type \eqref{Eq1} does not simply follow the classical normal form theory. This is because classical normal forms usually destroy the Eulerian and rotating structure of these vector fields. As a result, the bifurcation analysis of the truncated classical normal forms does not reflect what occurs in the actual dynamics of the system. Our goal in this paper is to do the analysis by taking into account the structural symmetry of such systems. This paper is the second draft in our project on bifurcation control of singular Eulerian flows with applications in {\it parametric oscillator systems}, {\it robotic team control}, {\it computer harmonic music design and analysis}; also see \cite{GazorShoghiMusic,GazorShoghiRobotic,GazorShoghiEulNF}.

We first provide a flow-invariant primary cell-decomposition of the state space into open-cells {\it invariant under all Eulerian flows}. Each cell is a \(2k\)-manifold for some \(k\leq n\) and is diffeomorphic to the product of \(k\)-copies of the cylinder \(\mathbb{S}^1\times {\mathbb{R}^+}\); see Lemma \ref{CellDecom}. A reduction of a given Eulerian (plus rotational) vector field over a \(2k\)-{\it closed cell} (the {\it \textbf{closure} of an open \(2k\)-cell}) gives rise to an Eulerian (plus rotational) {\it \(2k\)-cell vector field}. Each open cell admits an irreducible flow-invariant \(k+1\)-dimensional foliation. Each leaf of the foliations is an {\it integral manifold} whose {\it tangent bundle} is spanned by all Eulerian vector fields. In other words, leaves are minimal realizations of the state space configuration for all Eulerian flows with multiple Hopf singularity. Each of the leaves is a manifold homeomorphic to \(\mathbb{T}_k\times \mathbb{R}^+,\) where \(\mathbb{T}_k\) is a Clifford hypertorus of \(k\)-dimension for \(1\leq k\leq n\). These leaves are parameterized by positive vectors \(C\) from the \(k-1\)-dimensional unit sphere and an \(n\)-permutation \(\sigma\), say \(\MKC\); see Theorem \ref{ThmMK}. We refer to a vector field reduced to an invariant leaf by a {\it leaf-vector field}. Using permissible changes of state variables, the associated leaf-dependent simplified system is called a {\it leaf-normal form} system. We further allow time rescaling and dependence on bifurcation parameters to obtain (formal) {\it parametric leaf-normal forms}. Finite determinacy analysis and bifurcation analysis of (formal) parametric (leaf) normal forms gives rise to a comprehensive understanding about the local dynamics of the original singular system. We distinguish between a {\it leaf-bifurcation}, a {\it \(2k\)-cell bifurcation}, and a {\it toral CW complex bifurcation}. A {\it leaf-bifurcation} or a {\it leaf-transition} variety is associated with a {\it leaf-vector field}. A {\it \(2k\)-cell bifurcation} is concerned with the dynamics of an Eulerian \(2k\)-cell system. When \(k=n,\) {\it a cell bifurcation} is simply called {\it a bifurcation}. However, {\it a toral CW complex bifurcation} refers to {\it appearance} or {\it disappearance} of a flow-invariant toral CW complex through {\it a bifurcation} or {\it a cell bifurcation}.  This terminology is similar in a sense to the classical {\it limit cycle bifurcations}. The vector field \eqref{Eq1} may experience the leaf-bifurcation of multiple invariant \(k\)-hypertori for \(1\leq k\leq n\) and cell-bifurcations of multiple toral CW complexes.

There are topologically equivalent systems associated different parameters of a parametric Eulerian flow in \(2n\)-dimension whose \(\MKC\)-leaf dynamics
are different; see Theorems \ref{Thms1}-\ref{Thms2}, and compare them with Theorems \ref{LemS1Gamma+}, \ref{TopEqu1}, \ref{Lem7.5}, and \ref{TopEqu}. Therefore, a leaf-bifurcation variety is not necessarily a bifurcation variety for the \(2k\)-dimensional cell systems for \(k\leq n\). This is what enforces our distinction between {\it leaf-bifurcations} and {\it cell-bifurcations}. Cell-bifurcations are here involved with flow-invariant toral manifolds. Due to the complexity of these bifurcated flow-invariant manifolds, we introduce toral CW complexes as a technical means for their comprehensive description. Our definition is specific to the actual flow-invariant manifold cell-bifurcations in Section \ref{SecTCWBif}. We observe a \(2k\)-cell-bifurcation of an isolated secondary flow-invariant toral  CW complex whose leaf-sections within the state space are \(l\)-hypertori for \(1\leq l\leq k\). Secondary flow-invariant toral CW complexes refer to the flow-invariant manifolds bifurcated from an equilibrium (the origin) of a \(2k\)-cell system. The secondary invariant toral CW complex may further undergo a cell-bifurcation. An external tertiary toral CW complex is born from the secondary toral CW complex in a specific case, \ie the secondary toral manifold lives inside the tertiary manifold. Tertiary cell-bifurcations refer to toral CW complexes bifurcating from a secondary toral CW complex (but not from the origin). When the origin is stable, the external toral CW complex is also stable and solutions approach to either the origin or to the external invariant manifold.

The basic tools for analysis is the derivation of normal forms. The idea in normal form theory is to use near-identity changes of state variables to transform a singular vector field into a {\it qualitatively equivalent} but {\it simple} vector field, that is called {\it normal form vector field}. This facilitates the bifurcation analysis of a given {\it nonlinear singular} vector field. It is known that there is a one-to-one correspondence between near-identity transformations and the time-one mappings associated with the flow of {\it nonlinear vector fields without constant and linear parts}; see equations \eqref{v}-\eqref{w}. In the latter case, {\it nonlinear vector fields without constant and linear parts} are called by {\it transformation generators}. Due to the Lie bracket formulation (\(\exp\,\ad \) by equation \eqref{w}) in updating vector fields using time-one mappings, it is an advantage to use the transformation generators in normal form theory. Then, Lie algebraic structures and Lie subalgebras are important tools for structural symmetry-preserving of a given singular system in its normalization process; see \cite{GazorShoghiEulNF}. Our proposed approach here is to make an invariant leaf-reduction and then, obtain leaf-normal forms. Hence, we design a Lie algebra structure for leaf-vector fields on \(\MKC\) through a linear-epimorphism between Eulerian (plus rotational) vector fields on \(\mathbb{C}^{k}\times \mathbb{R}^+\). This Lie algebra structure is required for introducing permissible leaf-preserving transformations for the leaf-normal form derivation. Next, a complete parametric leaf-normal form classification for vector field types \eqref{Eq1} are provided.

This paper is organized as follows. Flow-invariant cell-decomposition of the state space is provided in Section \ref{SecCell}. We show that each cell admits an irreducible foliation that is flow-invariant under all Eulerian flows. Flow-invariant leaf reductions of singular Eulerian flows are introduced in Section \ref{sec4}. Parametric leaf normal form classification is provided in Section \ref{SecLeafNF}. Section \ref{SecBif} deals with bifurcation analysis of one-parametric truncated leaf normal forms of three most generic cases. Bifurcation analysis of \(2k\)-cell systems are described  in Section \ref{SecTCWBif}. A comprehensive description of the \(2k\)-cell bifurcations is achieved by an introduction of toral CW complexes associated with CW complexes. Our normal form approach provides an algorithmically computable method in bifurcation controller design for singular oscillator systems.

\section{Primary cell-decomposition and flow-invariant foliations }\label{SecCell}

This section first presents a primary decomposition of the state space into \(2k\)-dimensional cells invariant under all Eulerian flows. Each cell is homeomorphic to the product of \(k\)-copies of the cylinder \(\mathbb{R}^+\times \mathbb{S}^1.\) Next, we show that each \(2k\)-cell admits a \(k+1\)-dimensional foliation. This further splits the \(2k\)-cells into the minimal flow-invariant manifolds. More precisely, each leaf of the foliation is a {\it minimal} manifold that is {\it invariant} under all singular Eulerian flows and is homeomorphic to a \(\mathbb{T}_k\times \mathbb{R}^+,\) where \(\mathbb{T}_k\) is a \(k\)-dimensional Clifford hypertorus. More precisely, minimal invariant manifolds are integral leaves of the foliations associated with all Eulerian flows. Since the cell decomposition and leaves of the foliations are flow-invariant under all Eulerian flows, we employ these in sequence as a reduction technique for their normal from, bifurcation analysis and control in sections \ref{SecLeafNF} and \ref{SecBif}.

\begin{notation}\label{Not2.1}
\begin{enumerate}
\item We use \(\sqcup_i A_i\) for the {\it union} of {\it disjoint sets} \(A_i.\) The set \(\mathbb{R}^+\) stands for nonzero positive real numbers and
\(\mathbb{T}_k\) for a \(k\)-dimensional Clifford torus. Notation \(\mathbb{S}^{n-1}\) stands for the \(n-1\)-dimensional unit sphere in \(\mathbb{R}^n\). Denote the \(n\)-vector \((c_i)^n_{i=1}\) by \((c_1, c_2, \ldots, c_n)\) and
\be\mathbb{S}^{n-1}_{>0}:= \left\{(c_1, \ldots, c_n)\in \mathbb{S}^{n-1} |\, c_{i}>0 \hbox{ for } 1\leq i\leq n\right\}.\ee
\item Denote
\be\label{Skn} S^k_n:=\{\sigma\in S_n| \sigma(i)<\sigma(j) \hbox{ for } i<j\leq k \hbox{ and for } k<i<j\leq n\}\ee where \(S_n\) is the group of
permutations over \(\{1, 2, \ldots, n\}.\) We denote the identity map in \(S^k_n\) by \(I.\) Thus, \(S^k_n\) has \({n\choose k}\)-number of elements and \(S^n_n= S^0_n=\{I\}.\) For \(\sigma\in S^{k}_n,\) denote
\be\label{sksig}\mathbb{S}^{k-1, \sigma}_{>0}:= \left\{(c_1, \ldots, c_n)\in \mathbb{R}^n \,|\, (c_{\sigma(1)}, \ldots, c_{\sigma(k)})\in \mathbb{S}^{k-1} \hbox{ and } c_{\sigma(i)}>0 \hbox{ for } 1\leq i\leq k\right\}.\ee
Here, \(\sigma(i)\) for \(i=1, \ldots, k\) represents nonzero elements \(c_{\sigma(i)}\neq0\) from \(\mathbf{c}\in \mathbb{S}^{k-1, \sigma}_{>0}\) while \(c_{\sigma(i)}=0\) for \(i=k+1, \ldots, n.\) For an instance, \(\mathbb{S}^{0, \sigma}_{>0}=\{\mathbf{e}^n_{\sigma(1)}\}\) and  \(\mathbb{S}^{k-1, \sigma}_{>0}= \mathbb{S}^{k-1}_{>0}.\) Here, \(\mathbf{e}^n_i\in \mathbb{R}^n\) stands for the \(i\)-th element from the standard basis of \(\mathbb{R}^n.\)
\end{enumerate}
\end{notation}

\begin{lem}[Flow-invariant cell-decomposition of the state space]\label{CellDecom} There exists a disjoint Eulerian flow-invariant decomposition of the state space into open \(2k\)-manifolds \(\mathcal{M}_{k, \sigma}\) for \(1\leq k\leq n,\) and \(\sigma\in S^k_n,\) \ie \(\mathbb{R}^{2n}= \bigsqcup^n_{k=0} \bigsqcup_{\sigma\in S^k_n} \mathcal{M}_{k, \sigma}\) and \(\mathcal{M}_{0, I}=\{\0\}.\) For each \(k=1, \ldots, n\), there are \({n\choose k}\) number of \(2k\)-dimensional cells corresponding to permutations \(\sigma\in S^k_n,\) while cells of odd dimension are empty. Each \(\mathcal{M}_{k, \sigma}\) is diffeomorphic to the product of \(k\) number of cylinders \(\mathbb{S}^1\times\mathbb{R}^+\), \ie \((\mathbb{S}^1\times\mathbb{R}^+)^k.\)
\end{lem}
\bpr For any \bes\0\neq\x=(x, y)=(x_1, x_2, \ldots, x_n, y_1, y_2, \ldots, y_n)\in \mathbb{R}^{2n},\ees denote \((x_{n_j}, y_{n_j})\) (say \(1\leq j\leq k\)) for the nonzero pairs of \(\x\)  and let \(m_j\) with \(1\leq j\leq n-k\) stand for the remaining indices, \ie \((x_{m_j}, y_{m_j})=(0, 0).\) These spaces are pairwise disjoint and their union is the whole state space \(\mathbb{R}^{2n}\). Let \(\rho_{i}(t)= \|(x_i(t), y_i(t))\|\) and \(v:= E_f+\sum_{i=1}^{n}\Theta^i_{g_i}\). Consider the initial value problem associated with \(v\) and initial condition \(\x(t^\circ)\in \mathcal{M}_{k, \sigma}.\) Hence, the manifold \(\{\x\in\mathbb{R}^{2n}|\|(x_i, y_i)\|=0\},\) for an \(i\leq k,\) and its complement are both invariant under the flow associated with \(v.\) Thus, \({\rho}_{\sigma(i)}(t)\neq 0\) for all \(t\) when \(\x(t^\circ)\in \mathcal{M}_{k, \sigma}\) and \(1\leq i\leq k.\) When \(1\leq j\leq n-k\) and \(t\geq t^\circ,\) \({\rho}_{m_j}(t)=0\) for all trajectories starting from a point \(\x(t^\circ)\) on the manifold \(\mathcal{M}_{k, \sigma}.\) Hence, \(\mathcal{M}_{k, \sigma}\) is a \(2k\)-dimensional flow-invariant subspace and the union of \(\mathcal{M}_{k, \sigma}\) over \(k\) and \(\sigma\in S^k_n\) gives rise to \(\mathbb{R}^{2n}\setminus\{0\}\). Define the map \(\phi_{\mathcal{M}_{k, \sigma}}: \mathcal{M}_{k, \sigma}\rightarrow (\mathbb{S}^1\times {\mathbb{R}^+})^k\) by
\ba\label{phiMK}
&\phi_{\mathcal{M}_{k, \sigma}}(\x):= \Big(\frac{(x_{\sigma(i)}(t), y_{\sigma(i)}(t))}{\|(x_{\sigma(i)}(t), y_{\sigma(i)}(t))\|}, \|(x_{\sigma(i)}(t), y_{\sigma(i)}(t))\|\Big)^k_{i=1}.&
\ea
This is well-defined and smooth. Here, the cylinder \(\mathbb{S}^1\times {\mathbb{R}^+}\) is parameterized by \((\cos\theta, \sin\theta, r)\) for \(r\in \mathbb{R}^+\) and \(\theta\in \frac{\mathbb{R}}{2\pi \mathbb{Z}}.\) Therefore, the inverse function \(\phi^{-1}_{\mathcal{M}_{k, \sigma}}: (\mathbb{S}^1\times {\mathbb{R}^+})^k\rightarrow \mathcal{M}_{k, \sigma}\) is a diffeomorphism.
\epr

Now we show that each \(2k\)-manifold \(\mathcal{M}_{k, \sigma}\) admits a secondary decomposition (via a smooth foliation) as a union of disjoint flow-invariant connected \(k+1\)-submanifolds denoted by \(\MKC\). Each submanifold \(\MKC\) is called a leaf. Here, every point has an open neighborhood \(U\) and a local coordinate-system, say \((y_1, \cdots, y_{2k}) : U \rightarrow \mathbb{R}^{2k},\) so that each leaf within \(U\) can be described by \(y_{k+2}=constant, \ldots, y_{2k}=constant.\) The leaves \(\MKC\) of the foliations are parameterized by \(1\leq k\leq n\) and \(\C\in \mathbb{S}^{k-1}_{>0}\). Each leaf \(\MKC\) is a minimal manifold that is invariant under every Eulerian flow with multiple Hopf singularity. Then, each vector field type in equation \eqref{Eq1} can be reduced on these individual flow-invariant minimal leaves. Next, leaf parametric normal form classifications provide further reduction of the vector field \eqref{Eq1}. Hence, the analysis of the infinite level leaf parametric normal form \(v^{(\infty)}_{k, \C}\) for \(1\leq k\leq n\) and \(\C\in \mathbb{S}^{k-1}_{>0}\) provides all bifurcation scenarios of the vector field \(v.\)

\begin{thm}[Irreducible flow-invariant foliations]\label{ThmMK} There is a smooth \(k+1\)-dimensional foliation for further refinements of each \(\mathcal{M}_{k, \sigma}\) into the disjoint leaves \(\MKC\) of the foliations parameterized by \(C\in \mathbb{S}^{k-1, \sigma}_{>0},\) indeed, \(\mathcal{M}_{k, \sigma}=\sqcup_{C\in \mathbb{S}^{k-1}_{>0}}\MKC.\) Each leaf \(\MKC\) is a flow-invariant manifold homeomorphic to \(\mathbb{T}_k\times \mathbb{R}^+\) and \(\mathcal{M}_{k, \sigma}\) is homeomorphic to \(\mathbb{T}_k\times \mathbb{R}^+\times \mathbb{S}^{k-1}_{>0}.\) The set \(\MKC\) is a minimal manifold that is invariant under all flows associated with Eulerian and rotational vector fields.
\end{thm}
\bpr Consider a point \(\0\neq\x^\circ\in \mathcal{M}_{k, \sigma}\). We parameterize the leaves of the foliation by \(C\in \mathbb{S}^{k-1, \sigma}_{>0}.\) Let \(c^\circ_{\sigma(j)}:=\frac{||(x^\circ_{\sigma(j)}, y^\circ_{\sigma(j)})||}{||\x^\circ||}\neq 0\) for \(j\leq k,\) and \(c^\circ_{\sigma(j)}=0\) for \(k<j.\) Thus, \(C_{\x^\circ}:=(c^\circ_1, \ldots, c^\circ_n)\in \mathbb{S}^{k-1, \sigma}_{>0}.\) Denote \((\theta_{\sigma(j)}, \rho_{\sigma(j)})\) for \((x_{\sigma(j)}, y_{\sigma(j)})\) in the polar coordinates and \(\mathcal{N}_{\x^\circ}\subset \mathcal{M}_{k, \sigma}\) for a small open neighborhood around \(\x^\circ\). Note that for all \(\x\in \mathcal{N}_{\x^\circ},\) \((x_{\sigma(i)}, y_{\sigma(i)})\neq 0\) when \(i\leq k\) and \((x_{\sigma(i)}, y_{\sigma(i)})= 0\) for \(i>k.\)
Then, we introduce the map \(\varphi_{\x^\circ}: \mathcal{N}_{\x^\circ}\subset\mathcal{M}_{k, \sigma}\rightarrow \mathbb{R}^{2k}\) by
\bes \varphi_{\x^\circ}(\x):= \left(\theta_{\sigma(1)}, \theta_{\sigma(2)}, \ldots, \theta_{\sigma(k)}, ||(x_{\sigma(1)}, y_{\sigma(1)})||, \frac{c^\circ_{\sigma(1)}\rho_{\sigma(2)}}{c^\circ_{\sigma(2)}}-\rho_{\sigma(1)}, \ldots, \frac{c^\circ_{\sigma(k-1)}\rho_{\sigma(k)}}{c^\circ_{\sigma(k)}}-\rho_{\sigma({k-1})}
\right).\ees The neighborhood \(\mathcal{N}_{\x^\circ}\) can be chosen small enough so that the family \(\varphi_{\x^\circ}\) would construct a {\it smooth} system of local coordinates within the invariant \(\mathcal{M}_{k, \sigma}.\) Let \(\x\neq\y.\) Thus,
\bes \emptyset\neq \varphi_{\x}^{-1}(\mathbb{R}^{k+1}\times \0_{k-1})\cap\varphi_{\y}^{-1}(\mathbb{R}^{k+1}\times \0_{k-1}) \quad\hbox{ if and only if } \quad C_{\x}=aC_{\y}\ees for some \(0\neq a\in \mathbb{R}^+.\) Since \(C_{\x}, C_{\y}\in \mathbb{S}^{k-1,\sigma}_{>0},\) we have \(a=1.\)
Hence, the family \bes \cup_{\{x| C_x=C\}}\varphi_{\x}^{-1}(\mathbb{R}^{k+1}\times \0_{k-1}) \; \hbox{  for } \; C\in \mathbb{S}^{k-1, \sigma}_{>0} \;\hbox{ partitions } \; \mathcal{M}_{k, \sigma}\ees into disjoint connected sub-manifolds. Thereby, these sub-manifolds can be parameterized by \(C\in\mathbb{S}^{k-1,\sigma}_{>0}\). The leaf \(\MKC\) is a \(k+1\)-manifold invariant under all Eulerian flows.
This provides a diffeomorphism between \(\MKC\) and the toral cylinder \(\mathbb{T}_k\times {\mathbb{R}^+}.\)
\epr


\section{Transformation generators, leaf reductions and Lie algebras on a leaf}\label{sec4}

This section is devoted to leaf-reduction of vector fields and the study of leaf-preserving transformation generators using a Lie algebra structure for leaf-vector fields. We first recall how a nonlinear (formal) vector field \(Y(\x)\) generates a near-identity transformation when \(Y(0)= D_\x Y (0)=0,\) and \(D_\x\) stands for derivatives with respect \(x\); \eg see \cite[format 2b]{MurdBook}. Consider the initial value problem
\be\label{v}\frac{d}{dt}{\x}(t, \y)= Y(\x(t, \y)), \quad \x(0, \y)= \y.\ee Then, the time-one mapping \(\x:=\phi_Y(\y )= \x(1; \y)\) is a near-identity coordinate transformation generated by \(Y\). Assume that this transforms the new variable \(\y\) to the old variable \(\x\). Then, a vector field \(v(\x)\) is transformed to
\ba\label{w}
&w(y):= [(D_\y\phi_Y)(y)]^{-1}v(\phi_Y(y))= \exp \ad_Y v= v+ [Y, v]+ \frac{1}{2}\big[Y, [Y, v]\big]+ \cdots, &
\ea
where \(\ad_Yv= [Y, v]:= {\rm Wronskian}(v, Y)= (D_\x Y)v-(D_\x v)Y\); \eg see \cite{Wang2014,Wang3DJDE2014,MurdBook}. Then, system \eqref{v} is transformed into \(\dot{\y}= w(\y).\) Therefore, a Lie subalgebra structure for transformation generators is sufficient to preserve a structural symmetry using these types of transformations. Hence, the time-one flows associated with nonlinear vector fields of type
\ba&\label{Y} Y:=p(\mathbf{u})E_{\0}+\sum_{i=1}^{n} h_i(\mathbf{u})\Theta^i_{\0}, \hbox{ for } u_i=\bar{v}_i, p(\0) = h(\0)=0, &
\ea preserve the structural symmetry type given in \eqref{Eq1}. Therefore, the normal form and (universal asymptotic) unfolding problems of such singular vector fields can be treated using these time-one maps; see \cite{GazorShoghiEulNF,GazorYuSpec}. By asymptotic unfolding problem, we mean finding a \(k\)-truncated simplest parametric normal form with least number of parameters to fully unfold a system with respect to the \(k\)-equivalence relation; \eg see \cite{GazorSadri,GazorSadriBT}. However, the truncated simplest parametric normal form systems are not yet sufficiently simple for bifurcation analysis and control. Thus, we alternatively use the flow-invariant leaf reduction of the vector fields to obtain leaf-vector fields in this section. Given transformation generators described above, we further study the leaf-vector field preserving transformation generators through a Lie algebra structure for leaf-vector fields. Then, in section \ref{SecLeafNF}, we do the parametric normal form of the leaf-vector fields on the manifold \(\mathbb{T}_k\times \mathbb{R}^+\). Next, the estimated transition varieties associated with parametric leaf-normal forms establish a bifurcation control criteria in a parametric state-feedback controlled system. These are indeed necessary for any meaningful bifurcation analysis and bifurcation control of an Eulerian flow with multiple Hopf singularity; see sections \ref{SecLeafNF} and \ref{SecBif}. Now we add some extra notations to those described in Notation \ref{Not2.1}.

\begin{notation}
Given \(\sigma\in S^k_n,\) \(n\)-vectors \(a= (a_1, a_2, \ldots, a_n)\) and \(b:=(b_1, \ldots, b_n),\) we denote \(\hat{a}:= (\hat{a}_1, \hat{a}_2, \ldots, \hat{a}_k):= (a_{\sigma(1)}, a_{\sigma(2)}, \ldots a_{\sigma(k)}),\) \(a^b:= \Pi^n_{i=1}{a_i}^{b_i},\) \(\cos a:=\left(\cos a_{1}, \ldots, \cos a_{n}\right),\) \(\sin\!^{\hat{b}}\hat{a}:=\Pi^k_{j=1}
\sin\!^{b_{n_j}}a_{n_j},\) and \(|a|=\sum^n_{i=1} |a_i|.\)
\end{notation}

\begin{lem}[\(\MKC\)-leaf reduction]\label{LeafReduction}
Consider \(v\) given in equation \eqref{Eq1}, \(g(\x):=\!\sum_{|\alpha|+|\beta|\geq 1}\! a_{\alpha, \beta}x^{\alpha}y^{\beta},\) $f_i(\x):=\sum_{|\alpha|+|\beta|\geq 1} b^i_{\alpha, \beta}x^{\alpha}y^{\beta},$ \(\sigma\in S^k_n,\) \(C\in \mathbb{S}^{k-1, \sigma}_{>0}, k\leq n,\) and \((x, y)\in \mathcal{M}_{k, \sigma}.\) Then, the \(\MKC\)-leaf reduction of \(\Theta+v\) is given by \(\Theta_k+ v_{\sigma, C}\) where \(\Theta_k:=\sum_{j=1}^{k}\omega_{\sigma(j)}\frac{\partial}{\partial \theta_{\sigma(j)}}\) and
\begin{eqnarray}\label{LeafReduced}
&v_{\sigma, C}:=\sum_{|\hat{\alpha}|+|\hat{\beta}|=1}^{\infty} {\rho_{\sigma(k)}}\!^{|\hat{\alpha}|+|\hat{\beta}|}\cos^{\hat{\alpha}}\!{\hat{\theta}}\sin^{\hat{\beta}}\!{\hat{\theta}}
\left(\tilde{a}_{\hat{\alpha}, \hat{\beta}}^k(\C)\sum^k_{j=1}\frac{c_{\sigma(j)}}{c_{\sigma(k)}}\frac{\rho_{\sigma(k)}\partial}{\partial \rho_{\sigma(j)}}+\sum_{j=1}^{k}\tilde{b}_{\hat{\alpha}, \hat{\beta}}^{k, j}(\C)\frac{\partial}{\partial \theta_{\sigma(j)}}\right).&
\end{eqnarray} Here, \(\rho_{\sigma(k)}\in \mathbb{R}^+\) and \(\hat{\theta}:= (\theta_{\sigma(1)}, \theta_{\sigma(2)}, \ldots, \theta_{\sigma(k)})\in \mathbb{T}_k.\) When \(|\hat{\alpha}||\hat{\beta}|<|\alpha||\beta|,\) we have \(\tilde{a}_{\hat{\alpha}, \hat{\beta}}^k(\C)=\tilde{b}^{k, j}_{\hat{\alpha}, \hat{\beta}}(\C)=0.\) Otherwise,
\begin{eqnarray}\label{alphatilde}
&\tilde{a}_{\hat{\alpha}, \hat{\beta}}^k(\C):=\frac{\prod_{j=1}^{k}{c_{\sigma(j)}}^{\alpha_{\sigma(j)}+\beta_{\sigma(j)}}}{{c_{\sigma(k)}}^{|\alpha|+|\beta|}}a_{\alpha, \beta},\quad
\tilde{b}^{k, j}_{\hat{\alpha}, \hat{\beta}}(\C):=\frac{\prod_{j=1}^{k}{c_{\sigma(j)}}^{\alpha_{\sigma(j)}+\beta_{\sigma(j)}}}{{c_{\sigma(k)}}^{|\alpha|+|\beta|}}b^j_{{\alpha}, {\beta}}& \quad\text{for}\quad 1\leq j\leq k.
\end{eqnarray}
\end{lem}
\begin{proof} Proof is omitted. \epr


Since \(\MKC\) is homeomorphic to \(\mathbb{T}_k\times \mathbb{R}^+\), we may identify the leaf-vector field \eqref{LeafReduced} with an Eulerian type vector field plus a rotational vector field on \(\mathbb{T}_k\times \mathbb{R}^+\). This is described as follows. Let \((\varrho_j, \vartheta_j)\) stands for \((\rho_{\sigma(j)}, \theta_{\sigma(j)}).\) Thus, let \((\mathscr{X}_i, \mathscr{Y}_i):=(\cos\vartheta_i, \sin\vartheta_i),\) \((\mathscr{X}, \mathscr{Y})\in \mathbb{T}_k= \mathbb{S}^1\times\mathbb{S}^1\cdots\times\mathbb{S}^1\) and denote
\ba\label{LkC}
&\mathscr{L}_{\mathbb{T}_k\times \mathbb{R}^+}:= \underset{1 \leq i \leq k}{\Span}\!\left\{\frac{q_i\partial}{\partial \vartheta_i}, \sum^k_{j=1}\frac{\hat{c}_j\varrho_k h\partial}{\hat{c}_k\partial \varrho_j}\Big|\, h, q_i\in \mathbb{R}[[\mathscr{X}, \mathscr{Y}, \varrho_k]], q_i(\0)=h(\0)=0\right\}, &\\\nonumber
&\mathcal{R}:=\Span\{g\in\mathbb{C}[[\upsilon_{1}, \ldots, \upsilon_{k}, r]]\,|\,\upsilon_i\in\{z_i, w_i\}\, \hbox{ for }\, i\leq k\},&
\ea where \(\mathbf{z}:=(z, w)=(z_1, \ldots, z_k, w_1, \ldots, w_k),\) \(w_i= \overline{z}_i.\)
We remark that the monomials appearing in \(\mathcal{R}\) are in terms of either \(z_i\) or \(w_i\) but, a monomial cannot include both \(z_i\) and \(w_i\) for the same index \(i\), \eg \(z_1z_2w_3\in\mathcal{R}\) and \(z_1w_1\notin\mathcal{R}\). The main goal in the next lemma is to construct a leaf-dependent Lie algebra structure over the class of leaf-vector fields. The idea is to use the pushforward maps associated with projection of coordinate changes from a complex coordinate system on \(\mathbb{C}^k\times \mathbb{R}^+\) onto \(\mathbb{T}_k\times \mathbb{R}^+\). The Lie algebra structure introduces permissible transformation generators for their leaf-normal form classification.

\begin{thm}[Lie algebra structure on invariant leaves]\label{FractLie}  There exists a linear-isomorphism
\begin{eqnarray}\label{J}
&\Psi: \mathscr{L}_{\mathbb{T}_k\times \mathbb{R}^+}\rightarrow\mathscr{J}:= \underset{1 \leq i \leq k}{\Span}\left\{\frac{r g\partial}{\partial r}, \frac{f_iw_i\partial}{\mathbf{i}\partial w_i}-\frac{f_i z_i\partial}{\mathbf{i}\partial z_i}\Big| g, f_i\in \mathcal{R}\right\},\;\;\;&
\end{eqnarray} where \(f_i(z, w, r)=\overline{f_i}(z, w, r),\) \(g(z, w, r)=\overline{g}(z, w, r),\) \(f_i(\0)=g(\0)=0,\) and \(\varrho, r\in \mathbb{R}^+.\) Further, there are Lie algebra structures on \(\mathscr{L}_{\mathbb{T}_k\times \mathbb{R}^+}\) and \(\mathscr{J}\) so that \(\Psi\) is a Lie isomorphism.
\end{thm}
\begin{proof} Proof is omitted for briefness. \epr

\begin{rem}
An alternative dynamics reduction can be made using projective space of the state space. However, this is fruitless due to the fact that the dynamics on the projective space is trivial.
\end{rem}

\section{Parametric leaf-normal form classification }\label{SecLeafNF}

This section is devoted to derive the formal leaf-normal forms of singular systems with multiple Hopf singularity. We use near-identity changes of the state variables. For the parametric vector fields, we also use the rescaling of time and it is also important to allow the state- and time-transformations to depend on the bifurcation parameters; see \cite{Langford2HopdSIADS,LangfTori,AnnMathZung,YuHopfZero,YuNonlinearity,StolovitchNonl,StolovitchAnnMath,ZoladekNF2015,ZoladekNF2015,Stroyzyna17,MacKay,MacKayPhysicaD,Wang3DJDE2014,Wang2014,NFTAMS07,Knobloch86,KangIEEE,IoosNF,Walcher2012,AshwinHopfSN} for a recent literature on normal forms, convergence and their optimal truncations. Given the proof of Theorem \ref{FractLie} and the convenience of notations, we identify leaf-vector fields with those on \({\mathbb{T}_k\times \mathbb{R}^+}.\) Then, our leaf-normal form derivation uses the map \({\hat{{\psi}}_*}^{-1}\) to transform a vector field on \({\mathbb{T}_k\times \mathbb{R}^+}\) into a vector field on \(\mathbb{C}^{k}\times \mathbb{R}^+\). Next, normal forms are derived and then, the pushforward map \(\psi_*\) is applied to project the normal form vector field back to a normal form vector field on \(\mathbb{T}_k\times\mathbb{R}^+.\)

\begin{thm}[The first level \(\MKC\)-leaf normal form]\label{1stLeafNF} For any \(C\in \mathbb{S}^{k-1}_{>0},\) there exists a near-identity changes of the state variables transforming the \(\MKC\)-leaf vector field \(\Theta+v_{\sigma, C}\) given by \eqref{LeafReduced} into a first level \(\MKC\)-leaf normal form \(\Theta_k+v^{(1)}_{\sigma, C},\) where
\begin{eqnarray}\label{LeafNormalForm}
& v^{(1)}_{\sigma, C}:= \sum_{p=0}^{\infty}({x_{\sigma(k)}}^2+{y_{\sigma(k)}}^2)^{p}\mathbf{A}^\sigma_{p}(x_{\sigma(1)}, y_{\sigma(1)}, \ldots, x_{\sigma(k)}, y_{\sigma(k)})^t, &\\\nonumber&
\mathbf{A}^\sigma_{p}:=\diag(\mathbf{A}^\sigma_{1, p}, \mathbf{A}^\sigma_{2, p}, \ldots, \mathbf{A}^\sigma_{k, p}), \quad
\mathbf{A}^\sigma_{i, p}:=({a_{p}}^2+{b^i_{p}}^2)^\frac{1}{2} R_{\theta_{p}^i}, &
\end{eqnarray} \(R_{\theta_{p}^i}\) is the standard counterclockwise rotation matrix, \(\theta_{p}^i:=\tan^{-1}\left(\frac{b^i_{p}}{a_{p}}\right)\) is the rotation angle, and \(a_{0}=0, b^i_{0}=\omega_{\sigma(i)}\) for \(1\leq i\leq k.\) The coefficients \(a_{p}\) and \(b^{j}_p\) are \(C\)-dependent polynomials in terms of \(\tilde{a}_{\hat{\alpha}, \hat{\beta}}^k\) and \(\tilde{b}^{k, j}_{\hat{\alpha}, \hat{\beta}}\) given in equation \eqref{alphatilde} for \(|\hat{\alpha}+\hat{\beta}|\leq p\).
\end{thm}
\bpr The proof follows direct calculations.
\epr


The family of the first level leaf-normal form vector fields of type \eqref{L1st} is a Lie subalgebra in \(\mathscr{L}_{\mathbb{T}_k\times \mathbb{R}^+}.\) Hence, we denote them by 
\(\LST\!:= \underset{1 \leq i \leq k}{\Span}\left\{\frac{q_i\partial}{\partial \vartheta_i},
\sum^k_{j=1}h\frac{\hat{c}_j\varrho_k \partial}{\hat{c}_k\partial \varrho_j}\,\big|\,  q_i, h\in \mathbb{R}[[{\varrho_k}^{2}]], m\geq 1\right\}.\) A permissible direction-preserving time-rescaling is given by \(\tau:=(1+T)t,\) where \(T({x_{\sigma(k)}}\!^2+{y_{\sigma(k)}}\!^2)\) is a formal power series without constant terms. This time rescaling transforms a vector field \(v\in\LST\) into \(v+Tv\in \LST.\) Hence, we consider \(\mathscr{R}\) as the formal power series generated by \(Z_i:=\left({x_{\sigma(k)}}\!^2+{y_{\sigma(k)}}\!^2\right)\!^i\) and denote
\(\hat{E}_{\0}:=\sum^k_{j=1}\frac{\hat{c}_j\varrho_k \partial}{\hat{c}_k\partial \varrho_j}.\)

\begin{lem}
The vector space \(\LST\) constitutes a Lie subalgebra in \(\mathscr{L}_{\mathbb{T}_k\times \mathbb{R}^+}\) so that for every
\(\alpha, \beta\in \mathbb{R}\), \(m, n\in \mathbb{Z}^{\geq 0}\) and \(1\leq i\leq j\leq k\), the structure constants are given by
\ba\nonumber
&\left[{\rho_{\sigma(k)}}\!^{2m} \hat{E}_{\0}, \frac{\alpha}{2}{\rho_{\sigma(k)}}\!^{2n} \hat{E}_{\0}\!+\!\frac{\beta}{2}{\rho_{\sigma(k)}}\!^{2l} \Theta^{\sigma(i)}_{\0}\right]\!=\!(m\!-\!n)\alpha{\rho_{\sigma(k)}}\!^{2(m\!+\!n)}\hat{E}_{\0}\!-\!l\beta {\rho_{\sigma(k)}}\!^{2(m+l)}\Theta^{\sigma(i)}_{\0},&\\\label{LeafStructureConstant}
&\left[{\rho_{\sigma(k)}}\!^{2m} \Theta^{\sigma(i)}_{\0}, {\rho_{\sigma(k)}}\!^{2n} \Theta^{\sigma(j)}_{\0}\right]\!=\!0.\,\,&
\ea
The Lie algebra \(\LST\) is also an \(\mathscr{R}\)-module that is consistent with time rescaling of vector fields, where for every \(\beta_i\in
\mathbb{R}\) and \(1\leq i\leq k,\)
\begin{eqnarray}\label{Action R Over algebra}
&Z_m {\rho_{\sigma(k)}}\!^{2n} \hat{E}_{\0}={\rho_{\sigma(k)}}\!^{2(m+n)} \hat{E}_{\0},\quad Z_m\sum_{i=1}^{k} {\beta_i\rho_{\sigma(k)}}\!^{2\sigma(i)} \Theta^{\sigma(i)}_{\0}=\sum_{i=1}^{k}\beta_i{\rho_{\sigma(k)}}\!^{2(m+\sigma(i))} \Theta^{\sigma(i)}_{\0}.&
\end{eqnarray} 
\end{lem}
\begin{proof} The argument for relations \eqref{LeafStructureConstant} and \eqref{Action R Over algebra} follows direct calculations.
\end{proof}
The leaf-normal form \eqref{LeafNormalForm} in polar coordinates reads
\begin{eqnarray}\label{CartesianForm}
&\Theta_k+ \sum^k_{i=1}\sum_{p\geq 1}{\rho_{\sigma(k)}}\!^{{2p}}\,\frac{a_{p} c_{\sigma(i)}\rho_{\sigma(k)}\partial}{c_{\sigma(k)}\partial \rho_{\sigma(i)}}+ \sum_{j=1}^{k}\sum_{p\geq 1} b^{j}_p{\rho_{\sigma(k)}}\!^{ 2p}\frac{\partial}{\partial \Theta_{\sigma(j)}}.&
\end{eqnarray} The first level parametric normal form of every vector field \eqref{Eq1} is similar to equation \eqref{CartesianForm}, except that the coefficients \(a_{m}\) and \(b_{m}^i\) depend on the parameter vector \(\mu:=(\mu_1, \mu_2, \ldots, \mu_r)\) and we denote them by \(a_{m}(\mu, C)\) and \(b^i_{m}(\mu, C)\) where
\bes a_{0}(0,C)=0, \; b^i_{0}(0,C)=0\quad \hbox{ for }\quad 1\leq i\leq n.\ees
Hence, parametric version of \eqref{CartesianForm} with respect to Eulerian and rotational vector fields in \(\LST\) is expressed by
\begin{eqnarray}\label{PrametricReducedNF}
&\Theta_k+v^{(1)}_{\sigma, \C}:=\sum_{i=1}^{k}\omega_{\sigma(i)}\Theta^{\sigma(i)}_{\0}+\sum_{j\geq 0}a_j(\mu, C){\rho_{\sigma(k)}}\!^{2j} \hat{E}_{\0}+\sum_{i=1}^{k}\sum_{j\geq 0}b_{j}^i(\mu, C){\rho_{\sigma(k)}}\!^{2j}\Theta^{\sigma(i)}_{\0}, &
\end{eqnarray} where we take the notation \(\m:=(m_1, \ldots, m_r)\),\,\(\mu^{\m}:={\mu_1}^{m_1}\cdots{\mu_r}^{m_r},\)
\begin{eqnarray*}
&a_j(\mu, C)=\sum_{|\m|\geq 0}a_{j, \m}\!\left(C\right)\mu^{\m},\; \hbox{ and }\; b^i_j(\mu, C)=\sum_{|\m|\geq 0}b^i_{j, \m}\!\left(C\right)\mu^{\m}\text{\quad for\quad} 1\leq i\leq k.&
\end{eqnarray*} In order to do the hyper-normalisation of vector fields, we define
\begin{equation}\label{seq}
s:=\min\{j|\, a_{j}(\0, C)\neq 0\}
\end{equation} and a grading function by
\bes \delta(\mu^{\m}{\rho_{\sigma(k)}}\!^{2j}\hat{E}_{\0}):=|\m|(s+1)+j,\, \delta(\mu^{\m}{\rho_{\sigma(k)}}\!^{2j}\Theta^{\sigma(i)}_{\0}):=|\m|(s+1)+s+j \quad \hbox{ for } \; 1\leq i\leq k.\ees The grading \(\delta\) decomposes the Lie algebra \(\LST=\sum \LST_i\) into \(\delta\)-homogeneous spaces as a graded Lie algebra, \([\LST_i, \LST_j]\subseteq \LST_{i+j}\). Further, it will be a \(\mathscr{R}\)-graded module.
Let \(\Theta_k+v^{(1)}_{\sigma, \C}:=\sum^\infty_{i=0} v_i\) and \(v_i\in \LST_i\). The map \(d^{i, 1}(T_i, S_i):= T_iv_0+[S_i, v_0]\) is defined for \((T_i, S_i)\in \mathcal{R}_i\times \mathscr{L}_i.\) Using a slight abuse of notation, we inductively define the map
\(d^{i, r}: \ker d^{i-1, r-1}\times \mathcal{R}_i\times \mathscr{L}_i\rightarrow \mathscr{L}_i\) by
\begin{eqnarray*}\label{d}
& d^{i, r}(T^{r-1}_{i-r+1}, \cdots, T^{r-1}_{i-1}, T_i,
S^{r-1}_{i-r+1}, \cdots, S^{r-1}_{i-1}, S_i):= \sum^{r-1}_{i=1} \left(T^{r-1}_{i-i}v_i+ [S^{r-1}_{i-i}, v_i]\right) + T_iv_0+[S_i, v_0],&
\end{eqnarray*}
where \((T^{r-1}_{i-r+1}, \cdots, T^{r-1}_{i-1}, S^{r-1}_{i-r+1}, \cdots, S^{r-1}_{i-1})\in \ker d^{{i-1}, r-1},\) \(i\geq r,\) and \(r\geq 2\). The map \(d^{i, r}\) computes all possible spectral data available as transformation generators to simplify terms in grade \(i\) by using not only the linear part of the vector field but also all terms in the normalising vector field up-to grade \(r-1.\) Thus, \({\rm im}\, d^{i, r}\) represents the space that can be simplified from the vector field in grade \(r\) while a complement space \(\mathcal{C}^r_i\) to \({\rm im}\, d^{i, r},\) for any \(i,\) stands for all possible terms that may not be simplified in the \(r\)-level normal form step; \eg see \cite[Theorem 4.3 and Lemma 4.2]{GazorYuSpec}. Hence, using near-identity changes of state variable and time rescaling (direction preserving), the vector field \(\Theta_k+v^{(1)}_{\sigma, \C}\) can be transformed into a \(r\)-th level parametric normal form \(v^{(r)}:= \sum v_i^{(r)}\), where \(v^r_i\in \mathcal{C}^r_i\). Deriving \(d^{i, i}\) and \(\mathcal{C}^i_i\) for any \(i\geq 1\) gives rise to the computation of the infinite level (simplest) normal form; \eg see \cite{Wang3DJDE2014,GazorYuSpec}.


\begin{thm}[Infinite-level parametric leaf-normal forms]\label{InfLPNF} Consider $C\in \mathbb{S}^{k-1}_{>0}$ and a \(\MKC\)-leaf. Then, there are a natural number \(s,\) near-identity changes of state variables and time-rescaling such that they transform the parametric leaf-normal form \eqref{PrametricReducedNF} into the infinite-level parametric leaf-normal form \(\Theta_k+w^{(\infty)}_{\sigma, \C}\) where \(w^{(\infty)}_{\sigma, \C}\) is given by
\begin{small}\begin{eqnarray}\label{ReducedParametricNormalForm}
&\sum^{k, s}_{i=1, j=0}({x_{\sigma(k)}}^2\!+\!{y_{\sigma(k)}}^2)^j\!\Big((\hat{a}_{j}(\mu,\!C) x_{\sigma(i)}\!-\!\hat{b}^i_{j}(\mu,\! C) y_{\sigma(i)})\frac{\partial}{\partial x_{\sigma(i)}}\!+\!(\hat{a}_{j}(\mu,\!C) y_{\sigma(i)}\!+\!\hat{b}^i_{j}(\mu,\!C) x_{\sigma(i)})\frac{\partial}{\partial y_{\sigma(i)}}\Big), \;&
\end{eqnarray}\end{small} where \(\hat{a}_{j}(\0,\!C)=\hat{b}^i_{0}(\0,\! C)=\hat{b}^1_{j}(\mu,\! C)=0\) for \(0\leq j\leq s-1,\) \(1\leq i\leq k,\) and \(0\neq a_s\!:=\!a_s(\0,\!C)=a_s(\mu,\!C).\)
\end{thm}
\begin{proof} The index for zero vectors indicate their dimension. For \(\rho:= \rho_{\sigma(k)},\) we have
\begin{eqnarray*}
&d^{|\m|(s+1)+s+j,s+1}\left(\gamma_{j,\m}\mu^{\m}Z_{j},\0_s,\alpha_{j,\m} \mu^{\m}{\rho}^{2j}E_{\0}+\sum_{i=1}^{k}\beta^i_{j-s,\m} \mu^{\m}{\rho}^{2(j-s)}\Theta^{\sigma(i)}_{\0},\0_s\right)&\\
&=a_s\left(\gamma_{j, \m}+2\alpha_{j, \m}(j-s)\right)\mu^{\m}{\rho}^{2(s+j)}\hat{E}_{\0}+\sum_{i=1}^{k}\left(\omega_{\sigma(i)}\gamma_{j, \m}+2\beta^i_{j-s,\m}(j-s)a_s\right) \mu^{\m}{\rho}^{2j}\Theta^{\sigma(i)}_{\0}.&
\end{eqnarray*} This implies that all terms \(\mu^{\m}{\rho}^{2(s+j)}\hat{E}_{\0}\) for \(j\geq 0\) and \(\mu^{\m}{\rho}^{2j}\Theta^{\sigma(1)}_{\0}\) for \(j\geq 1\) are simplified from the \(s+1\)-th level orbital leaf-normal form system. Consider the case \(j=s\). Hence, we take \(\alpha_{s,\m}=\beta^i_{0,\m}=0\) for \(i=1, \ldots, k\) and \(\gamma_{s,\m}:=-\frac{1}{{a_s}}a_{2s,\m}(C)\). Therefore, in \(s+1\)-th level parametric leaf-normal form, the Eulerian terms \(\mu^{\m}{\rho}^{2(s+j)}\hat{E}_{\0}\) for \(j\geq 1, |\m|\geq 0\) and \(\mu^{\m}{\rho}^{2s}\hat{E}_{\0}\) for \(|\m|>0\) are normalized in this level. Further, the rotating terms \(\mu^{\m}{\rho}^{2(s+j)}\Theta^{\sigma(i)}_{\0}\) for \(j\geq 1, |\m|\geq 0, 2\leq i\leq k\) and \(\mu^{\m}{\rho}^{2j}\Theta^{\sigma(1)}_{\0}\) for all \(s\neq j\in \{0\}\cup\mathbb{N}\) are simplified. This is indeed the equation \eqref{ReducedParametricNormalForm} in polar coordinates. Time rescaling generators \(\mu^{\m}Z_{j}\) and state transformation generators \(\mu^{\m}{\rho}^{2j}\hat{E}_{\0}\) for \(j>s\) do not have any influence in enlarging the space \(\im\, d^{s+l, s+j+1}\) in normalization levels higher than \(s+1.\) On the other hand, terms \(\mu^{\m}Z_{j}\) for \(j\leq s\) and \(\mu^{\m}{\rho}^{2j}\hat{E}_{\0}\) for \(j<s\) have already contributed in \(\im\, d^{s+l, s+1}\). Further, \(\im\, d^{s+l, s+1}\subseteq \im\, d^{s+l, s+j+1}.\) Therefore,
\(\im\, d^{s+l, s+j+1}=\im\, d^{s+l, s+1}\) for all \(l>j.\) This completes the proof.
\end{proof}

\section{Leaf-bifurcation analysis }\label{SecBif}

When specific initial values are chosen, the state space configuration is realized within an individual flow-invariant leaf. Then,
leaf-transition varieties make a partition for the parameter space into connected regions. All parameters from an open connected region corresponds to qualitatively the same dynamics for the parametric leaf-vector field. Hence, leaf-varieties classify the {\it persistent} qualitative dynamics of the leaf vector field {\it subjected to small parameter-perturbation}. Thus, the individual leaf-choices and leaf-bifurcations contribute into the state-feedback controller designs in practical bifurcation control applications. This section studies the leaf-bifurcations associated with a leaf-parametric normal form \eqref{ReducedParametricNormalForm} for three most generic cases \(s=1, 2, 3\).

\subsection{Leaf cases $s=1$ and \(s=2\) }

\begin{thm}[Leaf case \(s=1\)]\label{Thms1} Consider the parametric leaf-normal form \eqref{ReducedParametricNormalForm} when \(s=1\) in equation \eqref{seq} for some \(2\leq k\leq n\) and \(C\in \mathbb{S}^{k-1, \sigma}_{>0}.\) Then, there is a leaf-dependent bifurcation of an invariant \(\mathbb{T}_k\)-torus from the origin. This leaf-bifurcation is three-determined and its associated bifurcation variety is given by \(T_{Pch}:=\{\nu_0| \nu_0=0\}.\) When \(\nu_0>0\) and \(a_1<0\), the invariant torus is stable while the origin is unstable. For \(a_1>0\) and \(\nu_0<0\), the origin is stable and the \(\mathbb{T}_k\)-torus is repelling.
\end{thm}
\bpr The assumption is equivalent with \(a_1:=a_{1}(0, C)\neq 0\). For \(i=k\) and only looking for the steady-state solutions, this represents a normal form for subcritical and supercritical pitchfork bifurcation at \((\rho_{\sigma(k)}, \nu_0)=(0, 0)\) when \(a_1>0\) and \(a_1<0,\) respectively. Thus, this is a three-determined differential system. The none-zero equilibrium of this system for $a_1\nu_0<0\) is associated with \(\rho^*=(\rho_1, \ldots, \rho_n) =\sqrt{\frac{-\nu_0}{a_1}}\frac{C}{c_{\sigma(k)} }\). Since \(\frac{d}{d \rho_{\sigma(k)}}f_{k, \sigma(k)}\left(\sqrt{\frac{-\nu_0}{a_1}}, \nu_0, C\right)= 2a_1\sqrt{\frac{-\nu_0}{a_1}}<0\) for \(\nu_0>0\) and \(a_1<0,\) an asymptotically stable \(\mathbb{T}_k\)-torus bifurcates from the origin and the origin is unstable. For \(\nu_0<0\) and \(a_1>0,\) an unstable \(\mathbb{T}_k\)-torus bifurcates from the origin while the origin is asymptotically stable.
\epr

\begin{figure}[t]
\centering
\subfloat[Transition sets when $a_2=-1.$ \label{S2a2-}]{
\includegraphics[width=2.2in]{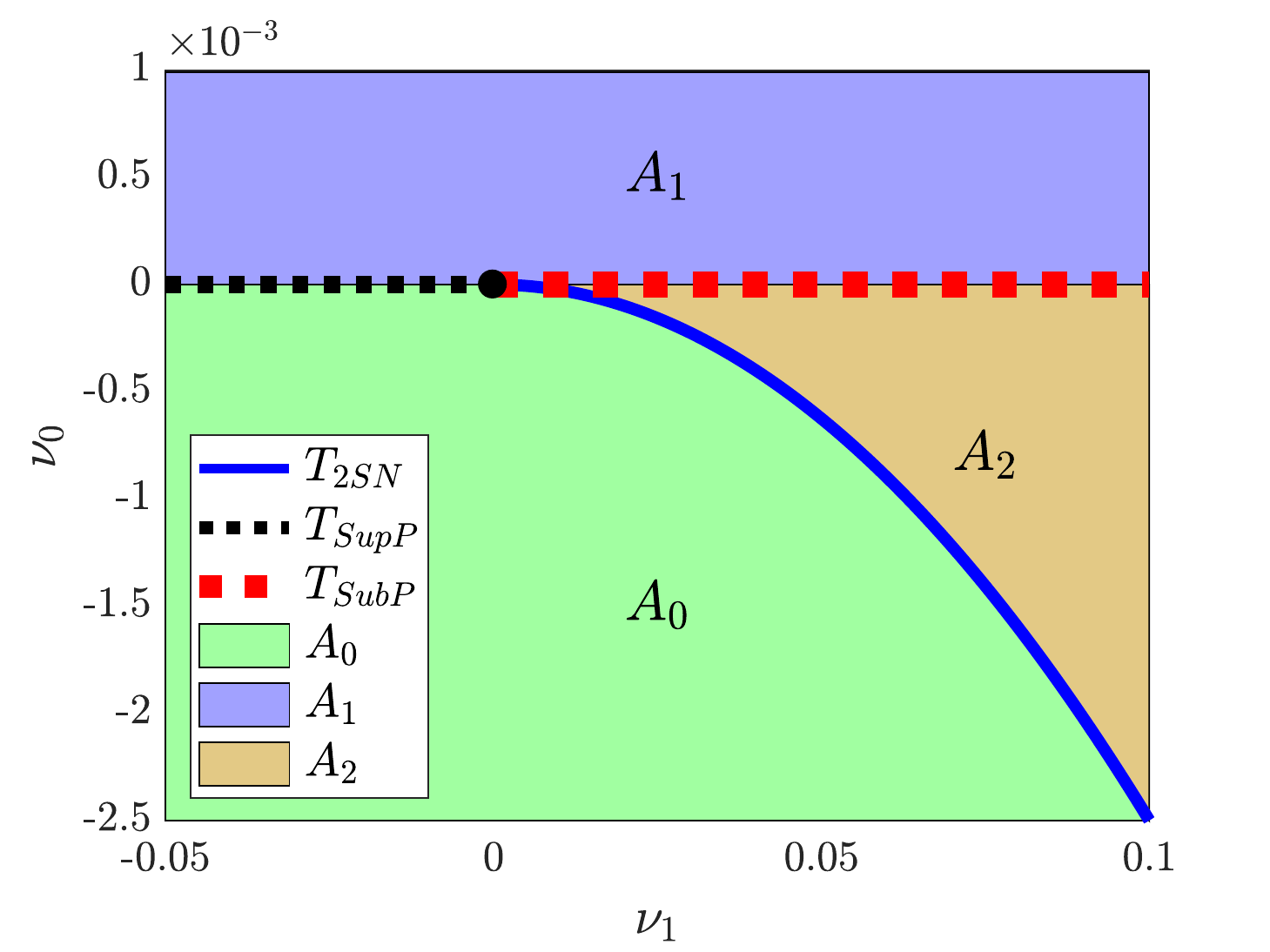}}\;
\subfloat[Curve \(\Gamma\) and bifurcation varieties for $a_2=1$. \label{S2a2+}]{\includegraphics[width=2.2in]{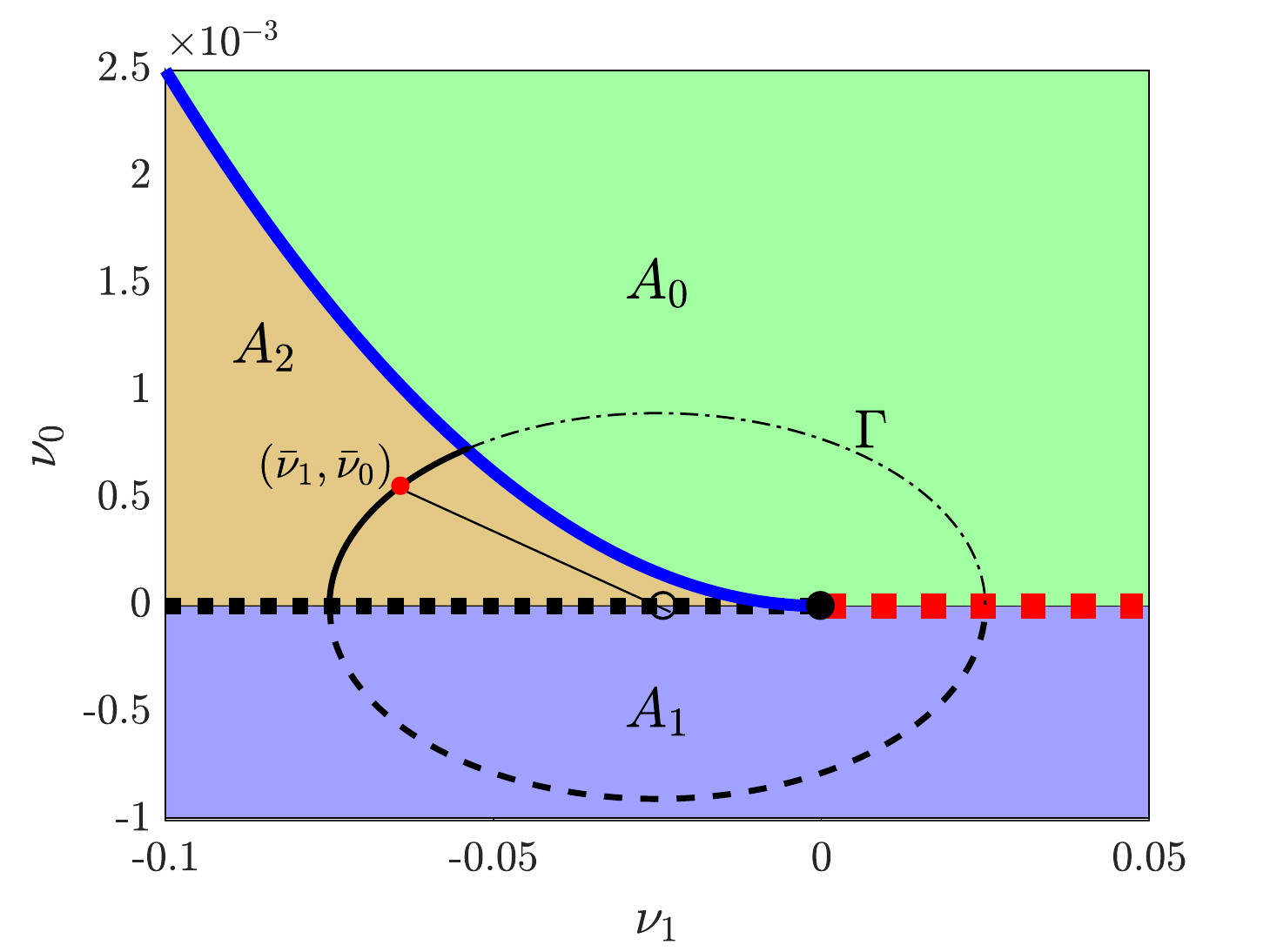}}\;
\subfloat[Root locus corresponding with \(\Gamma\) from Figure \ref{S2a2+}.\label{Sc}]{
\includegraphics[width=1.6in]{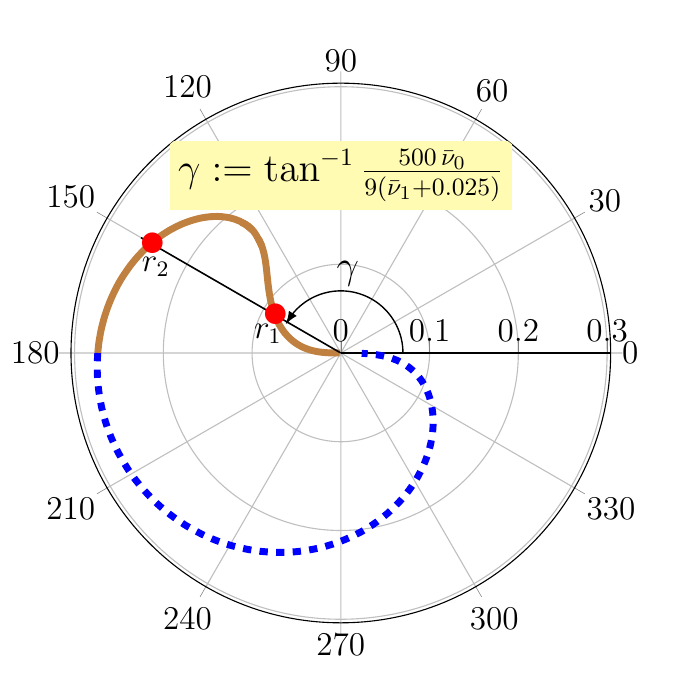}}
\caption{Bifurcation varieties and root loci for the vector field \eqref{ReducedParametricNormalForm}, leaf case \(s=2\). Here, $(\overline{\nu}_1, \overline{\nu}_0)=(\frac{-17}{250}, \frac{9}{2000})$ in Figure \ref{S2a2+} corresponds with roots \(r_1= 0.085\) and \(r_2= 0.245\) in Figure \ref{Sc}. \label{Figleafs2}}
\end{figure}

\begin{thm}[Leaf case \(s=2\)]\label{Thms2} Assume that the leaf parametric normal form \eqref{ReducedParametricNormalForm} is associated with \(s=2\) for some \(2\leq k\leq n\) and \(\C\in \mathbb{S}^{k-1, \sigma}_{>0}.\) A secondary stable flow-invariant \(k\)-hypertorus bifurcates from the origin on \(\MKC\)-leaf variety given by
\be\label{SupP} T_{SupP}:=\{(\nu_0, \nu_1)| \nu_0=0, \nu_1<0\}.\ee
A secondary unstable invariant \(k\)-hypertori bifurcates from the origin on the variety
\be\label{SubP} T_{SubP}:=\{(\nu_0, \nu_1)| \nu_0=0, \nu_1>0\}.\ee
There is a secondary double saddle-node type bifurcation of flow-invariant \(k\)-hypertori at the leaf-transition set
\be\label{2SND} T_{2SD}:=\Big\{(\nu_0, \nu_1)|\Big({\frac{\nu_1}{2a_2}}\Big)^2-\frac{\nu_0}{a_2}=0, a_2\nu_1<0\Big\};\ee
see Figures \ref{Figleafs2}. Here, two invariant hypertori (bifurcated at \(T_{SupP}\) and \(T_{SubP}\)) collide when their radiuses converge, and then, they both disappear as similar to a saddle-node type bifurcation. These bifurcations are five-determined. One of the invariant hypertori always live inside the other one until the radiuses of the inner hypertorus and the outer hypertorus converge. When \({\nu_1}^2>4a_2\nu_0\) and \(\nu_0>0\), the origin and the outer hypertorus are unstable while the inner invariant hypertorus is stable. For \({\nu_1}^2>4a_2\nu_0\) and \(\nu_0<0\), the origin and the outer invariant hypertorus are stable while the inner invariant \(\mathbb{T}_k\)-torus is repelling.
\end{thm}
\bpr Let \(a_2:=a_2(\0, C)\neq 0\) and \(\nu_1\neq 0.\) Consider the \(\mathbb{Z}_2\)-equivariant differential equation
\bes \dot{\rho}_{\sigma(k)}=f_{k, \sigma(k)}(\rho_{\sigma(k)}, \nu_0):=\rho_{\sigma(k)}(\nu_0+\nu_1 {\rho_{\sigma(k)}}^2+a_2{\rho_{\sigma(k)}}^4).\ees
Then,
\begin{eqnarray*}
&f_{k, \sigma(k)}(0, 0)=\frac{\partial}{\partial \rho_{\sigma(k)}}f_{k, \sigma(k)}(0, 0)=0, \quad \frac{\partial}{\partial \nu_0}f_{k, \sigma(k)}(0, 0)=\frac{\partial^2}{\partial {{\rho}_{\sigma(k)}}^2}f_{k, \sigma(k)}(0, 0)=0,&\\
&\frac{\partial^2}{\partial {{\rho}_{\sigma(k)}}\partial \nu_0}f_{k, \sigma(k)}(0, 0)=1\neq 0,
\quad \hbox{ and }\quad \frac{\partial^3}{\partial {{\rho}_{\sigma(k)}}^3}f_{k, \sigma(k)}(0, 0)=\nu_1\neq 0. &
\end{eqnarray*} This is a five-determined \(\mathbb{Z}_2\)-equivariant type bifurcation and corresponds with the \({\sigma(k)}\)-th amplitude dynamics for trajectories on the \(\MKC\)-leaf manifold. When \(\nu_1<0,\) we have a supercritical pitchfork bifurcation at the origin while \(\nu_1>0\) gives rise to a subcritical pitchfork bifurcation. For \(\nu_1a_2<0,\) take \(\mu:=\frac{\nu_0}{a_2}-\frac{{\nu_1}^2}{4{a_2}^2}\) such that \(f_{k, \sigma(k)}({\rho}_{\sigma(k)}, \nu_0)\) is read by
\(g(\rho_{\sigma(k)}, \mu):=\rho_{\sigma(k)}\big(a_2({\rho_{\sigma(k)}}^2+\frac{\nu_1}{2a_2})^2+\mu\big).\) This function takes its minima at \({\rho}^{\pm}_{\sigma(k)}=\pm\left(\frac{-\nu_1}{2a_2}\right)^{\frac{1}{2}}.\) Since
\begin{eqnarray*}
&g({\rho}^{\pm}_{\sigma(k)}, 0)=\frac{\partial}{\partial \rho_{\sigma(k)}}g({\rho}^{\pm}_{\sigma(k)}, 0)=0,\quad\frac{\partial}{\partial \mu}g({\rho}^{\pm}_{\sigma(k)}, 0)={\rho}^{\pm}_{\sigma(k)}\neq 0,\quad \frac{\partial^2}{\partial {{\rho}_{\sigma(k)}}^2}g({\rho}^{\pm}_{\sigma(k)}, 0)=8a_2{{\rho}^{\pm}_{\sigma(k)}}^3\neq 0,&
\end{eqnarray*} the points \((\rho, \mu)=({\rho}^{\pm}_{\sigma(k)}, 0)\) are bifurcation points and each of them represents a saddle-node type bifurcation.
\epr

 \subsection{Leaf case $s=3$}

\begin{figure}
\centering
\subfloat[$\nu_1=-0.1$.\label{S3v2-}]{\includegraphics[width=2in]{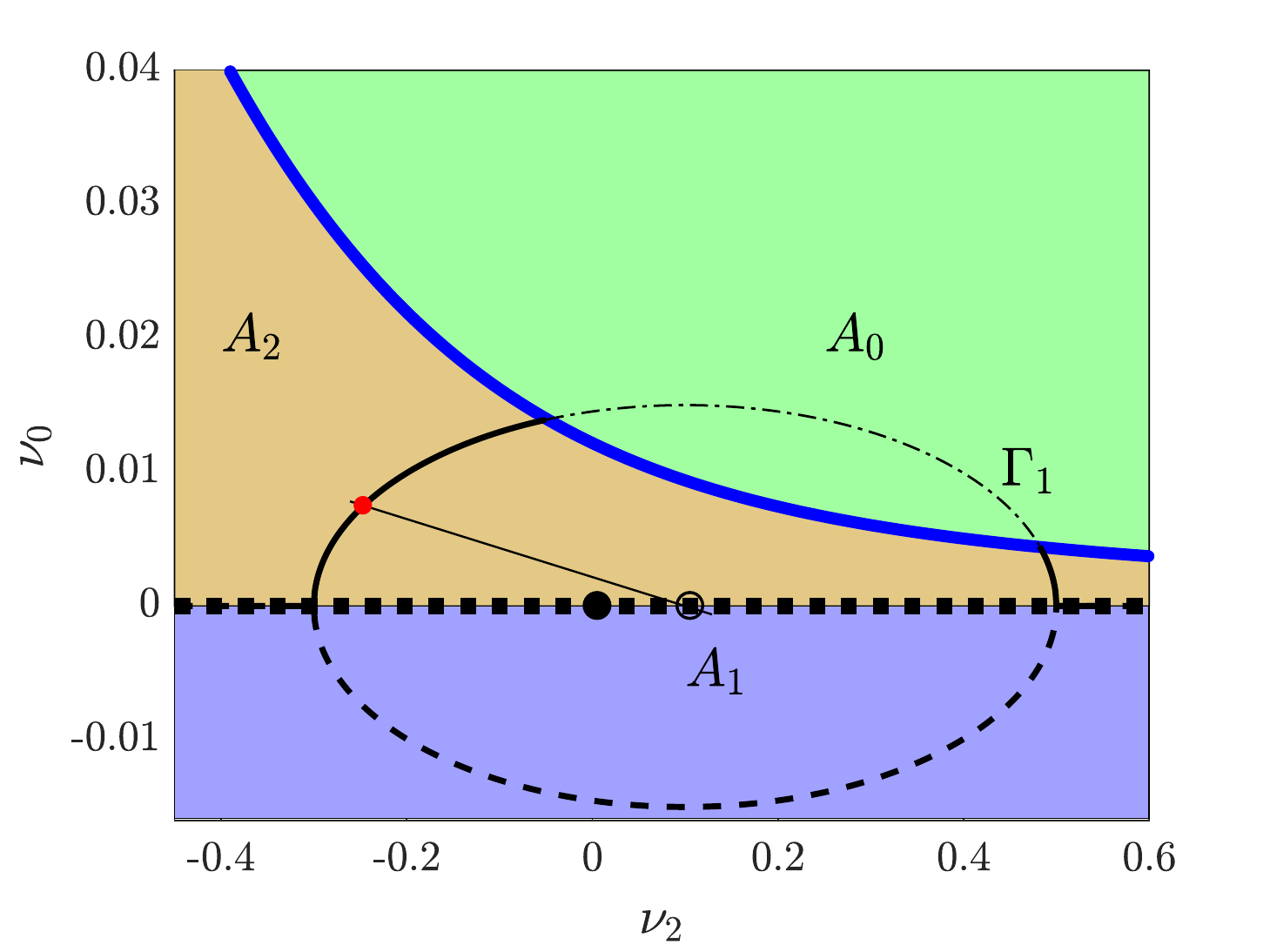}}\!
\subfloat[$\nu_1=0$. \label{S3v2zero}]{\includegraphics[width=2in]{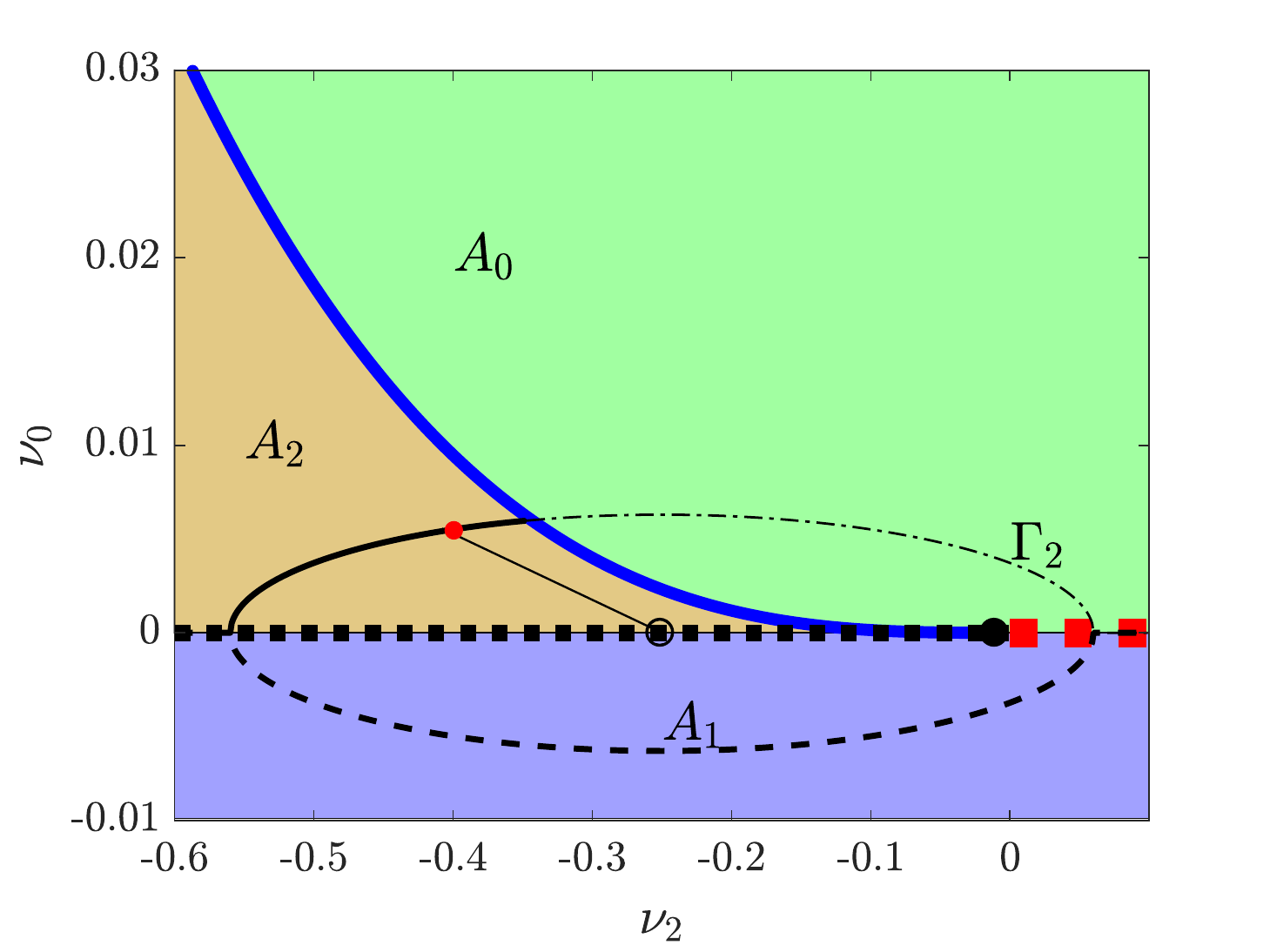}}\!
\subfloat[$\nu_1=0.1$.\label{S3v2+}]{\includegraphics[width=2.46in]{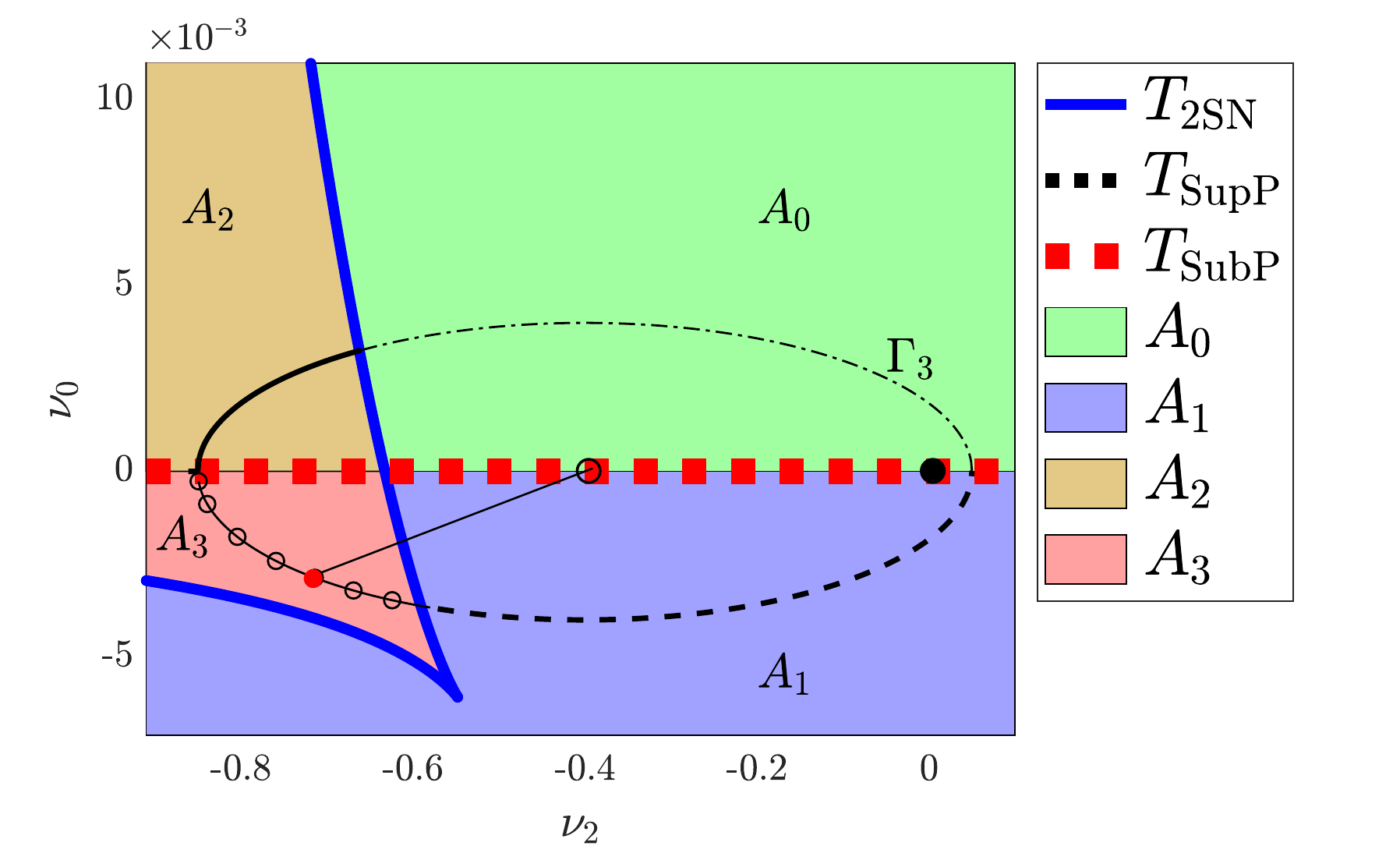}}\\
\caption{Transition sets for leaf case \(s=3\) \label{FigTransS3}}
\end{figure}

\begin{thm}[Leaf case \(s=3\)] Consider the \(\MKC\)-leaf parametric normal form \eqref{ReducedParametricNormalForm}, where \(s=3\) for \(2\leq k\leq n\) and \(\C\in \mathbb{S}^{k-1, \sigma}_{>0}.\) A \(k\)-hypertorus bifurcates from the origin at \(\MKC\)-leaf bifurcation variety
\begin{eqnarray}\label{T_Ps3}
T_{Psup}:=\{(\nu_0, \nu_1, \nu_2) | \nu_0=0\, \text{\rm when either\,} \nu_1<0 \text{\, \rm or\,} \nu_1=0, \nu_2<0\},\nonumber \\
 T_{Psub}:=\{(\nu_0, \nu_1, \nu_2) | \nu_0=0\, \text{\rm when either\,} \nu_1>0 \text{\, \rm or\,} \nu_1=0, \nu_2>0\}.
\end{eqnarray} This hypertorus is unstable when \(\nu_1>0,\) and stable for \(\nu_1<0\). There is a double saddle-node bifurcation variety of hypertori at
\be\label{T_SNs3}
T_{2SN}:=\left\{(\nu_0, \nu_1, \nu_2)\Big| D=0, \hbox{ and either } (\frac{\nu_2}{a_3}<0, \frac{\nu_1}{a_3}> 0) \text{ or } (\frac{\nu_0}{a_3}>0, \frac{\nu_1}{a_3}\leq 0) \right\},
\ee where \(D:= 4\big(\frac{\nu_1}{a_3}-\frac{{\nu_2}^2}{3{a_3}^2}\big)^3 +27\big(\frac{2{\nu_2}^3}{27{a_3}^3} -\frac{\nu_2\nu_1}{3{a_3}^2}+\frac{\nu_0}{a_3}\big)^2.\) These bifurcations are seven-determined.
\end{thm}

\bpr Consider the \(\mathbb{Z}_2\)-equivariant equation \(\dot{\rho}_{\sigma(k)}= f_{k, \sigma(k)}(\rho_{\sigma(k)}, \nu_0, \nu_1, \nu_2):=\nu_0\rho_{\sigma(k)}+\nu_1{\rho_{\sigma(k)}}^3+\nu_2{\rho_{\sigma(k)}}^5 + a_3 {\rho_{\sigma(k)}}^7,\) where \(a_3\neq0.\) It is easy to prove that this system is a seven-determined differential equation and so is the \(\MKC\)-leaf parametric normal form \eqref{ReducedParametricNormalForm}. We first prove the following claims:
\begin{itemize}
  \item[] Claim 1. The function \(f_{k, \sigma(k)}\) has 3 distinct positive roots if and only if \((\nu_0, \nu_1, \nu_2)\in A_3.\)
  \item[] Claim 2. The parameters \((\nu_0, \nu_1, \nu_2)\in A_2\cup B_2\) if and only if \(f_{k, \sigma(k)}\) has 2 distinct positive roots.
  \item[] Claim 3. The map \(f_{k, \sigma(k)}\) has one positive root if and only if \((\nu_0, \nu_1, \nu_2)\in A_1\cup B_1.\)
  \item[] Claim 4. The function \(f_{k, \sigma(k)}\) has no positive root if and only if \((\nu_0, \nu_1, \nu_2)\in A_0\cup B_0.\)
\end{itemize}
We take \(R:={\rho_{\sigma(k)}}^2\) and consider
\begin{eqnarray}\label{Eq3}&\frac{\nu_0}{a_3} +\frac{\nu_1 }{a_3}{R}+\frac{\nu_2 }{a_3} {R}^2 + {R}^3=0.&\end{eqnarray}
By substitution $ r:={R}+\frac{\nu_2}{3a_3},$ we obtain
\ba\label{Eq4}
&r^3+p r+q=0, \qquad\hbox{ where }\quad p:=-\frac{{\nu_2}^2}{3{a_3}^2}+ \frac{\nu_1}{a_3} \,\hbox{ and } \, q:=-\frac{\nu_2\nu_1}{3{a_3}^2}+\frac{2{\nu_2}^3}{27{a_3}^3}+\frac{\nu_0}{a_3}.&
\ea The number of real roots of the equations \eqref{Eq3} and \eqref{Eq4} are equal. Let $D=\left(\frac{p}{3}\right)^3+\left( \frac{q}{2}\right)^2.$ We discuss the number of real roots by \({\sc sign}(D)\) while the number of roots with positive real part are addressed by Routh-Hurwitz Theorem.
\begin{figure}[t]
\centering
\subfloat[Corresponding with \(\Gamma_1\) from Figure \ref{S3v2-} for \(\nu_1=-0.1\).\label{RLS3v1Min}]{\includegraphics[width=2in]{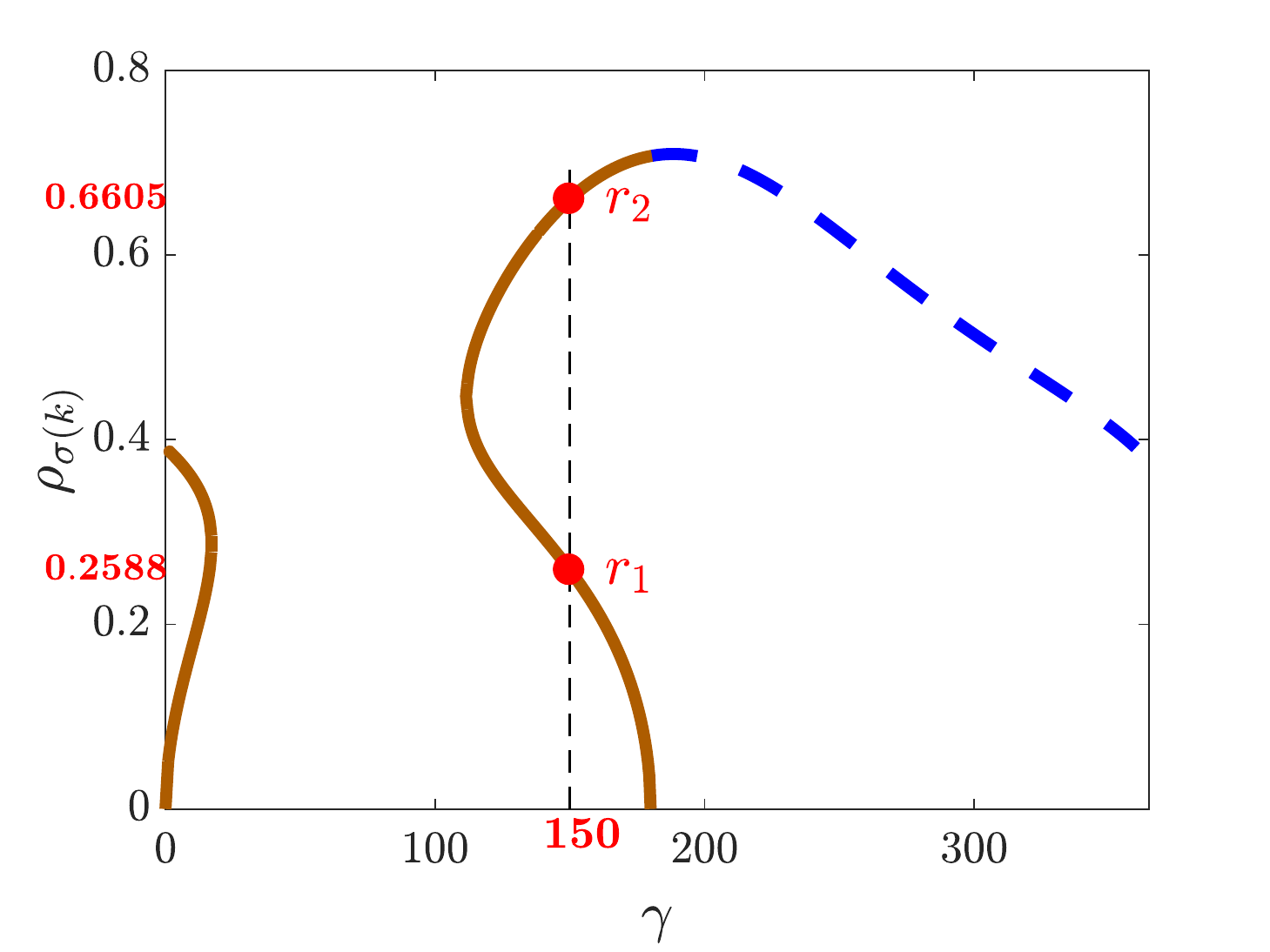}}\,
\subfloat[Associated with \(\Gamma_2\) on Figure \ref{S3v2zero} for \(\nu_1=0\).\label{RLS3v1Zero}]{\includegraphics[width=2in]{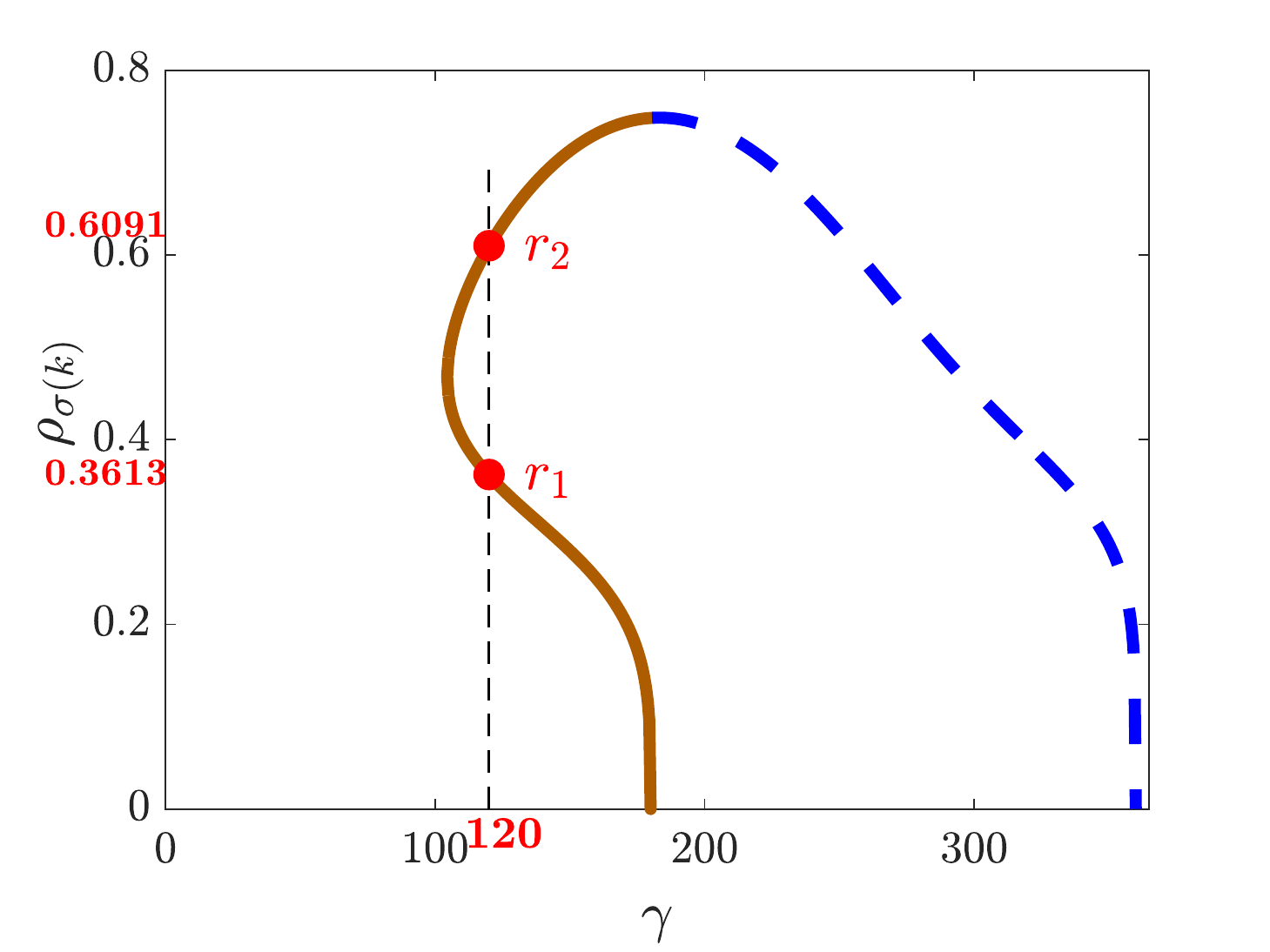}}\,
\subfloat[Radius loci for the ellipse \(\Gamma_3\) from Figure \ref{S3v2+} when \(\nu_1=0.1\).\label{RLS3v1Zero}]{\includegraphics[width=2.6in]{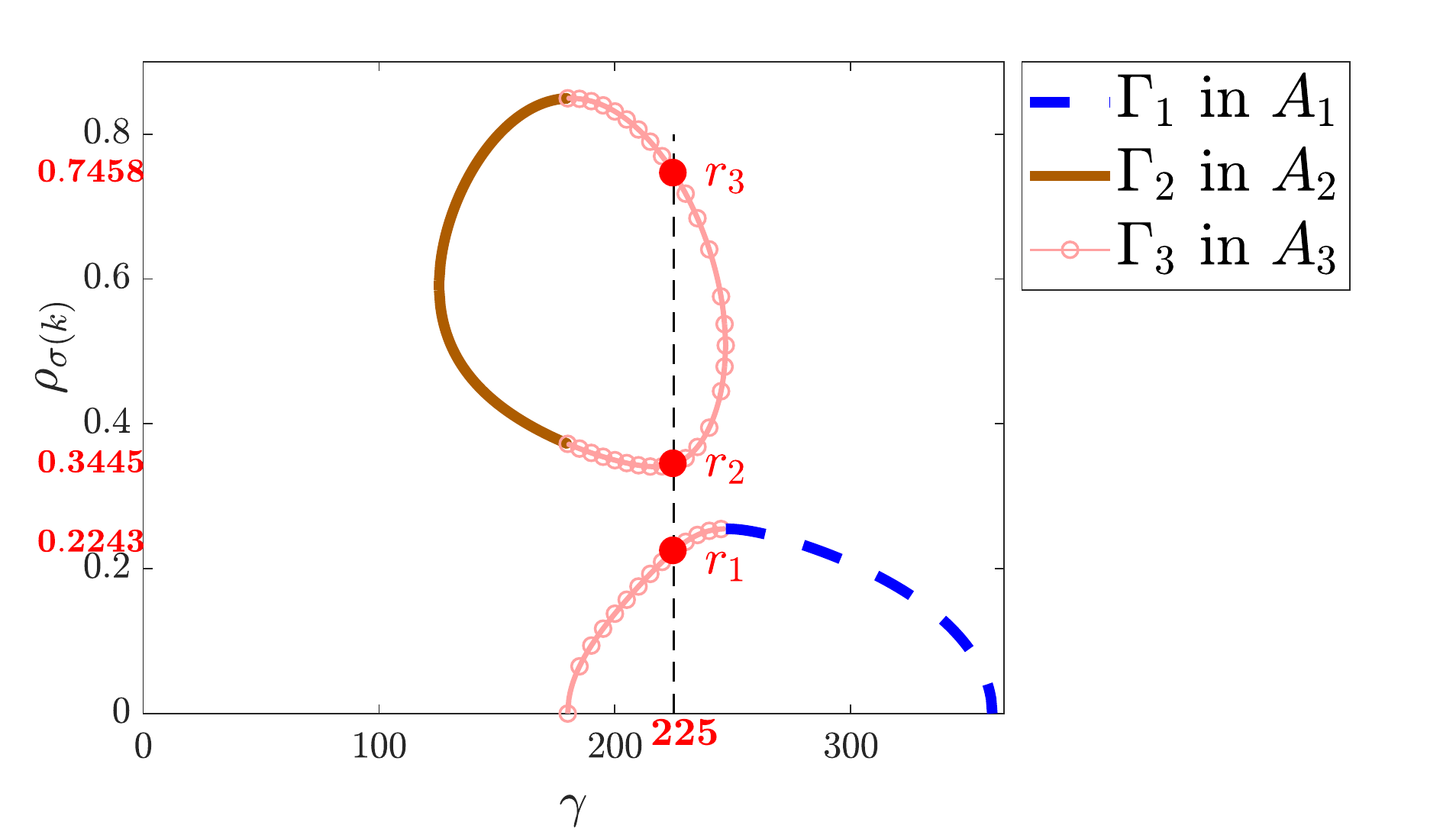}}
\caption{Radius loci of the invariant tori on $\MKC$ along the ellipses \(\Gamma_i, i=1, 2, 3\) when \(\gamma_1:=\frac{4\nu_0}{0.15(\nu_2-0.1)},\) \(\gamma_2:=\frac{31\nu_0}{.54(\nu_2+0.25)},\) and \(\gamma_3:=\frac{46\nu_0}{.47(\nu_2+0.4)},\) associated with
equation \eqref{ReducedParametricNormalForm} where \(s=3.\) }
\end{figure}
The number $0\leq n\leq 3$ of roots of the real polynomial \eqref{Eq3} which lie in the right half-plane is given by the formula
\(n=\var(1, \Delta_1, \frac{\Delta_2}{\Delta_1}, \frac{\Delta_3}{\Delta_2})\) where the function \(\var\) represents the number of sign variations in its  arguments while Hurwitz determinants follow
\begin{eqnarray}\label{HurwitzDeterminants}
&\Delta_1:=\frac{\nu_2}{a_3},\quad \Delta_2:=\frac{\nu_1\nu_2}{{a_3}^2}-\frac{\nu_0}{a_3},\quad \Delta_3:=\frac{\nu_0}{a_3}\frac{\nu_1\nu_2-a_3\nu_0}{{a_3}^2}.&
\end{eqnarray} We only discuss the singular cases of either \(\Delta_1=0\) or \(\Delta_2=0.\) The case \(\Delta_3=0\) leads to either \(\Delta_2=0\) or \(\nu_0=0\). The latter gives rise to \(\Delta_1:=\frac{\nu_2}{a_3},\) \(\Delta_2:=\frac{\nu_1\nu_2}{{a_3}^2},\) and the claim is straightforward.

For \(\Delta_1=0,\) we have \(\nu_2=0.\) Thus, we may use the modified Hurwitz determinants \(\Delta^*_2\) and \(\Delta^*_3\) as long as they retain the same signs as of \(\Delta_2\) and \(\Delta_3,\) respectively. Hence, we introduce \(\nu_2=\epsilon\) where \(\epsilon\) is a positive small number and the modified Hurwitz determinants by \(\Delta^{*}_1:=\epsilon, \Delta^{*}_2:=\frac{\epsilon\nu_1}{a_3}-\frac{\nu_0}{a_3}\) and \(\Delta^{*}_3:=\frac{\nu_0(\epsilon\nu_1-\nu_0)}{{a_3}^2}.\) Therefore, we have
\begin{eqnarray}\label{nu2=0}
&\sign(\Delta^{*}_1)>0, \quad \sign(\frac{\Delta_2^*}{\Delta_1^*})=\sign(\frac{-\nu_0}{\epsilon a_3}),\quad \hbox{ and } \quad \sign(\frac{\Delta_3^*}{\Delta_2^*})=\sign(\frac{\nu_0}{a_3}).&
\end{eqnarray} Hence, we can alternatively study the sign variation function \(\var(1, \Delta_1^*, \frac{\Delta_2^*}{\Delta_1^*}, \frac{\Delta_3^*}{\Delta_2^*}).\)

For the singular case \(\Delta_2=0,\) we have \(\nu_1\nu_2=a_3\nu_0.\) Thus, equation \eqref{Eq3} can be factored and is read as \(({\rho_{\sigma(k)}}^2+\frac{\nu_2}{a_3})({\rho_{\sigma(k)}}^4+\frac{\nu_1}{a_3})=0\). Hence, we discuss the roots of this reduced equation.

Claim 1. The equations \eqref{Eq3} has 3 distinct positive roots iff \(D<0\) and \(n=3.\) Since \(\var(1, \Delta_1, \frac{\Delta_2}{\Delta_1}, \frac{\Delta_3}{\Delta_2})=3,\) \(\Delta_1<0, \frac{\Delta_2}{\Delta_1}>0, \frac{\Delta_3}{\Delta_2}<0.\)
By Hurwitz determinants \eqref{HurwitzDeterminants}, we introduce
\begin{eqnarray*}
&S_3:=\{(\nu_0, \nu_1, \nu_2)|\frac{\nu_2}{a_3}<0, \frac{\nu_1\nu_2-a_3\nu_0}{a_3\nu_2}>0, \frac{\nu_0}{a_3}<0, D<0\}.&
\end{eqnarray*} Therefore, \(\frac{\nu_1}{a_3}-\frac{\nu_0}{\nu_2}>0\) and \(\frac{\nu_1}{a_3}>\frac{\nu_0}{\nu_2}>0.\) This implies that \(S_3\subseteq A_3.\) Let \((\nu_0, \nu_1, \nu_2)\in A_3\). Hence, \(\frac{\nu_2}{a_3}<0,\) \(\frac{\nu_0}{a_3}<0,\) \(\frac{\nu_1}{a_3}>0,\) and \(D<0.\) We claim that
\(\frac{\nu_1\nu_2-a_3\nu_0}{a_3\nu_2}=\frac{\Delta_2}{\Delta_1}>0\). Our argument is by contradiction. Suppose that \(\frac{\Delta_2}{\Delta_1}=\frac{\nu_1\nu_2-a_3\nu_0}{a_3\nu_2}<0.\) Since \(a_3\nu_2<0,\) \(a_3\nu_0<\nu_1\nu_2<0.\) Hence, we have
\begin{eqnarray}\label{Ineq1}
 &D=-\frac{{\nu_2}^2{\nu_1}^2}{ 108{a_3}^4 }-\frac{{\nu_1}{\nu_2}{\nu_0}}{ 6{a_3}^3 }+\frac{{\nu_1}^3}{ 27{a_3}^3 }+\frac{{\nu_2}^3{\nu_0}}{ 27{a_3}^4 }+\frac{{\nu_0}^2}{4{a_3}^2}
 \geq \frac{{\nu_1}^3}{ 27{a_3}^3 }+\frac{{\nu_2}^3{\nu_0}}{ 27{a_3}^4 }+\frac{2{\nu_0}^2}{ 27{a_3}^2 }>0.&
 \end{eqnarray} This, however, is a contradiction. Therefore, \(\frac{\Delta_2}{\Delta_1}>0\) and \(A_3= S_3\).

Claim 2. Let \(\nu_0\neq 0\). The equation \eqref{Eq4} has two distinct positive roots iff $D< 0$ and \(n=2.\) We define
\begin{eqnarray*}
&S_2:=\{(\nu_0, \nu_1, \nu_2)|\var(1, \Delta_1, \frac{\Delta_2}{\Delta_1}, \frac{\Delta_3}{\Delta_2})=2, D<0\}&
\end{eqnarray*} and show that \(S_2=A_2\cup B_2.\) The condition \(n=2\) for none-zero Hurwitz determinants gives rise to three different cases:
\begin{eqnarray*}
&\text{(i)}\; \Delta_1>0, \frac{\Delta_2}{\Delta_1}<0,\frac{\Delta_3}{\Delta_2}>0,\quad
\text{(ii)}\; \Delta_1<0, \frac{\Delta_2}{\Delta_1}>0,\frac{\Delta_3}{\Delta_2}>0,\quad
\text{(iii)}\; \Delta_1<0, \frac{\Delta_2}{\Delta_1}<0,\frac{\Delta_3}{\Delta_2}>0.&
\end{eqnarray*} The singular cases \(\Delta_1^*=0\) and \(\Delta_2^*=0\) adds the following two more cases:
\begin{eqnarray*}
&\text{(iv)}\;\Delta^{*}_1=\epsilon>0, \frac{\Delta_2^*}{\Delta_1^*}=\frac{-\nu_0}{\epsilon a_3}<0, \frac{\Delta_3^*}{\Delta_2^*}=\frac{\nu_0}{a_3}>0, \quad \hbox{ and }\quad
\text{(v)}\;\frac{\nu_1}{a_3}<0, \frac{\nu_2}{a_3}<0.&
\end{eqnarray*} Let \((\nu_0, \nu_1, \nu_2)\in S_2\) satisfy case (i). So, \(\frac{\nu_2}{a_3}>0,\) \(\frac{\nu_0}{a_3}>0,\) and \(\frac{\nu_1\nu_2-a_3\nu_0}{a_3\nu_2}<0.\)
 The latter inequality implies that \(\nu_1\nu_2<a_3\nu_0\). These conditions along with $D<0$ are satisfied only when
 \(\frac{\nu_1}{a_3}<0\). Otherwise, \(\frac{\nu_1}{a_3}\geq 0\) gives rise to \eqref{Ineq1}. This is a contradiction. Hence, \(\frac{\nu_1}{a_3}<0\) and \((\nu_0, \nu_1, \nu_2)\in B_2.\) If \((\nu_0, \nu_1, \nu_2)\in S_2\) satisfies
case (ii), we have \(\frac{\nu_2}{a_3}<0,\) \(\frac{\nu_0}{a_3}>0,\) and \(\frac{\nu_1\nu_2-a_3\nu_0}{a_3\nu_2}>0.\) Similar to the case (i),  the condition \(D<0\) infers that \(\frac{\nu_1}{a_3}>0\) and \((\nu_0, \nu_1, \nu_2)\in A_2.\) Now assume that \((\nu_0, \nu_1, \nu_2)\in S_2\) satisfies the case (iii).
Then, \(\frac{\nu_2}{a_3}<0,\) \(\frac{\nu_0}{a_3}>0,\) and \(\frac{\nu_1}{a_3}< \frac{\nu_0}{\nu_2}<0\). This implies \((\nu_0, \nu_1, \nu_2)\in B_2.\) Either of the singular cases (iv) and (v) along with \(D<0\) concludes that \((\nu_0, \nu_1, \nu_2)\in B_2.\) For \((\nu_0, \nu_1, \nu_2)\in A_2,\) we have \(\frac{\nu_1}{a_3}>0>\frac{\nu_0}{\nu_2}.\) So, \((\nu_0, \nu_1, \nu_2)\) satisfies case (iii) and as a result \((\nu_0, \nu_1, \nu_2)\in S_2.\) When \((\nu_0, \nu_1, \nu_2)\in B_2\) and \(\frac{\nu_2}{a_3}>0\), this results in
\(\frac{\Delta_2}{\Delta_1}= \frac{\nu_1\nu_2-a_3\nu_0}{a_3\nu_2}<0.\) Thereby, the parameters satisfy the case (i). When \((\nu_0, \nu_1, \nu_2)\in B_2,\) \(\frac{\nu_2}{a_3}<0\), and \(\frac{\Delta_2}{\Delta_1}= \frac{\nu_1\nu_2-a_3\nu_0}{a_3\nu_2}\) is either positive or negative, the parameters satisfy one of cases (ii) and (iii). Hence, \((\nu_0, \nu_1, \nu_2)\in S_2.\) For \((\nu_0, \nu_1, \nu_2)\in B_2, \frac{\nu_2}{a_3}<0\) and \(\frac{\Delta_2}{\Delta_1}= \frac{\nu_1\nu_2-a_3\nu_0}{a_3\nu_2}=0,\) singular case (v) is satisfied. Finally when \((\nu_0, \nu_1, \nu_2)\in B_2\) and \(\nu_2=0,\) the singular case (iv) results.


Claim 3. Let
\bas &S_1:=\{(\nu_0, \nu_1, \nu_2)|\, \hbox{Either } \var(1, \Delta_1,\frac{\Delta_2}{\Delta_1}, \frac{\Delta_3}{\Delta_1})=1 \hbox{ or } (\var(1, \Delta_1, \frac{\Delta_2}{\Delta_1}, \frac{\Delta_3}{\Delta_1})= 3 \hbox{ when } D>0) \}.&\eas
Thereby, the equation \(\eqref{Eq3}\) has one distinct positive root iff \((\nu_0, \nu_1, \nu_2)\in S_1\). Hence, it suffices to prove \(S_1=A_1\cup B_1.\) Let \((\nu_0, \nu_1, \nu_2)\in S_1.\) Since \(\var(1, \Delta_1,\frac{\Delta_2}{\Delta_1}, \frac{\Delta_3}{\Delta_1})=1\) or \(3,\) we have 4 possible different cases:
  \(\text{(a)}\, \Delta_1<0, \frac{\Delta_2}{\Delta_1}>0,\frac{\Delta_3}{\Delta_2}<0\),
\begin{eqnarray}\label{PossibleCases}
&\text{(b)}\;  \Delta_1>0, \frac{\Delta_2}{\Delta_1}>0,\frac{\Delta_3}{\Delta_2}<0,\quad
 \text{(c)}\; \Delta_1>0, \frac{\Delta_2}{\Delta_1}<0,\frac{\Delta_3}{\Delta_2}<0,\quad
 \text{(d)}\; \Delta_1<0, \frac{\Delta_2}{\Delta_1}<0,\frac{\Delta_3}{\Delta_2}<0.&
\end{eqnarray}
For the singular case \({\Delta_2}=\frac{\nu_1\nu_2}{{a_3}^2}-\frac{\nu_0}{a_3}=0,\) the equation \(({\rho_{\sigma(k)}}^2+ \frac{\nu_2}{a_3})({\rho_{\sigma(k)}}^4+ \frac{\nu_1}{a_3})=0\) has only one positive root iff \(\nu_1\nu_2<0.\) Since \(D=\frac{{\nu_1}}{27a_3}(\frac{\nu_1}{{a_3}}+\frac{2{\nu_2}^2}{{a_3}^2})^2,\) \(D>0\) implies that \(\frac{\nu_1}{a_3}>0,\) \(\frac{\nu_2}{a_3}<0,\) and \(\frac{\nu_0}{a_3}<0.\) This results in \((\nu_0, \nu_1, \nu_2)\in A_1.\) Similarly, for \(D<0\) we have \((\nu_0, \nu_1, \nu_2)\in B_1.\) Given equations \eqref{nu2=0}, for the singular case
\bas
&\Delta_1=\frac{\nu_2}{a_3}=0 \; \hbox{ and }\; var(1, \Delta^*_1,\frac{\Delta^*_2}{\Delta^*_1}, \frac{\Delta^*_3}{\Delta^*_1})=1,&
\eas
the only possible case is \(\frac{\nu_0}{a_3}<0.\) When \(\nu_1>0\), \((\nu_0, \nu_1, \nu_2)\in A_1\) while \(\nu_1\leq 0\) implies that \((\nu_0, \nu_1, \nu_2)\in B_1.\) When \(D>0\) and the conditions of the case (a) are met, \(\frac{\nu_2}{a_3}<0,\) \(\frac{\nu_0}{a_3}<0,\) and \(\frac{\nu_1}{a_3}>\frac{\nu_0}{\nu_2}>0.\) So, \((\nu_0, \nu_1, \nu_2)\in A_1.\) For the case (b), we have
\bas
&\frac{\nu_2}{a_3}>0,\frac{\nu_0}{a_3}<0\; \hbox{ and }\; \frac{\nu_1}{a_3}>\frac{\nu_0}{\nu_2}.&
\eas Hence, for \(\frac{\nu_1}{a_3}>0\) we have \((\nu_0, \nu_1, \nu_2)\in A_1\) and when \(\frac{\nu_1}{a_3}\leq 0,\) \((\nu_0, \nu_1, \nu_2)\in B_1.\) The case (c) leads to \(\frac{\nu_2}{a_3}>0,\) \(\frac{\nu_0}{a_3}<0\) and \(\frac{\nu_1}{a_3}<\frac{\nu_0}{\nu_2}<0.\) This concludes that \((\nu_0, \nu_1, \nu_2)\in B_1.\) Finally, the conditions (d) implies that \(\frac{\nu_2}{a_3}<0,\frac{\nu_0}{a_3}<0\) and \(\frac{\nu_1}{a_3}<\frac{\nu_0}{\nu_2}.\) When \(\frac{\nu_1}{a_3}\leq 0\), \((\nu_0, \nu_1, \nu_2)\in B_1\). By inequalities in \eqref{Ineq1}, we have \(D>0\) for \(\frac{\nu_1}{a_3}> 0\). Therefore, \((\nu_0, \nu_1, \nu_2)\in A_1.\) Hence, \(S_1\subseteq A_1\cup B_1.\)

Now assume that \((\nu_0, \nu_1, \nu_2)\in A_1\cup B_1.\) If \((\nu_0, \nu_1, \nu_2)\in A_1,\) we have the following conditions:
\begin{eqnarray*}
&(\frac{\nu_0}{a_3}<0, \frac{\nu_1}{a_3}>0, \frac{\nu_2}{a_3}<0, D>0), \quad \quad (\frac{\nu_0}{a_3}<0, \frac{\nu_1}{a_3}>0, \frac{\nu_2}{a_3}> 0),
\text{\quad or \quad} (\frac{\nu_0}{a_3}<0, \frac{\nu_1}{a_3}>0, \nu_2=0).&
\end{eqnarray*} From the first group of inequalities, for \(\frac{\nu_1}{a_3}<\frac{\nu_0}{\nu_2}\), we have \(\frac{\Delta_2}{\Delta_1}=\frac{\nu_1\nu_2 -a_3\nu_0}{a_3\nu_2}<0\) and the case (d) is satisfied. For \(\frac{\nu_1}{a_3}>\frac{\nu_0}{\nu_2}\), \(\frac{\Delta_2}{\Delta_1}>0\) and the inequalities in (a) hold. When \(\frac{\nu_1}{a_3}=\frac{\nu_0}{\nu_2}\), the condition \(\Delta_2=0\) is met. Further, recall that \(\nu_1\nu_2<0\). Second group gives rise to \(\frac{\Delta_2}{\Delta_1}=\frac{\nu_1\nu_2-a_3\nu_0}{a_3\nu_2}>0\) and the case (b) while the third group of conditions is a subset of conditions in the singular case \(\Delta_1=0\). By the latter, we have \(\var(1, \Delta^*_1,\frac{\Delta^*_2}{\Delta^*_1}, \frac{\Delta^*_3}{\Delta^*_1})=1\) and \(\frac{\nu_1}{a_3}>0\).
 Hence, \((\nu_0, \nu_1, \nu_2)\in S_1.\)
Let \((\nu_0, \nu_1, \nu_2)\in B_1\) and \(\frac{\nu_2}{a_3}>0\). We decompose the conditions in \(B_1\) into the following cases:
\begin{eqnarray*}
&(\frac{\nu_1}{a_3}< 0, \nu_1\nu_2<a_3\nu_0, \frac{\nu_0}{a_3}<0),\quad (\frac{\nu_1}{a_3}< 0, \nu_1\nu_2=a_3\nu_0, \frac{\nu_0}{a_3}<0),\quad (\frac{\nu_1}{a_3}\leq 0, \nu_1\nu_2>a_3\nu_0, \frac{\nu_0}{a_3}<0).&
\end{eqnarray*} Each group of the above inequalities for nonsingular cases gives rise to \(\var(1, \Delta_1,\frac{\Delta_2}{\Delta_1}, \frac{\Delta_3}{\Delta_1})=1.\) For the singular case \(\nu_1\nu_2=a_3\nu_0,\) we have \(\nu_1\nu_2<0.\) Thus, the equation \(({\rho_{\sigma(k)}}^2+\frac{\nu_2}{a_3})({\rho_{\sigma(k)}}^4+\frac{\nu_1}{a_3})=0\) have one root
 for both cases \(D\leq 0\) and \(D>0\). Hence, \((\nu_0, \nu_1, \nu_2)\in S_1.\) Since \(\frac{\nu_1}{a_3}\leq 0\) and \(\frac{\nu_0}{a_3}< 0,\) the condition \(\frac{\nu_2}{a_3}<0\) concludes that \(\frac{\Delta_2}{\Delta_1} =\frac{\nu_1\nu_2-a_3\nu_0}{a_3\nu_2}<0\). Thereby, \(\var(1, \Delta_1,\frac{\Delta_2}{\Delta_1}, \frac{\Delta_3}{\Delta_1})=1.\)
  So, for both cases \(D\leq 0\) and \(D>0\), \((\nu_0, \nu_1, \nu_2)\in S_1.\)
Since \(\frac{\nu_0}{a_3}< 0\) for \(\nu_2=0\), the condition \(\var(1, \Delta^*_1,\frac{\Delta^*_2}{\Delta^*_1}, \frac{\Delta^*_3}{\Delta^*_1})=1\) holds. Similar to the above, \((\nu_0, \nu_1, \nu_2)\in S_1\).  Then, \(A_1\cup B_1=S_1.\)

Claim 4. The equation \eqref{Eq3} has no roots iff \((\nu_0, \nu_1, \nu_2)\in A_0\cup B_0.\) The boundaries of the sets \(A_3, A_2\cup B_2, A_1\cup B_1\) and \(A_0\cup B_0\) is expressed by the varieties \(T_{P}\) introduced in \eqref{T_Ps3} and \(T_{2SN}\) given by \eqref{T_SNs3}. Further, \(\mathbb{R}^3\setminus(\cup_{i=0}^3A_i\sqcup\cup_{i=0}^2B_i)=T_P\cup T_{2SN}.\) There is always pitchfork bifurcation type on the variety \(T_{P}\) and double saddle node bifurcation on the variety \(T_{2SD}\). The saddle-node bifurcation corresponds with equation \(\frac{d}{dt}\rho_{\sigma(k)}=f_{k, \sigma(k)}(\rho_{\sigma(k)}, \nu_0,\nu_1, \nu_2)\) and the equilibria are given by
\begin{eqnarray*}
&(\gamma^{\pm}_{1}, D)=\Big(\big(-\frac{\nu_2}{a_3}\pm\sqrt{\frac{-p}{3}}\big)^\frac{1}{2}, 0\Big) \hbox{ and }
(\gamma^{\pm}_{2}, D)=\Big(-\big(-\frac{\nu_2}{a_3}\pm\sqrt{\frac{-p}{3}}\big)^\frac{1}{2}, 0\Big).&
\end{eqnarray*} Let
\begin{eqnarray*}
&g(\rho_{\sigma(k)}, D):=a_3\rho_{\sigma(k)}\Big(\big({\rho_{\sigma(k)}}^2+\frac{\nu_2}{3a_3}\pm2\sqrt{\frac{-p}{3}}\big)\!\big({\rho_{\sigma(k)}}^2 +\frac{\nu_2}{3a_3}\mp\sqrt{\frac{-p}{3}}\big)^2
+\frac{4D}{q\pm2(\frac{-p}{3})^\frac{3}{2}}\Big).&
\end{eqnarray*} Here, for \(i=1, 2,\) we have
\begin{eqnarray*}
&g(\gamma^{\pm}_{i}, 0)=\frac{\partial}{\partial \rho_{\sigma(k)}}g(\gamma^{\pm}_{i}, 0)=0,\quad\frac{\partial}{\partial D}g(\gamma^{\pm}_{i}, 0)=\frac{4a_3\gamma^{\pm}_{i}}{q\pm2\left(\frac{-p}{3}\right)^\frac{3}{2}}\neq 0,&\\&
\hbox{ and }\quad \frac{\partial^2}{\partial {{\rho}_{\sigma(k)}}^2}g(\gamma^{\pm}_{i}, 0)=\pm24\sqrt{\frac{-p}{3}}\gamma^{\pm}_i\neq 0.&
\end{eqnarray*}
These correspond with a saddle-node bifurcation.
Let \(\nu_1\neq 0.\) Then,
\begin{eqnarray*}
&f_{k, \sigma(k)}(0, 0)=\frac{\partial}{\partial \rho_{\sigma(k)}}f_{k, \sigma(k)}(0, 0)=0, \quad \frac{\partial}{\partial \nu_0}f_{k, \sigma(k)}(0, 0)=\frac{\partial^2}{\partial {{\rho}_{\sigma(k)}}^2}f_{k, \sigma(k)}(0, 0)=0,&\\&
\frac{\partial^2}{\partial {{\rho}_{\sigma(k)}}\partial \nu_0}f_{k, \sigma(k)}(0, 0)=1\neq 0, \quad\hbox{ and }\quad \frac{\partial^3}{\partial {{\rho}_{\sigma(k)}}^3}f_{k, \sigma(k)}(0, 0)=\nu_1\neq 0&
\end{eqnarray*} The invariant hypertorus is stable for \(\nu_1>0\) while \(\nu_1<0\) corresponds with an unstable flow-invariant hypertorus.
When \(\nu_0\neq 0, \nu_1=0\) and \(D(\nu_0, 0, \nu_2)\neq 0,\) \(D\) does not change its sign as \(\nu_1\) slightly varies. Therefore, there is no qualitative type change in the vicinity of \(\nu_1=0\). The case \((\nu_1, D(\nu_0, 0, \nu_2))=(0, 0)\) implies that \(\frac{1}{27}\frac{{\nu_2}^3\nu_0}{{a_4}^3} +\frac{1}{4}\frac{{\nu_0}^2}{{a_3}^2}=0\) and thereby, \((\nu_0, 0, \nu_2)\in T_{2SN}.\)
\epr

\subsection{Examples on leaf-bifurcation control }

Consider the equation
\begin{eqnarray}\label{EulBifExm}
&\frac{d}{dt}\x=\sum_{i=1}^{3}\omega_i\Theta^i_{\0}+f(\mu, \x)E_{\0}&
\end{eqnarray} where \(\omega_i=\sqrt{i}\) for \(1\leq i\leq 3,\)\,\(\x:=(x_1, y_1, x_2, y_2, x_3, y_3)\in \mathbb{R}^6\),
\begin{eqnarray}\label{ScalarFunction}
f(\mu, \x):= \alpha_{0}+ \alpha_1 x_1+\alpha_2y_1 + \alpha_{3}{x_1}^{2}+\alpha_{4}{y_1}^{2}+\alpha_{5}{x_2}^{2}+\alpha_{6}{y_2}^{2}+\alpha_{7}x_{1}{x_3}^2+\alpha_{8}x_{1}{y_3}^2
\end{eqnarray}
and $\alpha_{i}=a_{i}+\mu_{i}$ for $0 \leq i \leq 8$ and \(a_{0}=0.\) Thus, \(n=3.\)
\begin{figure}[t]
\centering
\subfloat[Leaf case \(s=1\) in example \ref{DifLeafTrans1}. Leaf-bifurcation varieties for \(C_1\) and \(C_2.\) \label{BifVerS1S2} ]{\includegraphics[width=.34\linewidth,height=1.4in]{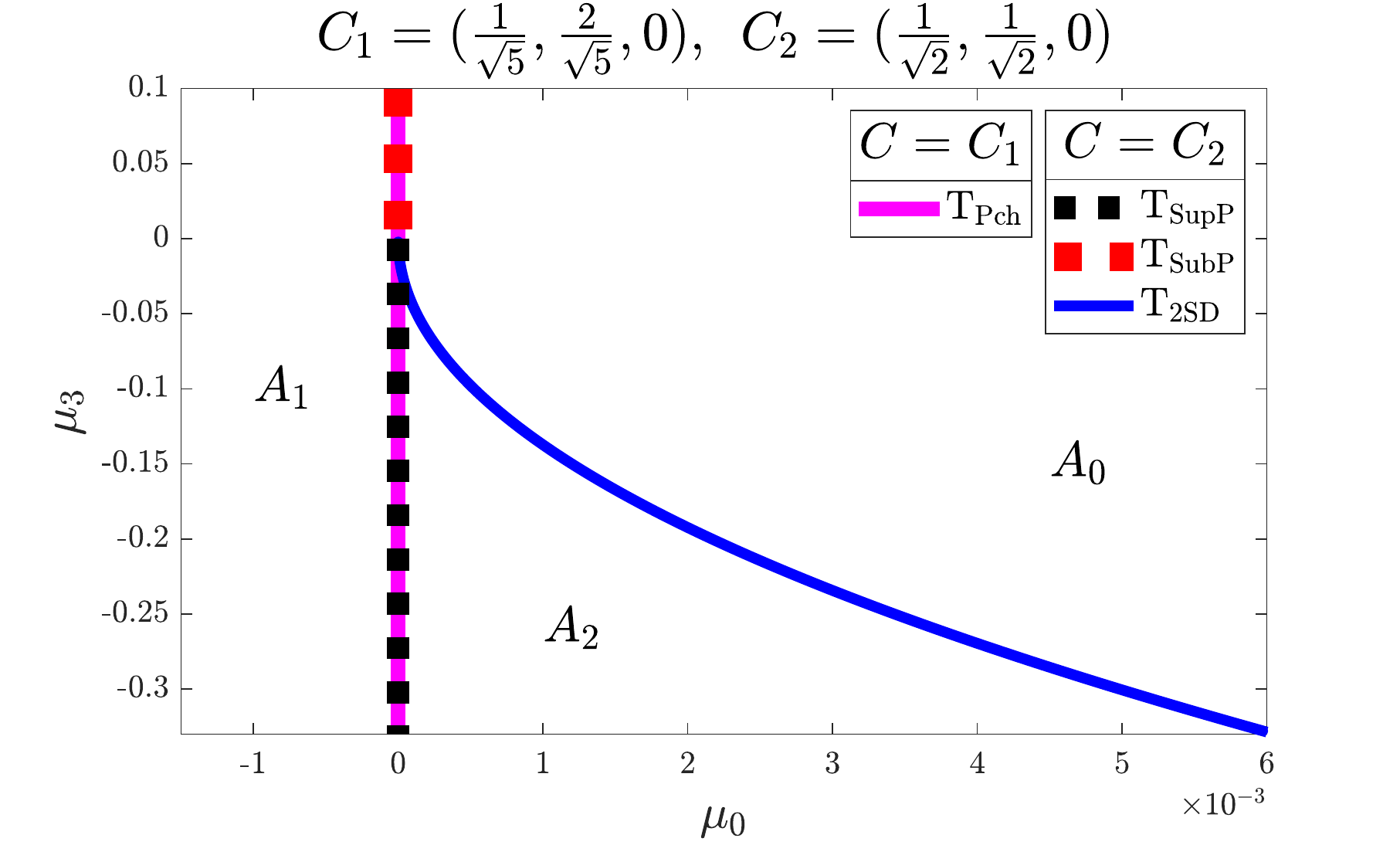}}\;
\subfloat[The transition sets for the leaf case $s=2$ in example \ref{DifLeafTrans2}. \label{S2GTS}]{\includegraphics[width=.26\linewidth,height=1.4in]{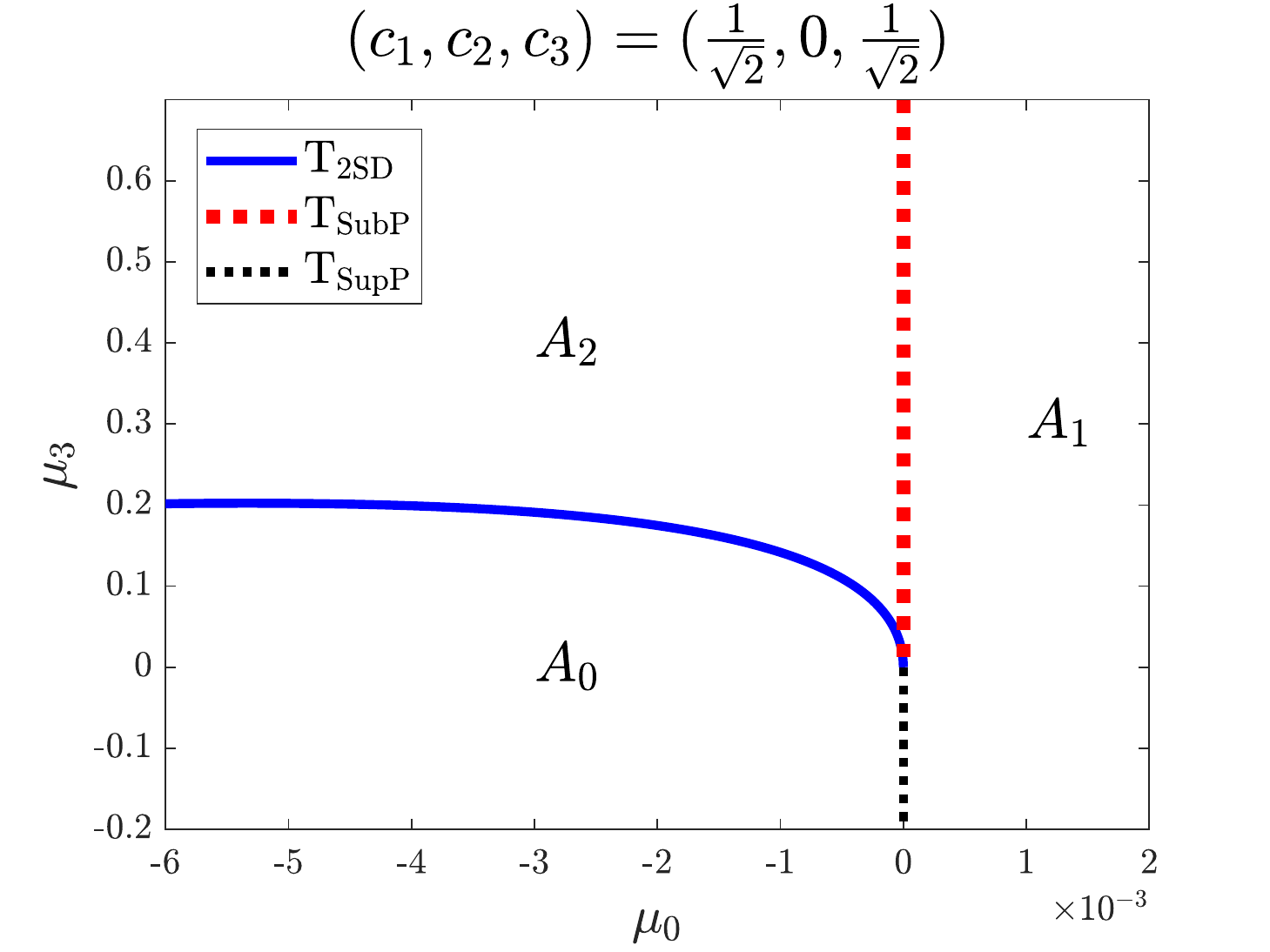}}\;
\subfloat[The leaf-transition sets for leaf case $s=3$ in example \ref{DifLeafTrans2}. \label{S3TS} ]{\includegraphics[width=.26\linewidth,height=1.4in]{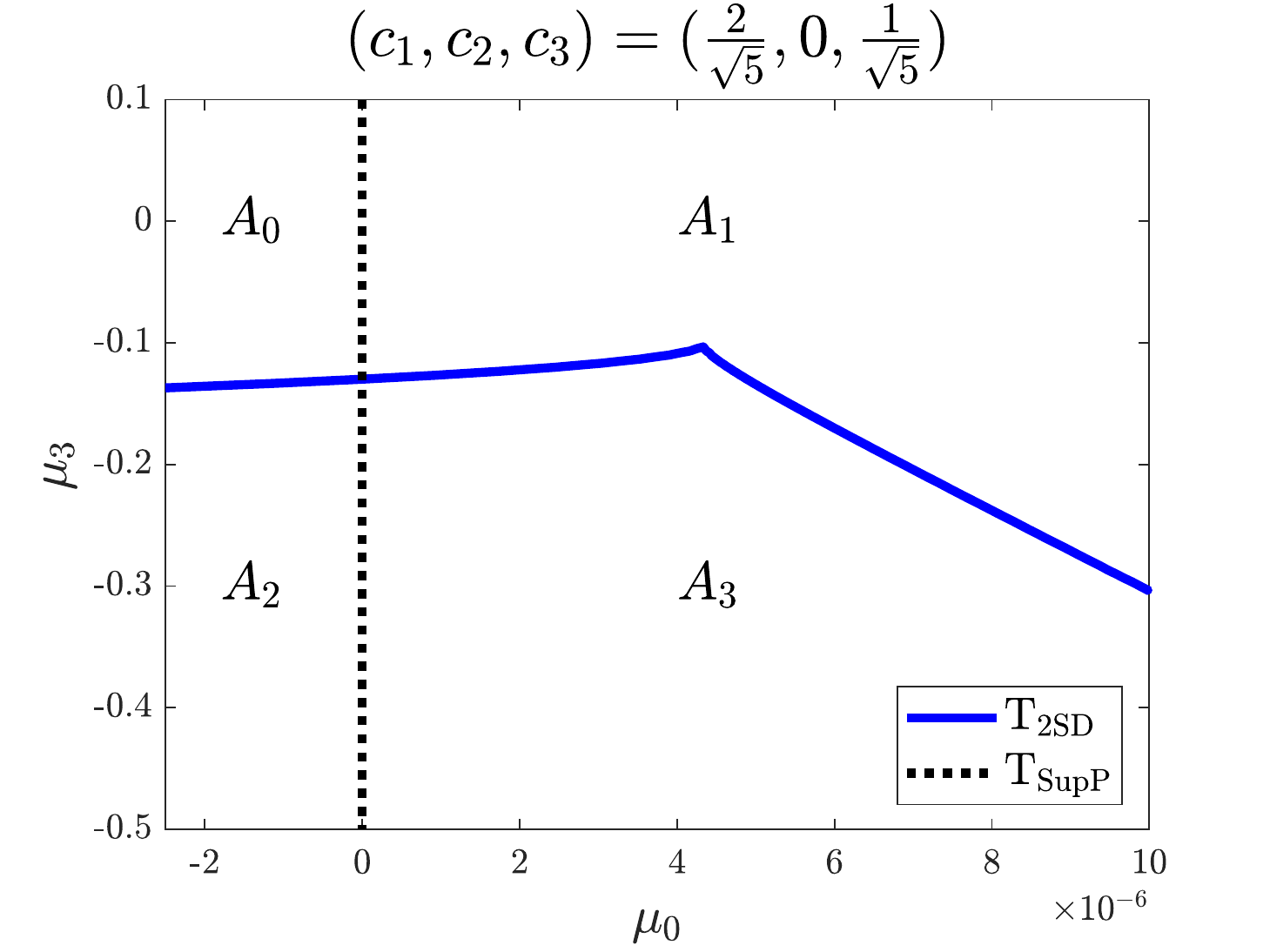}}\,
\caption{Leaf-bifurcation varieties as $(c_1, c_2, c_3)$ varies in examples \ref{DifLeafTrans1} and \ref{DifLeafTrans2}. }\label{TSCaseS2}
\end{figure}
\begin{exm}\label{DifLeafTrans1}
 We take
\bes  k=2 \quad\hbox{ and } \quad\sigma\in S^2_3 \quad\hbox{ for } \quad\sigma:=I.\ees
Hence, we have \(\mathcal{M}_{2, \sigma}=\{\x\in\mathbb{R}^6| (x_1, y_1)\neq 0, (x_2, y_2)\neq 0, (x_3, y_3)=0\}.\) For notation brevity, \(C=(c_1, c_2, 0)\in \mathbb{S}^2_{\geq 0}\) is denoted by \((c_1, c_2).\) Then, the leaf \(\mathcal{M}_{2, \sigma}^C\) follows
\begin{equation*}
\mathcal{M}_{2, \sigma}^{(c_1, c_2)}=\{\x\in\mathcal{M}_{2, \sigma}\,|\,c_2\|(x_1, y_1) \|=c_1\|(x_2, y_2)\| \}.
\end{equation*} The \(\mathcal{M}_{2, \sigma}^{(c_1, c_2)}\)-leaf reduction of the differential system \eqref{EulBifExm} in polar coordinates is associated with
\begin{eqnarray*}
&\sum_{i=1}^{2}\!\left(\frac{\omega_i\partial}{\partial \theta_i}\!+\!\big(\alpha_{0}\!+\!\rho_1(\alpha_1\cos\theta_1\!+\!\alpha_2\sin\theta_1)\!+\!{\rho_1}^2\!( \alpha_{3}\cos^2\theta_1\!+\!\alpha_{4}\sin^2\theta_1\!+\!\frac{\alpha_{5}\cos^2\theta_2\!+\!
\alpha_{6}\sin^2\theta_2}{{c_1}^2{c_2}^{-2}})\big)\!\frac{c_i\rho_{1}\partial}{{c_1}\partial\rho_{i}}\right).&
\end{eqnarray*} The associated vector field in the Lie algebra \(\mathscr{J}\) via the homeomorphism \(\Psi\) is given by
\begin{eqnarray*}
&\sum_{i=1}^{2}\left(\frac{\sqrt{i}w_i\partial}{\mathbf{i}\partial w_i}-\frac{\sqrt{i}z_i\partial}{\mathbf{i}\partial z_i}\right)+\left(\mu_0+\left(\frac{\alpha_1-\mathbf{i}\alpha_2}{2}\right)rz_1+\left(\frac{\alpha_1+\mathbf{i}\alpha_2}{2}\right)rw_1
+\left(\frac{\alpha_3-\alpha_4}{4}\right)r^2{z_1}^2
+\left(\frac{\alpha_3-\alpha_4}{4}\right)r^2{w_1}^2  \right)\frac{r\partial}{\partial r}&\\
&+\left(\left(\frac{{c_1}^2(\alpha_3+\alpha_4)+{c_2}^2(\alpha_5+\alpha_6)}{2{c_1}^2}\right)r^2+\frac{{c_2}^2}{{c_1}^2}
\left(\frac{\alpha_5-\alpha_6}{4}\right)r^2{z_2}^2+
\frac{{c_2}^2}{{c_1}^2}\left(\frac{\alpha_5-\alpha_6}{4}\right)r^2{w_2}^2\right)\frac{r\partial}{\partial r}.&
\end{eqnarray*}
We omit the parameters \(\mu_i\) for \(i=1, 2, 6, 7, 8\) by setting them to zero. Using a {\sc Maple} implementation of Theorem \ref{1stLeafNF} and its proof, a truncated first level parametric leaf-normal form in \(\mathscr{L}_{\mathbb{T}_k\times \mathbb{R}^+}\) (via the map \(\Psi^{-1}\)) up to grade-seven is
\begin{eqnarray*}
&\sum_{i=1}^{2}\frac{\sqrt{i}\partial}{\partial \vartheta_i}+\sum_{i=1}^{2}\left(\mu_0+b_1{\varrho_1}^2+b_2{\varrho_1}^4\right)\frac{c_i\varrho_{1}\partial}{{c_1}\partial\varrho_{i}}&
\end{eqnarray*} where \(b_1(\mu, C):=\frac{{c_1}^2(a_3+a_4)+{c_2}^2(a_5+a_6)}{2{c_1}^2}+\frac{\mu_3+\mu_4}{2}\) and \(b_2(\mu, C):=\frac{{a_1}^2(a_3+3a_4)+{a_2}^2(3a_3+a_4)}{8\ {\omega_{1}}^2}+\frac{({a_{1}}^2+{a_{2}}^2)(a_5+a_6)}{4\ {\omega_{1}}^2}.\)
When \(b_1(0, C)\neq 0,\) we have the generic leaf case \(s=1.\) By Theorem \ref{InfLPNF}, the third level (infinite-level) parametric leaf-normal forms is given by
\begin{eqnarray*}
&\sum_{i=1}^{2}\frac{\sqrt{i}\partial}{\partial \theta_i}+\sum_{i=1}^{2}\left(\mu_0+\frac{{c_1}^2(a_3\!+\!a_4)\!+\!{c_2}^2(a_5\!+\!a_6)}{2{c_1}^2}{\rho_1}^2\right) \frac{c_i\rho_{1}\partial}{{c_1}\partial\rho_{i}}. &
\end{eqnarray*} Let
\be\label{S1cons} a_1:=1, a_2=0, a_3:=-4, a_4:=6, a_5:=-1,\hbox{ and } a_6=-1.\ee Hence, \(b_1(0, C)=\frac{{c_1}^2-{c_2}^2}{{c_1}^2}.\)
By Theorem \ref{Thms1}, \(T_{Pch}:=\{(\mu_0, \mu_3)| \mu_0=0\}\) and for \(\frac{\mu_0{c_1}^2}{{c_1}^2-{c_2}^2}<0\) and sufficiently small values of \((\mu_0, \mu_3)\), an invariant \(\mathbb{T}_2\)-torus bifurcates from origin; see Figures \ref{BifVerS1S2} and \ref{x1y1}. For numerical bifurcation control of the system \eqref{EulBifExm}, we take the leaf corresponding with \((c_1, c_2)=(\frac{1}{\sqrt{5}}, \frac{2}{\sqrt{5}})\) and \(\mu_0=0.025.\) Thus, the initial condition \((x_1, y_1, x_2, y_2, x_3, y_3)=(0.01, 0,  0.02, 0, 0, 0)\) from inside the invariant torus and \((x_1, y_1, x_2, y_2, x_3, y_3)=(0.2, 0,  0.4, 0, 0, 0)\) from outside the stable invariant torus give rise to the numerical phase portraits in \((x_1, y_1)\)-plane and \((x_2, y_2)\)-plane depicted by Figures \ref{x1y1}, respectively.

\begin{figure}[t]
\centering 
\subfloat[Leaf case \(s=1.\) Two numerical phase portrait trajectories ($x_1(t)$ versus \(y_1(t)\)) and (\(x_2(t)\) versus \(y_2(t)\)) converging to an invariant torus when \((c_1, c_2)=(\frac{1}{\sqrt{5}}, \frac{2}{\sqrt{5}}).\) \label{x1y1} ]{
\includegraphics[width=.25\linewidth,height=1.4in]{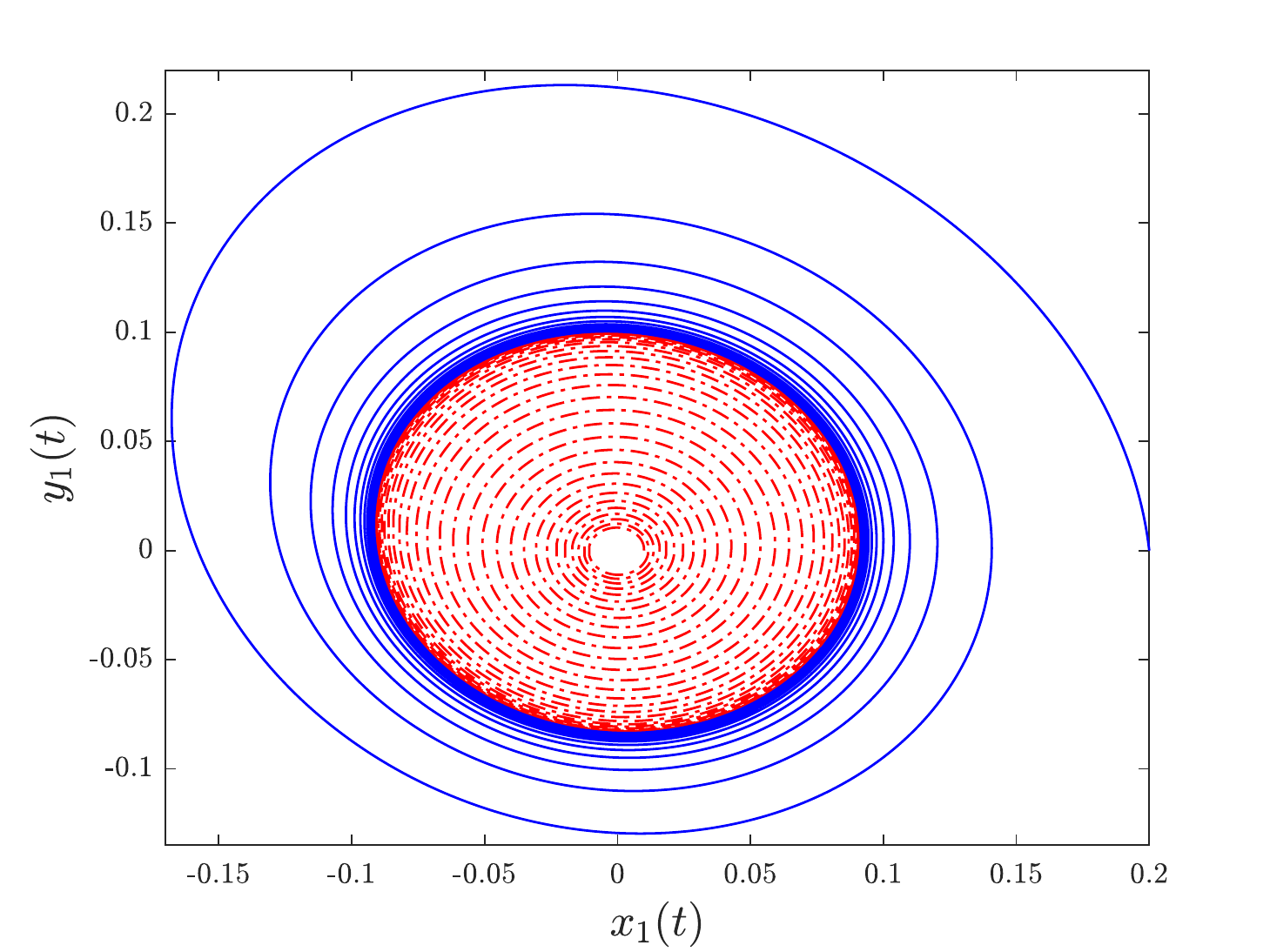}\hskip-1ex
\includegraphics[width=.25\linewidth,height=1.4in]{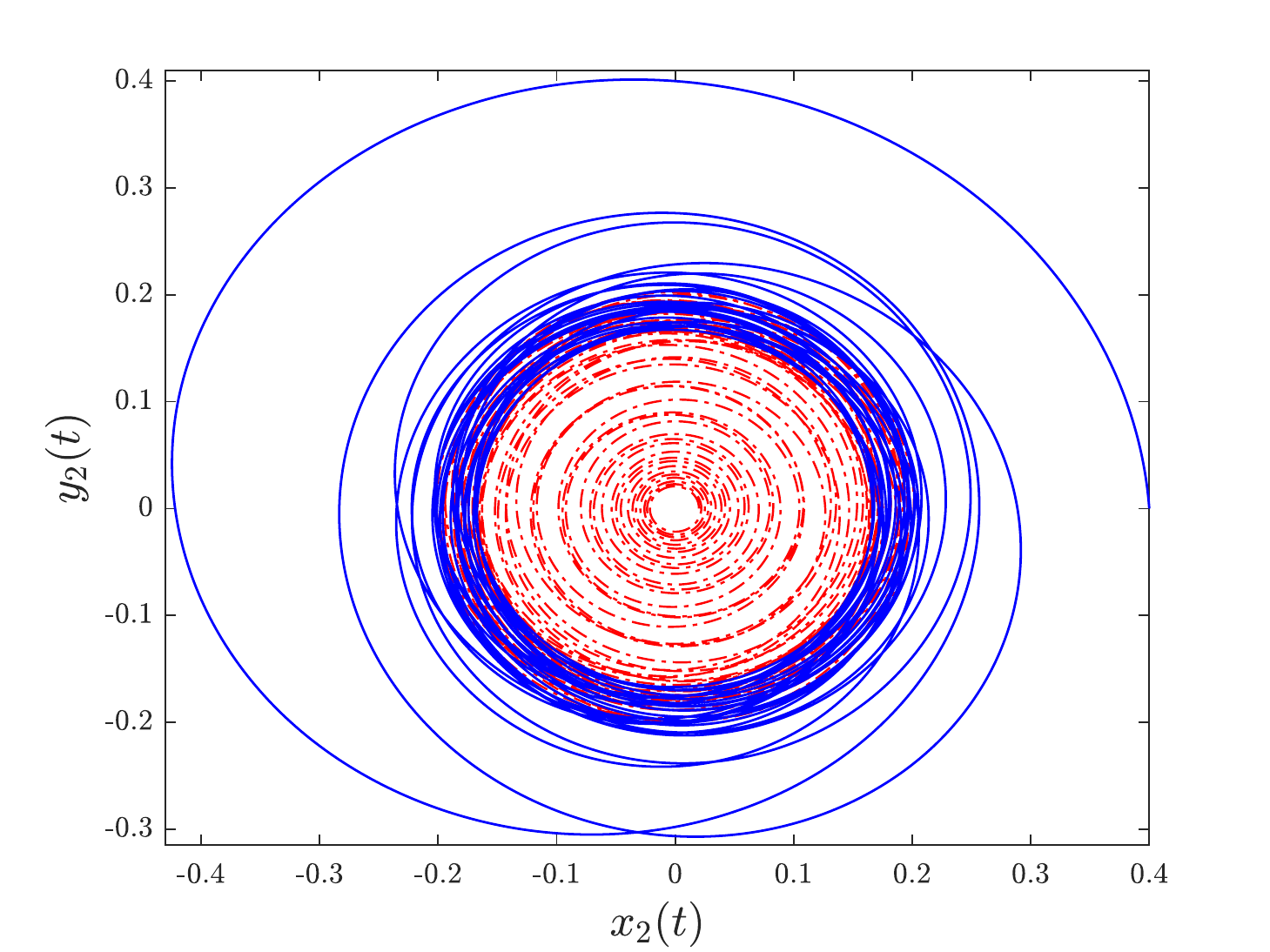}}\;
\subfloat[Leaf case \(s=2.\) Three numerical trajectories depicted in \((x_1(t), y_1(t))\)-plane phase portrait. There are two different invariant limit 4-tori for \((c_1, c_2)=(\frac{1}{\sqrt{2}}, \frac{1}{\sqrt{2}}).\) \label{x1y1deg}]{
\includegraphics[width=.48\linewidth,height=1.4in]{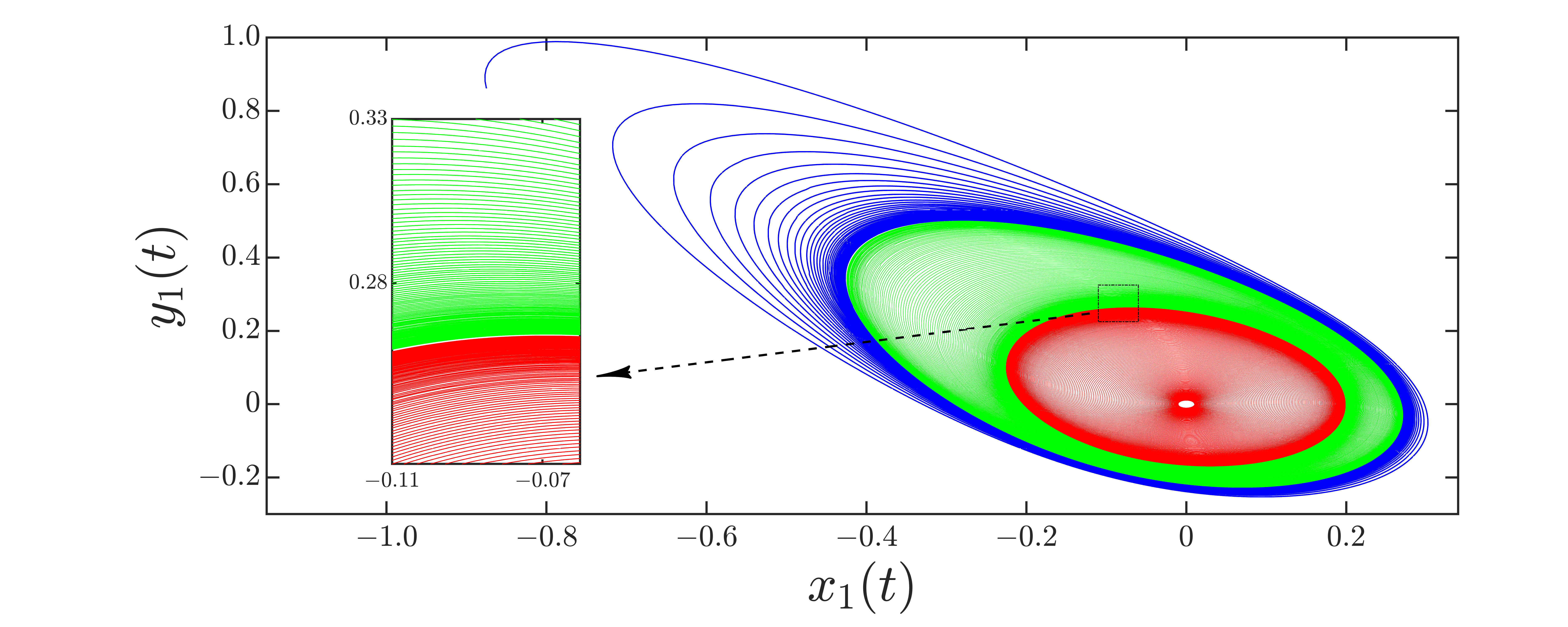}}
\caption{The controlled numerical phase portraits in \((x_1, y_1)\)- and \((x_2, y_2)\)-planes for the system \eqref{EulBifExm}, constants given by \eqref{S1cons} in Example \ref{DifLeafTrans1} and $\mu_0=0.025$. Figures \ref{x1y1} depict a stable invariant \(\mathbb{T}_4\)-torus on the leaf-\(\mathcal{M}_{2, \sigma}^{(\frac{1}{\sqrt{5}}, \frac{2}{\sqrt{5}})}.\) Figure \ref{x1y1deg} illustrate two tori living on the leaf-\(\mathcal{M}_{2, \sigma}^{(0.5\sqrt{2}, 0.5\sqrt{2})}.\) }\label{GenericCaseS1}
\end{figure}

Now take \(c_1=c_2=\frac{1}{\sqrt{2}}.\) Then, \(b_1(0, C)=0\) and \(b_2(0, C)=\frac{7{c_1}^2-2{c_2}^2}{4{c_1}^2}=\frac{5}{4}\neq 0\). Hence, we have the leaf case \(s=2\). Then, the amplitude equation of fourteenth-grade truncation of parametric leaf-normal form is
\begin{eqnarray*}
&\sum_{i=1}^{2}\frac{\sqrt{i}\partial}{\partial \theta_i}+\sum_{i=1}^{2}\left(\mu_0+\frac{163}{15}{\mu_0}^2+\frac{3}{10} \mu_3\mu_0+\left(\frac{1}{2}\mu_3-\mu_0\right){\rho_1}^2+\frac{5}{4}{\rho_1}^4\right)\frac{c_i\rho_{1}\partial}{{c_1}\partial\rho_{i}}. &
\end{eqnarray*} Then, the estimated transition varieties are given by $T_{SupP}=\{(\mu_0, \mu_3)|\mu_0=0, \mu_3<0\}$,
\begin{eqnarray*}
&T_{SubP}=\{(\mu_0, \mu_3)|\mu_0=0, \mu_3>0\}\text{\, and \,} T_{2SD}=\{(\mu_0, \mu_3)|\!\left(\frac{\mu_3}{2}\!-\!\mu_0\right)^2\!-\!5\mu_0\!-\!\frac{163{\mu_0}^2}{3}\!-\!\frac{3\mu_3\mu_0}{2}, \mu_3<0 \}.&
\end{eqnarray*} Note that \(T_{2SD}\) is not a transition set for the leaf-system associated with \((c_1, c_2)=(\frac{1}{\sqrt{5}}, \frac{2}{\sqrt{5}})\).
The transition variety \(T_{Pch}\) changes into two intrinsically different transition varieties \(T_{SupP}\) and \(T_{SubP}\) for the case \((c_1, c_2)=(\frac{1}{\sqrt{2}}, \frac{1}{\sqrt{2}})\). Figure \ref{x1y1deg} depicts the bifurcation of two invariant \(\mathbb{T}_2\)-tori from origin living on the leaf \(\mathcal{M}_{2, \sigma}^{(0.5\sqrt{2}, 0.5\sqrt{2})}.\) There are three trajectories in Figure \ref{x1y1deg} when \(\mu_0=0.005, \mu_5=-0.35\): 1) the blue trajectory starts from the initial condition $(-0.87, 0.86, -0.87, 0.86, 0, 0)$ outside the external unstable torus. 2) the green trajectory is associated with the initial condition $(-0.3, 0.3, -0.3, 0.3, 0,0)$ and converges to the stable internal invariant torus. Figure \ref{x1y1deg} depicts green trajectory in both forward and backward time. 3) the red trajectory starts at $(0.01, 0, 0.01, 0, 0, 0)$ from inside the internal stable torus. In order to illustrate the invariant tori, the trajectories associated with blue and red are plotted with inverse time (backward-time trajectory).
\end{exm}

\begin{figure}[t]
\centering
\subfloat[Forward-time series for \(x_1(t)\) and $y_1(t)$; also see Figure \ref{S2x2y2A}. \label{S2x1y1A}  ]{\includegraphics[width=2.2in]{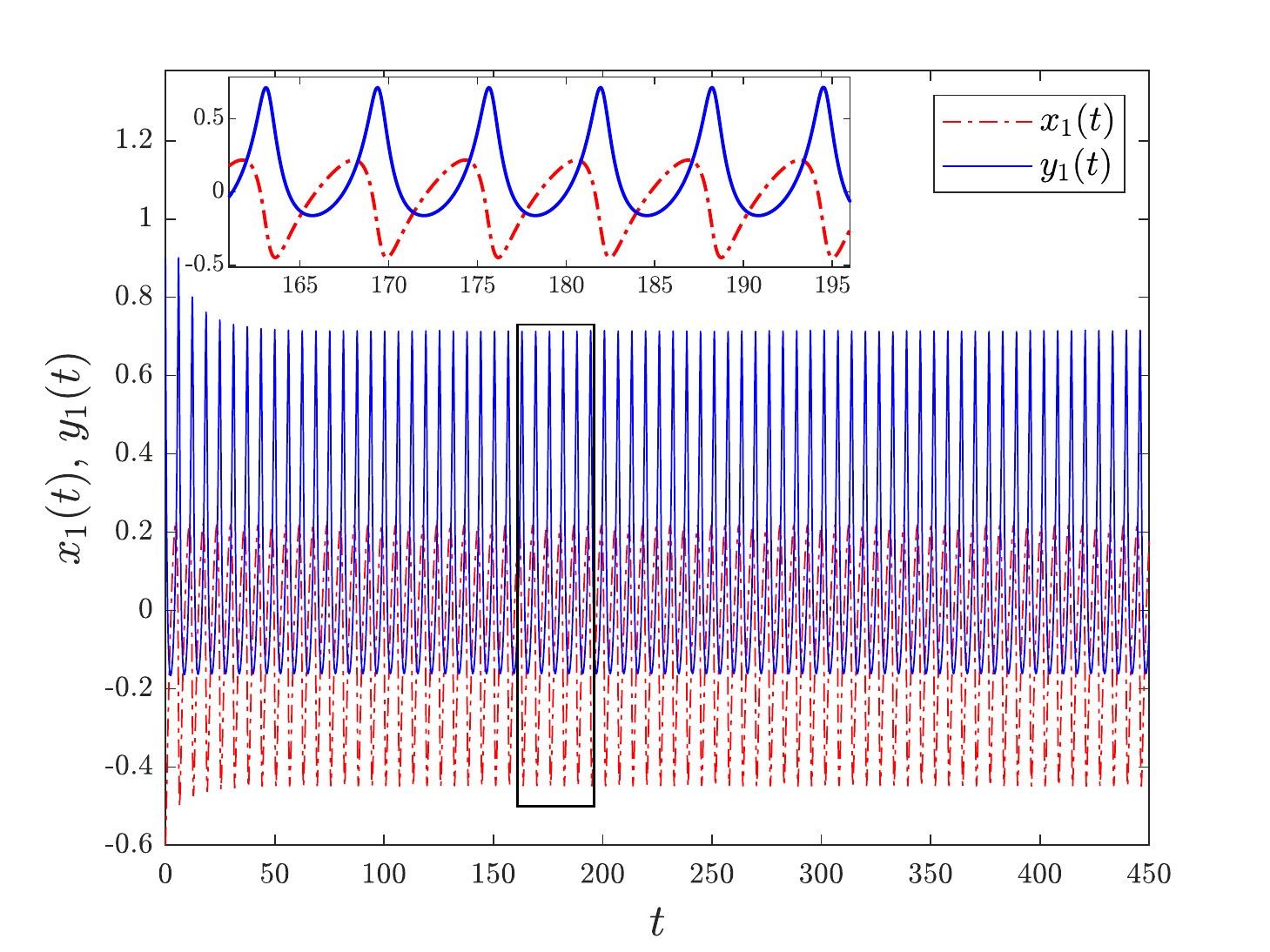}}\,
\subfloat[Two trajectories for \(x_1(t)\) and $y_1(t).$ They converge to two invariant tori in backward and forward time. \label{S2x2y2A} ]{\includegraphics[width=2.2in]{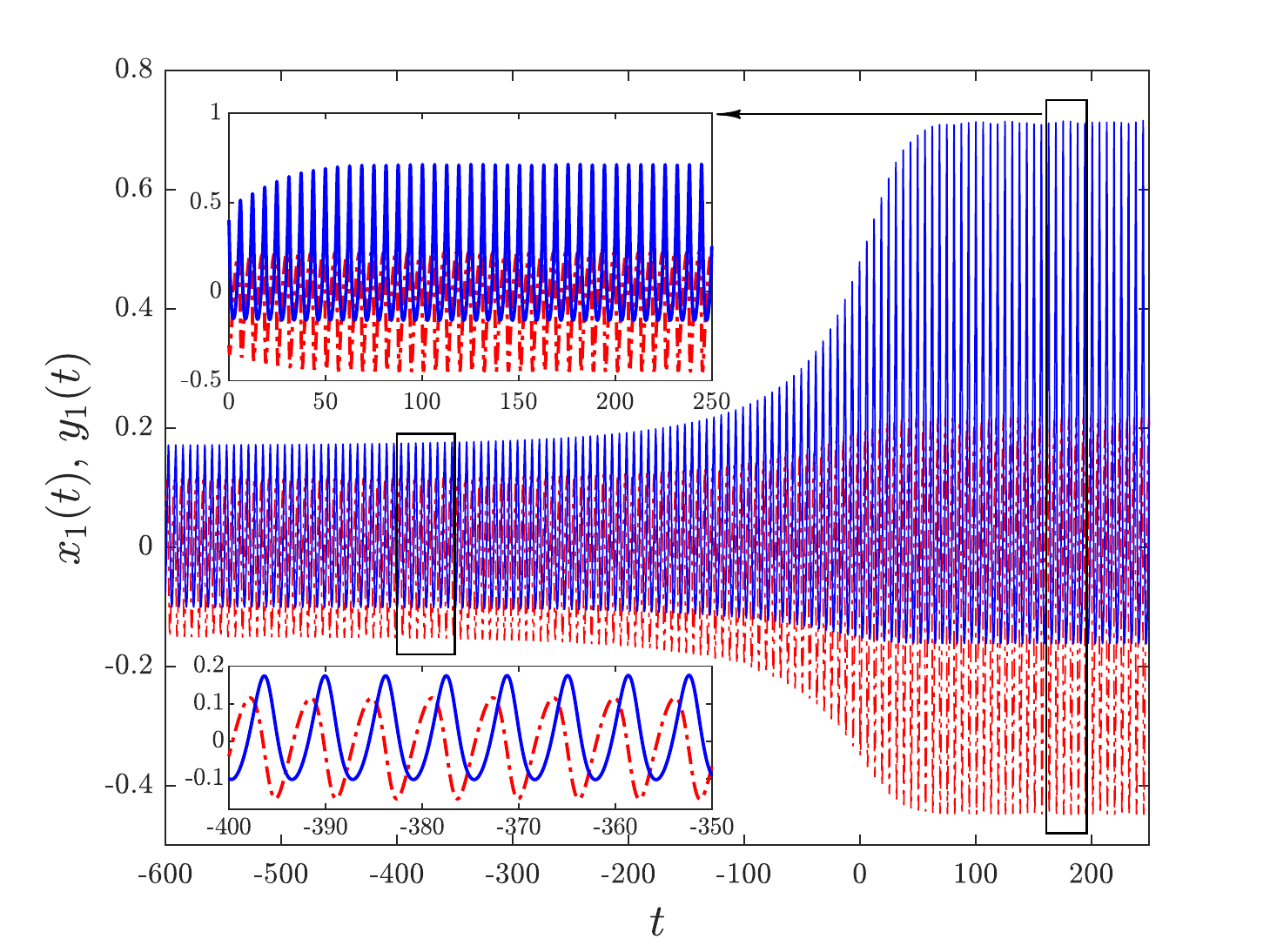}}\,
\subfloat[Trajectories for \(x_1(t)\) and $y_1(t)$ in inverse time converge to the internal invariant torus. \label{6c}]{\includegraphics[width=2.2in]{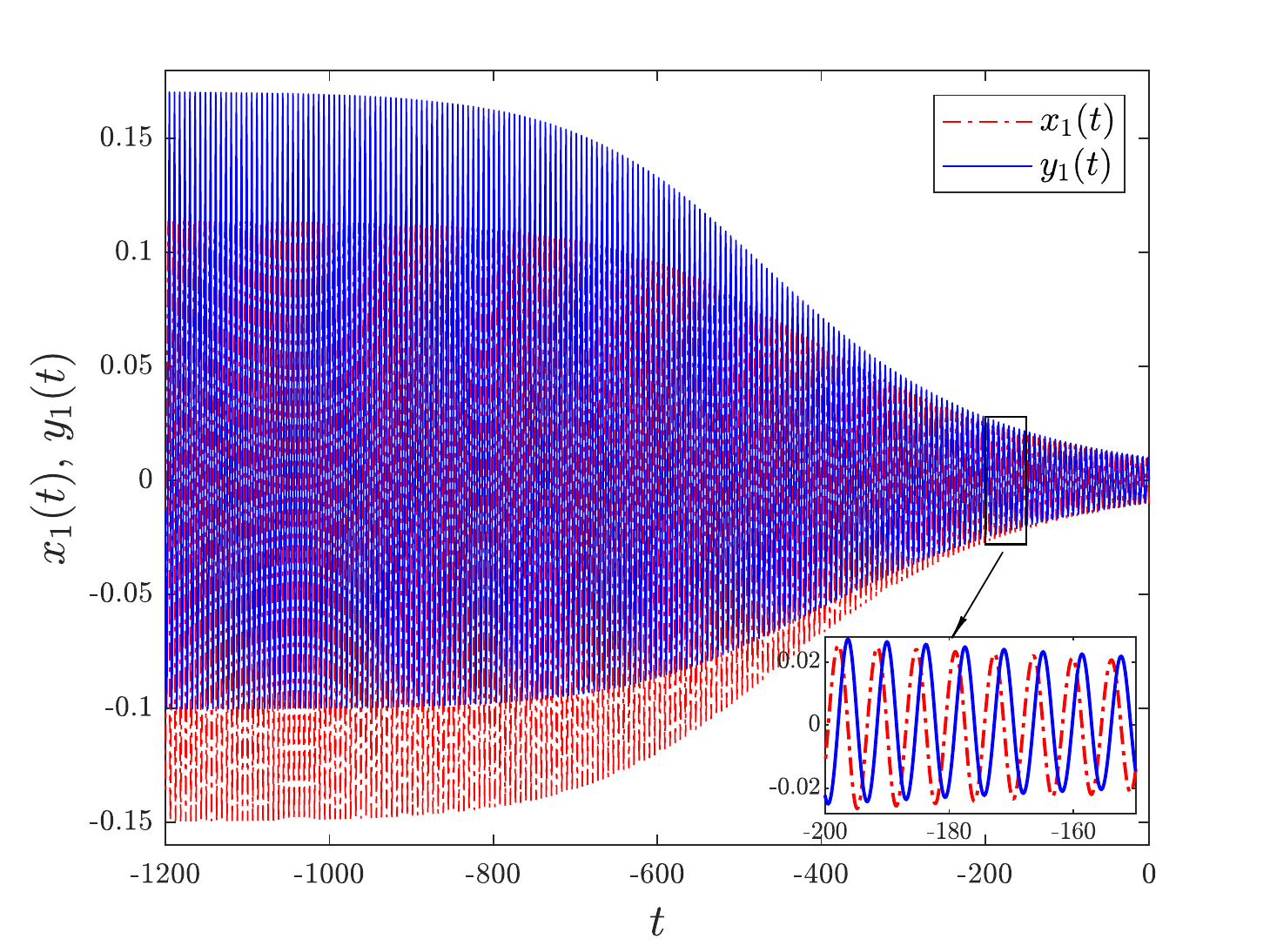}}\\
\subfloat[Time series for \(x_3(t)\) and $y_3(t)$. These along with Figure \ref{S2x1y1A} converge to the external \(\mathbb{T}_4\)-torus. \label{S2x2y2A}  ]{\includegraphics[width=2.2in]{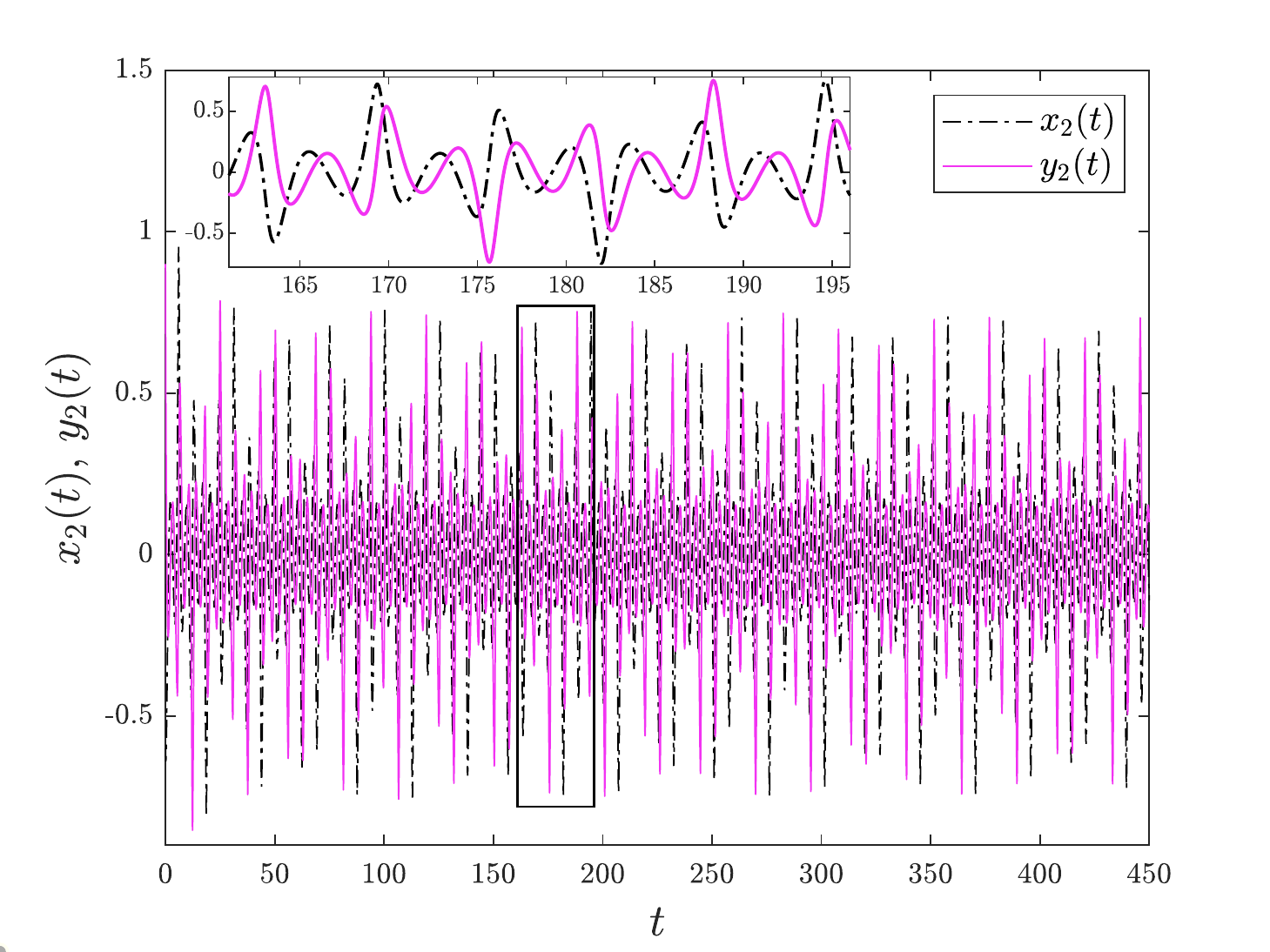}}\,
\subfloat[Two trajectories for \(x_3(t)\) and $y_3(t)$ whose \(\alpha\)- and \(\omega\)-limit sets are the internal and external 4-tori. \label{S2x2y2B} ]{\includegraphics[width=2.2in]{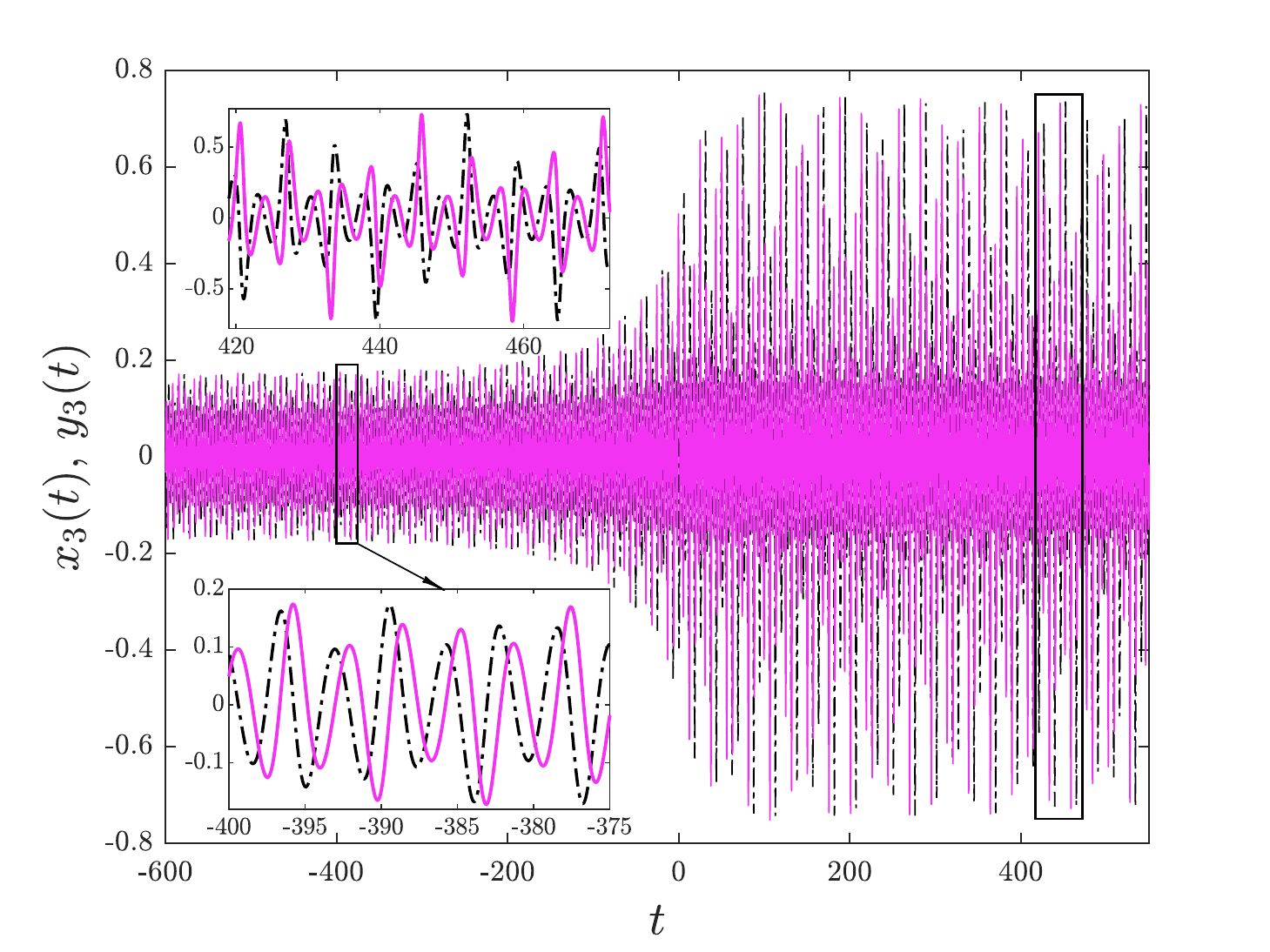}}\,
\subfloat[Backward-time trajectories of \(x_3(t)\) and $y_3(t)$.]{\includegraphics[width=2.2in]{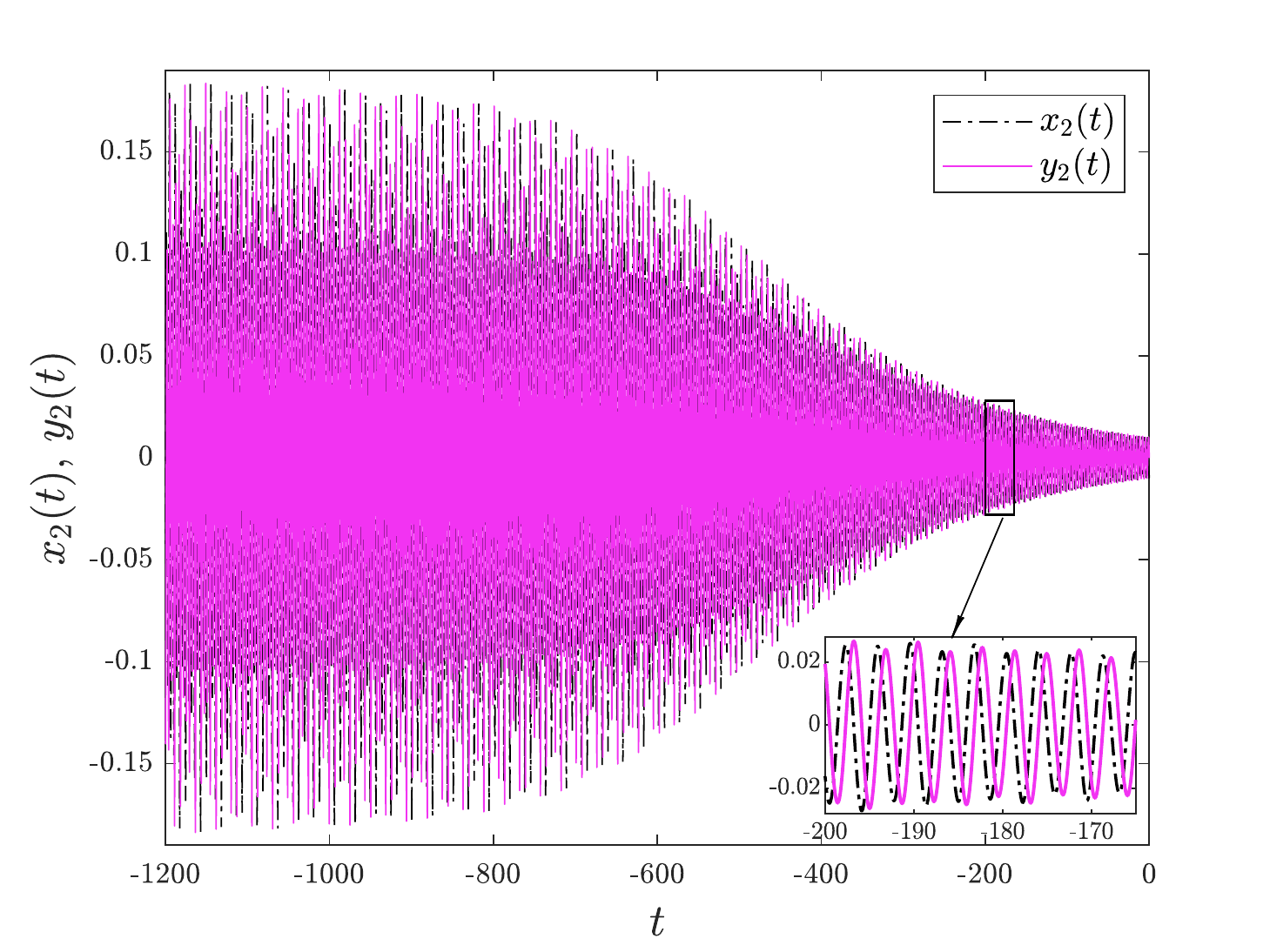}}
\caption{The controlled numerical trajectories depicting two invariant tori for system \eqref{EulBifExm}, \((c_1, c_2, c_3)= (\frac{2}{\sqrt{2}}, \frac{1}{\sqrt{2}}, 0)\in \mathbb{S}^2_{\geq 0},\) $\mu_0:=-0.005, \mu_3:=0.7, \mu_4:=0,$ and constants in example \ref{DifLeafTrans2}.  }\label{GenericCaseS1}
\end{figure}

\begin{exm}\label{DifLeafTrans2}
Let \(k=2\) and \(\sigma\in S^2_3\) where \((\sigma(1),\sigma(2), \sigma(3))=(1, 3, 2).\) As a result, \(\mathcal{M}_{2, \sigma}=\{\x\in\mathbb{R}^6| (x_1, y_1)\neq 0, (x_3, y_3)\neq 0, (x_2, y_2)=0\}\) and
\begin{equation*}
\mathcal{M}_{2, \sigma}^{(c_1, c_3)}=\{\x\in\mathcal{M}_{2, \sigma}\,|\,c_3\|(x_1, y_1) \|=c_1\|(x_3, y_3)\| \}
\end{equation*} where \((c_1, c_3)\) is denoted on behalf of \((c_1, 0, c_3)\in \mathbb{S}^2\). By transforming the \(\mathcal{M}_{2, \sigma}^{(c_1, c_3)}\)-leaf vector field associated with \eqref{EulBifExm} into the Lie algebra \(\mathscr{J}\) via the homeomorphism \(\Psi\), we obtain
\begin{eqnarray*}
&\sum_{i=1,3}\left(\frac{\sqrt{i}w_i\partial}{\mathbf{i}\partial w_i}-\frac{\sqrt{i}z_i\partial}{\mathbf{i}\partial z_i}\right)+\left(\mu_0+\left(\frac{\alpha_1-\mathbf{i}\alpha_2}{2}\right)rz_1+\left(\frac{\alpha_1+\mathbf{i}\alpha_2}{2}\right)rw_1
+\left(\frac{\alpha_5-\alpha_6}{4}\right)r^2{z_1}^2
+\left(\frac{\alpha_5-\alpha_6}{4}\right)r^2{w_1}^2  \right)\frac{r\partial}{\partial r}&\\
&+\left(\left(\frac{\alpha_5+\alpha_6}{2}\right)r^2\!+\!\frac{{c_3}^2}{{c_1}^2}\left(\frac{z_1+w_1}{2}\right)
\left(\left(\frac{\alpha_7-\alpha_8}{4}\right)r^3{z_3}^2+\left(\frac{\alpha_7+\alpha_8}{2}\right)r^3+\left(\frac{\alpha_7-\alpha_8}{4}\right)r^3{w_3}^2\right) \right)\frac{r\partial}{\partial r}.&
\end{eqnarray*}
For numerical bifurcation control simulation, let $a_1=2, a_2=1, a_3=-1, a_4=1, a_7=a_8=3$ and set \(\mu_i=0\) for \(i\neq 0, 3, 4.\) Then, the parametric leaf normal form up to degree seven is given by
\begin{eqnarray*}
&\sum_{i=1, 3}\!\frac{\sqrt{i}\partial}{\partial \theta_i}\!+\!\sum_{i=1, 3}\!\left(\mu_0\!+\!\frac{\mu_3\!+\!\mu_4}{2}{\rho_1}^2\!+\!\big(\frac{7\mu_3 \!+ 13\mu_4}{8}\!+\!\frac{2({c_1}^2\!+\!3{c_2}^2)\mu_0}{{c_1}^2}\!+\!
\frac{3({c_1}^2-4{c_2}^2)}{4{c_1}^2}\big)\!{\rho_1}^4\!+\!\frac{3(5{c_1}^2-34{c_2}^2)}{4{c_1}^2}{\rho_1}^6\right)\!
\frac{c_i\rho_{1}\partial}{{c_1}\partial\rho_{i}}.&
\end{eqnarray*}
Since \(b_1(0, C)=0\), we choose \(C=(c_1,0 , c_3)\in \mathbb{S}^3\) such that \(b_2(0, C)=\frac{3({c_1}^2-4{c_2}^2)}{{c_1}^2}\neq 0\). Then, the leaf case \(s=2\) is satisfied. Hence, we take \(c_1:= c_3:= \frac{1}{\sqrt{2}}\) and have \(b_2(0, C)= -\frac{9}{4}\neq 0.\) The infinite level parametric leaf-normal form up to twelfth-grade is as follows:
\begin{eqnarray*}
&\sum_{i=1, 3}\frac{\sqrt{i}\partial}{\partial \theta_i}+\sum_{i=1, 3} \left(\mu_0+\frac{4805{\mu_0}^2}{162}-\frac{95}{54}\mu_0\mu_3+(-\frac{29\mu_0}{3}+\frac{\mu_3}{2}){\rho_1}^2 -\frac{9}{4}{\rho_1}^4\right)\frac{c_i\rho_{1}\partial}{{c_1}\partial\rho_{i}}.&
\end{eqnarray*} For numerical simulation in Figures \ref{GenericCaseS1}, let \(\mu_0:=-0.005, \mu_3:=0.7,\) and \(\mu_4:=0.\) Solutions start from initial solutions $(-0.6, 0.9 ,0, 0, -0.6, 0.9),(-0.3, 0.4 ,0, 0, -0.3, 0.4,0)$ and $(-0.01, 0, 0, 0, -0.01, 0)$.
Figures \ref{S2x1y1A} and \ref{S2x2y2A} show forward time series converging to an external torus on the leaf-\(\mathcal{M}_{2, \sigma}^{(\frac{2}{\sqrt{2}}, \frac{1}{\sqrt{2}}, 0)}\) while trajectories in both backward and forward time are depicted in Figures \ref{S2x2y2A} and \ref{S2x2y2B}. These in forward time/backward time converge to the external/internal invariant torus. Figures \ref{6c} and \ref{S2x2y2B} depict a solution converging to the internal 4-torus in backward time.

Alternatively, we take \(c_1=2c_3=\frac{2}{\sqrt{5}}.\) Thereby, \(b_2(0, C)=0\) and \(b_3(0, C)=-\frac{21}{8}.\) This leads to the leaf case \(s=3.\) Next, the fourteenth-grade  truncation of the infinite level parametric leaf-normal form is
\begin{eqnarray*}
&\sum_{i=1, 3}\frac{\sqrt{i}\partial}{\partial \theta_i}+\sum_{i=1, 3}\left(\mu_0+(-\frac{3055\mu_0}{168}+\frac{\mu_3+\mu_4}{2}){\rho_1}^2+\left(\frac{7418665\mu_0}{28224}-\frac{2761\mu_3}{336} -\frac{2509\mu_4}{336}\right){\rho_1}^4-\frac{21}{8}r^6\right)\frac{c_i\rho_{1}\partial}{{c_1}\partial\rho_{i}}.&
\end{eqnarray*} We let \(\mu_4=-0.98\mu_3.\) The associated transition sets are depicted in Figures \ref{S2GTS} and \ref{S3TS} for the cases \(s=2\) and \(s=3\). The estimated transition varieties corresponding with equations \eqref{SupP}, \eqref{SubP} and \eqref{2SND} are given by
\begin{eqnarray*}
&T_{2SN}=\{(\mu_0, \mu_3)|-4(\frac{8\nu_1}{21}+\frac{64{\nu_2}^2}{1323})^3+27(\frac{8\mu_0}{21}+\frac{1024{\nu_2}^3}{250047}+\frac{64\nu_1\nu_2}{1323})^2=0, \mu_0>0\}&\\
&T_{Psup}=\{(\mu_0, \mu_3)| \mu_0=0, \mu_3> -0.13\}, \quad T_{Psub}=\{(\mu_0, \mu_3)| \mu_0=0, \mu_3< -0.13\}&
\end{eqnarray*}
where \(\nu_1:=\frac{3055\mu_0}{168}+\frac{1\mu_3}{100}\) and \(\nu_2:=\frac{7418665\mu_0}{7418665}+\frac{15109\mu_3}{16800}\); see Figure \ref{S3TS}.
\begin{figure}[t]
\centering
\subfloat[{Forward-time trajectories of $\rho_1(t)$ and $\rho_2(t).$}\label{S3R1R3A}]{\includegraphics[width=.47\linewidth,height=1.6in]{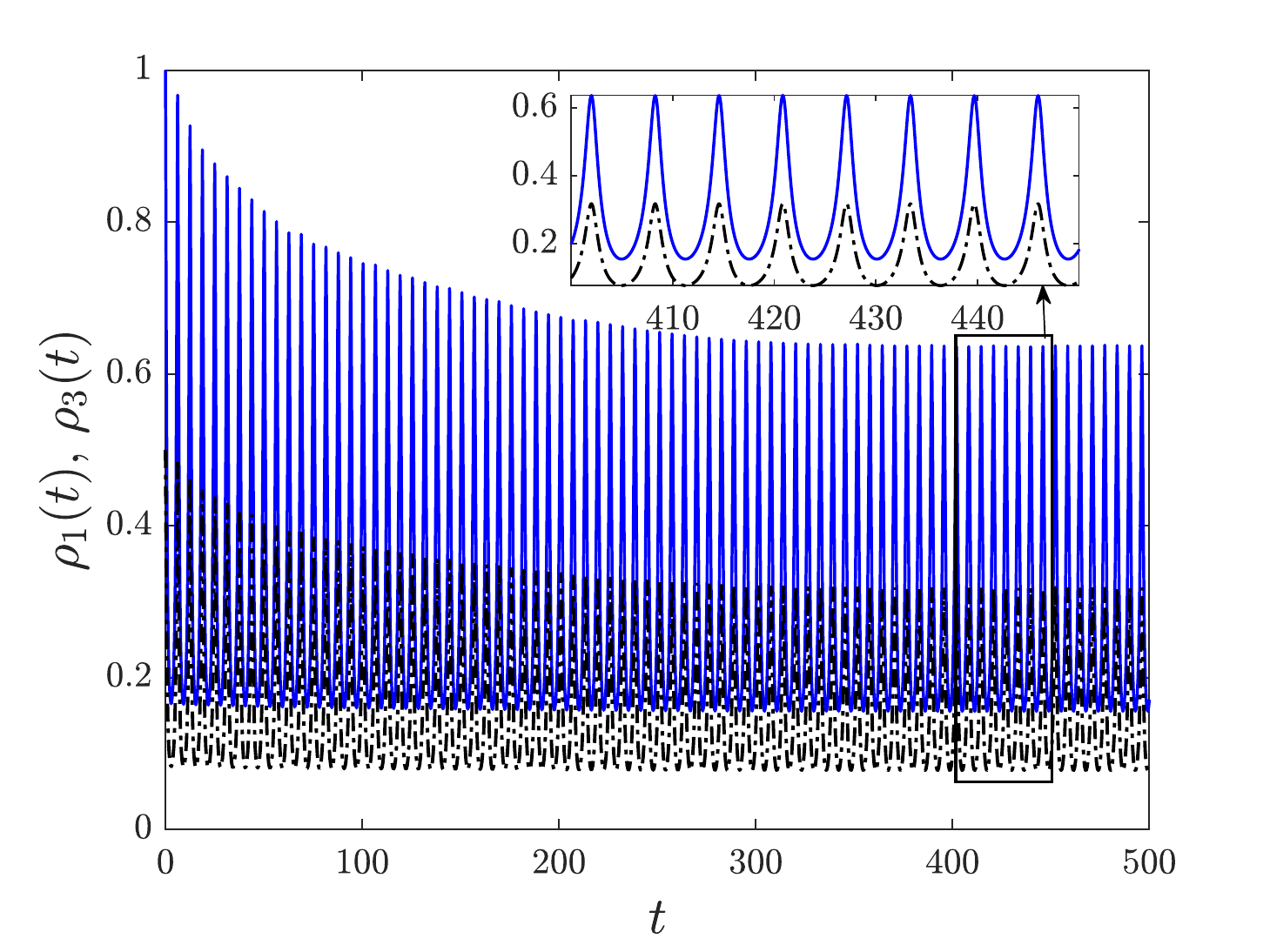}}\;
\subfloat[{Forward and backward-time trajectories of $\rho_1(t)$ and $\rho_2(t).$ } \label{S3R1R3B}]{\includegraphics[width=.47\linewidth,height=1.6in]{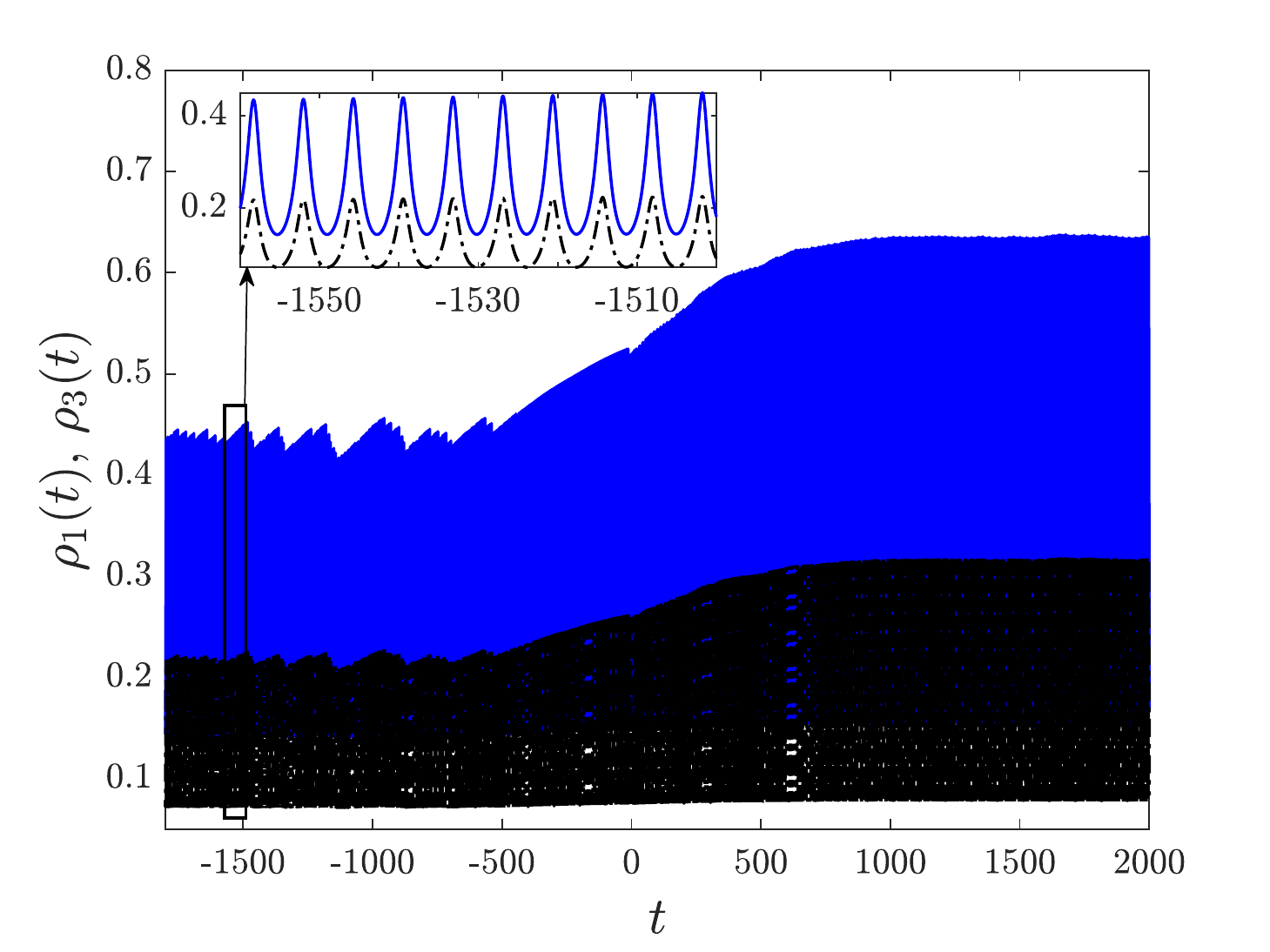}};
\subfloat[{Trajectories of $\rho_1(t)$ and $\rho_2(t)$.} \label{S3R1R3C}]{\includegraphics[width=.47\linewidth,height=1.6in]{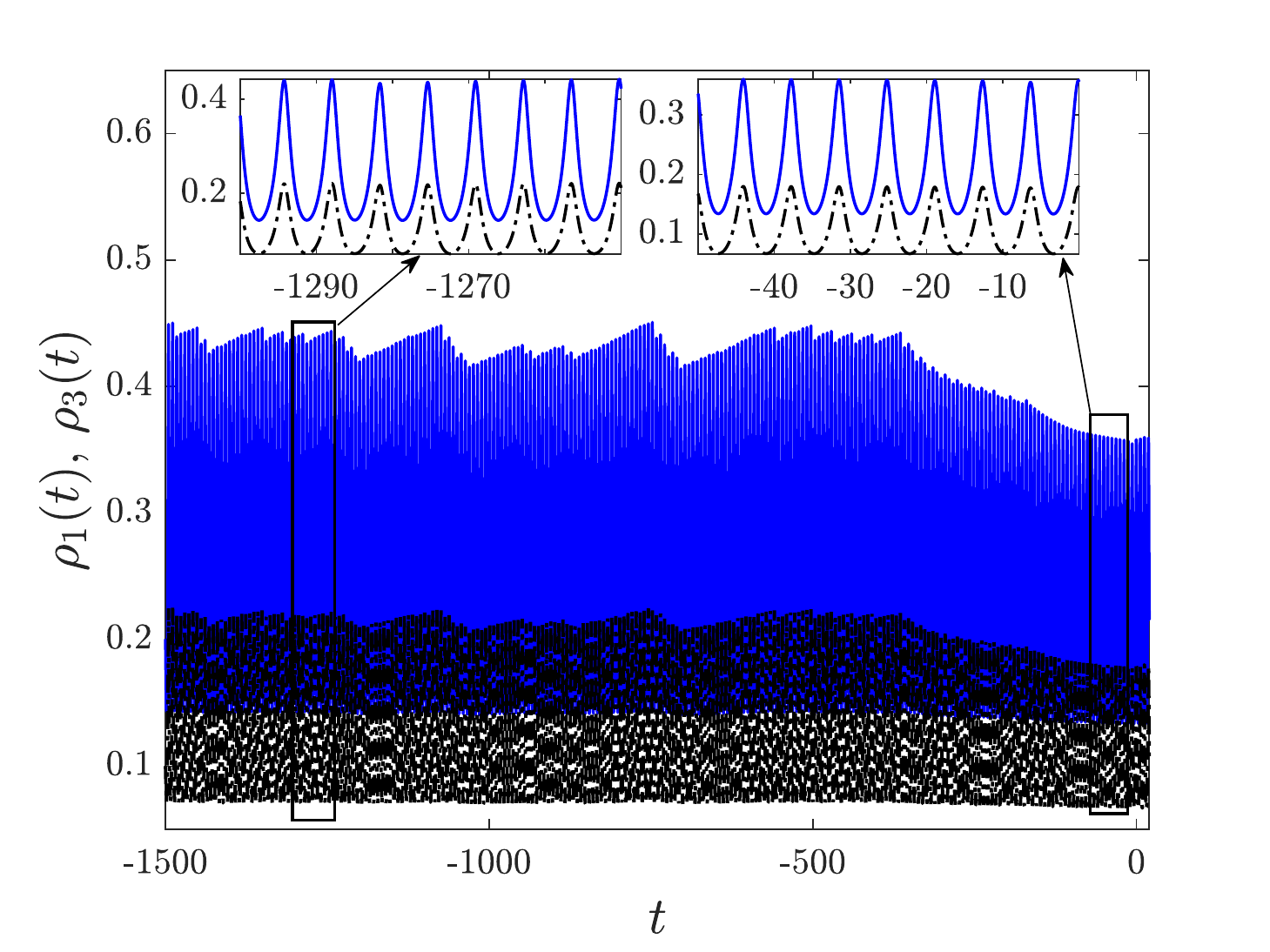}}\;
\subfloat[{Inverse-time trajectories of $\rho_1(t)$ and $\rho_2(t).$ } \label{S3R1R3D}]{\includegraphics[width=.47\linewidth,height=1.6in]{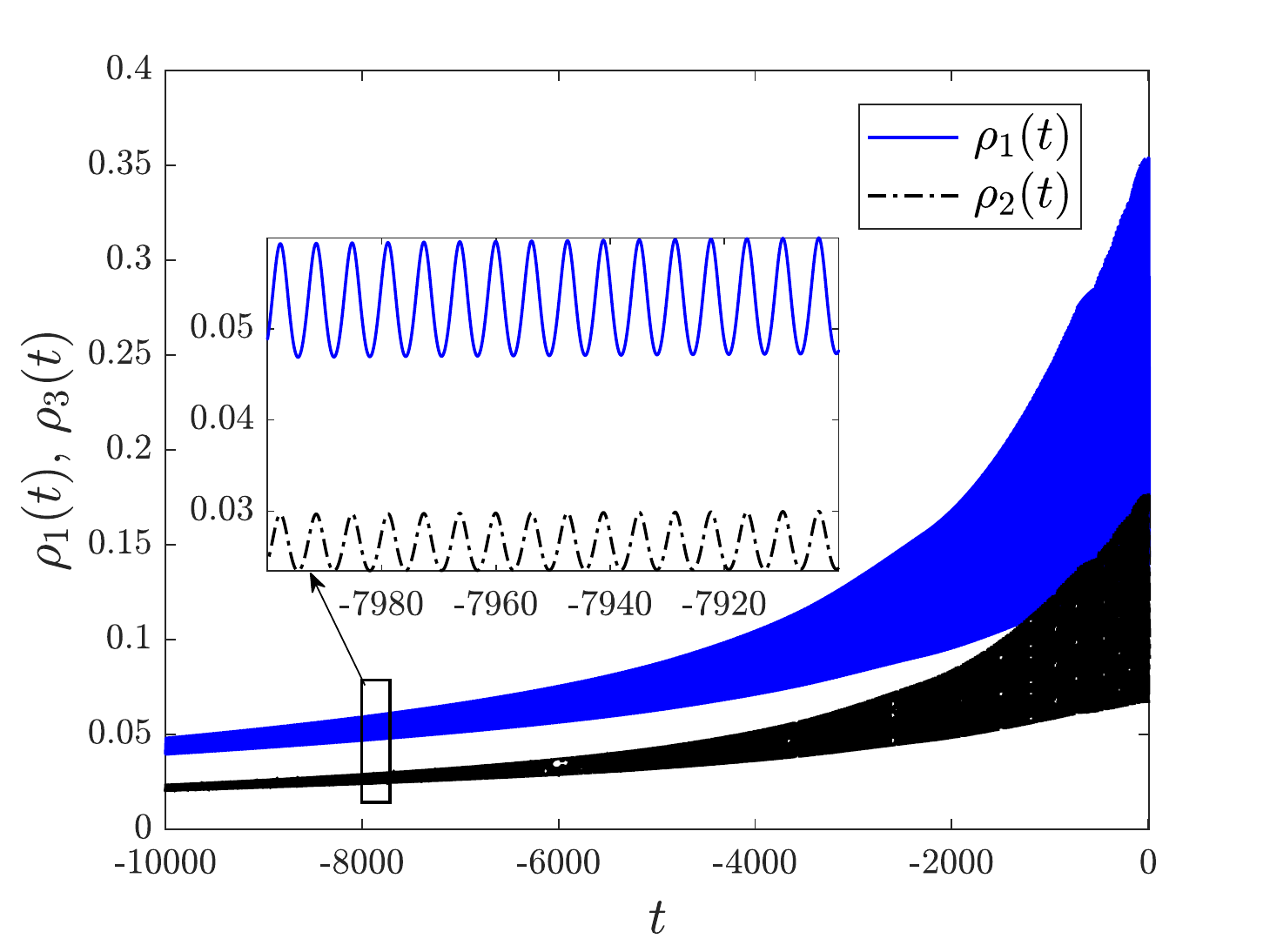}}
\caption{Three flow-invariant \(\mathbb{T}_4\)-tori on the leaf-\(\mathcal{M}_{2, \sigma}^{(\frac{2}{\sqrt{5}}, 0, \frac{1}{\sqrt{5}})}\) associated with example \ref{DifLeafTrans2}.\label{Fig3Tori} }
\end{figure}
For a numerical bifurcation control simulation, we take \(\mu_0=5e-5,\) \(\mu_3=-0.4\). Figure \ref{Fig3Tori} illustrates the existing three tori on the leaf \(\mathcal{M}_{2, \sigma}^{(\frac{2}{\sqrt{5}}, 0, \frac{1}{\sqrt{5}})}\). The forward-time trajectories of $\rho_1(t)$ and $\rho_2(t)$ with the initial values $(-0.6, 0.8, 0, 0, -0.3, 0.4)$ in Figure \ref{S3R1R3A} demonstrates a stable (external flow-invariant) \(\mathbb{T}_4\)-torus. Forward and backward-time trajectories of \(\rho_1(t)\) and \(\rho_2(t)\) corresponding with the initial values \((-0.35, 0.35, 0, 0, -0.175, 0.175)\) depicts a stable external (limit) torus and an unstable torus living inside the stable external torus; see Figure \ref{S3R1R3B}. The inverse-time trajectories of $\rho_1(t)$ and $\rho_2(t)$ with the initial values $(-0.194, 0.264, 0, 0, -0.097, 0.132)$ in Figure \ref{S3R1R3D} confirms the instability of the origin while it resides inside the unstable torus. Hence, there is also a third stable flow-invariant torus living the unstable torus illustrated by Figures \ref{S3R1R3B} and \ref{S3R1R3C}. Figure \ref{S3R1R3C} is the numerical solution of the system \eqref{EulBifExm}-\eqref{ScalarFunction} associated with initial values \((-0.2, 0.298, 0, 0, -0.1, 0.149).\)
\end{exm}

The local bifurcation of Eulerian flows are not only associated with the corresponding bifurcations of the reduced leaf-systems but also they are associated with the changes of the invariant leaf-manifolds. However, the analysis of the \(2n\)-dimensional system is not a straightforward corollary of those on individual leaves. Section \ref{SecTCWBif} deals with the bifurcation analysis of \(2k\)-dimensional cell-systems for \(k\leq n.\)

\section{Toral CW complexes and cell-bifurcations } \label{SecTCWBif}

Bifurcation varieties for a \(2n\)-dimensional vector field are not necessarily the same as leaf-transition sets. Leaf-transition sets provide a partition to the parameter space according to the topological qualitative changes in parametric leaf-vector fields. However, cell-bifurcation transition varieties here refer to the partition of the parameter space according to the dynamics of the Eulerian system on a closed cell \(\overline{\mathcal{M}_{k, \sigma}},\) that is, the {\it closure of an open \(2k\)-cell} \(\mathcal{M}_{k, \sigma}\) for \(k\leq n\) and \(\sigma\in S^k_n\). Cell-bifurcations are involved with flow-invariant toral CW complexes. Hence, we first describe toral CW complexes and then, deal with their cell-bifurcations for two most generic truncated one-parametric \(2k\)-cell normal form systems; also see \cite{GazorShoghiEulNF}.

\begin{notation}
\begin{itemize}
  \item Using the notation \(S^k_n\) in equation \eqref{Skn}, we introduce
\be
S^{l, \sigma}_n:= \left\{\gamma\in S^l_n|\, \{\sigma(i)| i> k\}\subseteq\{\gamma(i)| i> l\} \right\} \quad\hbox{ for } \sigma\in S^k_n\; \hbox{ and }\; l\leq k. \ee
Hence, \(S^{l, \sigma}_n\) has \({k\choose k-l}\)-number of elements and \(S^{k, \sigma}_n=\{\sigma\}\) for \(\sigma\in S^k_n.\) For instance, let \(n=4, k=2,\) \(\sigma(1):=2,\) \(\sigma(2):=3,\) \(\sigma(3):=1,\) and \(\sigma(4):=4.\) Then, \(\sigma\in S^2_n\) and \(S^{1, \sigma}_n=\{\sigma_1, \sigma_2\},\) where \(\sigma_1(1)=2, \sigma_1(2)=1, \sigma_1(j)=j\) for \(j=3, 4,\) and \(\sigma_2(1)=3,\) \(\sigma_2(2)=1,\) \(\sigma_2(3)=2,\) and \(\sigma_2(4)=4.\)
  \item Denote \(\mathbb{B}^{k}\subset\mathbb{R}^k\) for the \(k\)-open ball when \(k>0,\) and \(\mathbb{B}^{0}:= \{0\}\). Notation
\(\overline{\mathbb{B}^{k}}\) is used for the \(k\)-closed ball in \(\mathbb{R}^k\) while \(\mathbb{T}^x_{n+1}\) stands for an \(x\)-dependent \(n+1\)-dimensional Clifford torus. For \(\gamma\in S^{l, \sigma}_n,\) denote
\bas
&\mathbb{B}^{l, \gamma}:= \{(c_{i})^n_{i=1}\in \mathbb{R}^n |\, \sum^l_{j=1} c_{\gamma(j)}^2=1, c_{\gamma(j)}=0 \hbox{ for } j>l\},&
\eas and \(\mathbb{S}^{l-1, \gamma}_{>0}\) by equation \eqref{sksig}.
\end{itemize}
\end{notation}

Our main goal in the next lemma (and the illustrations in examples \ref{exm2}-\ref{exm3}) is to provide a regular CW complex decomposition for \(\overline{ \mathbb{S}^{k-1, \sigma}_{>0}}.\) This decomposition is the actual decomposition imposed by the closed cell-dynamics associated with Eulerian flows latter in this section.

\begin{lem}\label{Splus} The space \(\overline{ \mathbb{S}^{k-1, \sigma}_{>0}}\) is a regular CW complex.
\end{lem}
\begin{proof}
Recall that \(\mathbb{S}^{k-1, \sigma}_{>0}:=\{C=(c_1, \cdots, c_n)|\sum_{i=1}^{k}{c_{\sigma(i)}}^2=1, c_{\sigma(i)}>0 \hbox{ for } i\leq k \hbox{ and } c_{\sigma(i)}=0 \hbox{ for } j>k\}.\) Let
\be\label{ParI}\partial_i\mathbb{S}^{l-1,\gamma}_{>0}:=\sqcup_{\bar{\gamma}\in S^{l-i, \gamma}_{n}}\mathbb{S}^{l-i-1,\bar{\gamma}}_{>0}\ee be a union of disjoint \(l-i-1\)-dimensional submanifolds in the boundary of $\mathbb{S}^{l-1,\gamma}_{>0},$ where \(\gamma\in S^{l, \sigma}_n,\) \(l\leq k,\) and \(i\geq0.\) Then,
\begin{eqnarray*}
&\partial_0\mathbb{S}^{l-1, \gamma}_{>0}=\mathbb{S}^{l-1,\gamma}_{>0},\quad \partial\mathbb{S}^{l-1, \gamma}_{>0}=\partial_1\mathbb{S}^{l-1, \gamma}_{>0}\sqcup\partial\partial_1\mathbb{S}^{l-1, \gamma}_{>0}, \quad\partial\partial_1\mathbb{S}^{l-1, \gamma}_{>0}=\partial_2\mathbb{S}^{l-1, \gamma}_{>0}\sqcup\partial\partial_2\mathbb{S}^{l-1, \gamma}_{>0}, &
\end{eqnarray*} \(\partial\partial_i\mathbb{S}^{l-1, \gamma}_{>0}=\partial_{i+1}\mathbb{S}^{l-1, \gamma}_{>0}\sqcup\partial\partial_{i+1}\mathbb{S}^{l-1, \gamma}_{>0}\) and
\be\label{DecParI} \partial \mathbb{S}^{l-1, \gamma}_{>0}=\sqcup_{i=1}^{l-1}\partial_i\mathbb{S}^{l-1,\gamma}_{>0}, \qquad\hbox{ for }\; l\leq k \;\hbox{ and } \; \gamma\in S^{l, \sigma}_n.\ee
Further,
\(\partial \mathbb{S}^{l-1, \sigma}_{>0}=\sqcup_{i=1}^{l-1}\sqcup_{\bar{\gamma}\in S^{l-i, \gamma}_{n}}\mathbb{S}^{l-i-1,\bar{\gamma}}_{>0}\) and \(\overline{ \mathbb{S}^{l-1, \gamma}_{>0}}= \sqcup^l_{i=1}\sqcup_{\bar{\gamma}\in S^{i, \gamma}_n} \mathbb{S}^{i-1, \bar{\gamma}}_{>0}.\) We claim that each \(\mathbb{S}^{l-1, \gamma}_{>0}\) represents an \(l-1\)-cell, is homeomorphic to \(\mathbb{B}^{l-1}\), and this cell decomposition constitutes a CW complex structure for \(\overline{ \mathbb{S}^{l-1, \sigma}_{>0}}.\) For each \(l\) and \(\gamma,\) we need to introduce an attaching map
\ba\label{Philgamma}
&{\tilde{\Phi}_{l, \gamma}}: \overline{\mathbb{B}^{l-1}}\rightarrow \overline{\mathbb{S}^{l-1, \gamma}_{>0}}= \sqcup^{l}_{j=1}\sqcup_{\bar{\gamma}\in S^{j, \gamma}_n} \mathbb{S}^{j-1, \bar{\gamma}}_{>0},&
\ea that is a homeomorphism. Since \(\overline{\mathbb{S}^{l-1}_{>0}}\) and \(\overline{\mathbb{S}^{l-1, \gamma}_{>0}}\) are homeomorphic, let \(q_l=\sum^l_{i=1}\frac{1}{\sqrt{l}}\mathbf{e}^l_{i}\in \mathbb{S}^{l-1}_{>0}\subset\mathbb{R}^l.\) For a construction of \({\tilde{\Phi}_{l, \gamma}}\), consider the \(l\)-disc \(\mathbb{B}_{\frac{1}{2\sqrt{l}}}(q_l)\) centered at \(q_l\) and radius \(\frac{1}{2\sqrt{l}},\) the intersection of \(\overline{\mathbb{B}_{\frac{1}{2\sqrt{l}}}(q_l)}\) with the \(l-1\) sphere, \ie \(\big(q_l+\frac{1}{2\sqrt{l}}\overline{\mathbb{B}^l}\big)\cap \mathbb{S}^{l-1},\) and the family of all \(l-2\)-hyperplanes \(H\) passing through the origin and \(q_l.\) The spaces \((q_l+\frac{1}{2\sqrt{l}}\overline{\mathbb{B}^l})\cap \mathbb{S}^{l-1}\) and \(\overline{\mathbb{S}^{l-1}_{>0}}\) are homeomorphic. A homeomorphism can be constructed via a uniform rescaling of the arcs obtained from the intersection of \(H\) with \(\overline{\mathbb{S}^{l-1}_{>0}}\) and \(\mathbb{B}_{\frac{1}{2\sqrt{l}}}(q_l)\cap \mathbb{S}^{l-1},\) respectively. Since there is a homeomorphism between \((q_l+\frac{1}{2\sqrt{l}}\overline{\mathbb{B}^l})\cap \mathbb{S}^{l-1}\) and \(\overline{\mathbb{B}^{l-1}},\) the combination of these homeomorphisms constructs the expected homeomorphism \({\tilde{\Phi}_{l, \gamma}}.\) Hence, the space \(\overline{\mathbb{S}^{k-1, \sigma}_{>0}}\) is a regular CW complex.
\end{proof}

\begin{exm}\label{exm2} Let \(k=n=2\) and \(\overline{\mathbb{B}^{1}}= [-1, 1].\) The space \(\overline{\mathbb{S}^{1, I}_{>0}}=\{C=(c_1, c_2)| {c_{1}}^2+{c_{2}}^2=1, c_{i}\geq 0 \hbox{ for } i=1, 2\}\) is a regular CW complex, where \(\mathbb{S}^{1, I}_{>0}=\{C=(c_1, c_2)|\, {c_{1}}^2+{c_{2}}^2=1, c_{i}>0 \hbox{ for } i=1, 2\}.\) We have \(S^{1, I}_n= \{\gamma_1, \gamma_2\}\) for \((\gamma_1(1), \gamma_1(2)):= (1, 2)\) and  \((\gamma_2(1), \gamma_2(2)):= (2, 1).\) Hence,
\(\partial_1\mathbb{S}^{1, I}_{>0}= \mathbb{S}^{0, \gamma_1}_{>0}\sqcup \mathbb{S}^{0,\gamma_2}_{>0}= \{(1, 0), (0, 1)\}.\) A continuous attaching map associated with \(\mathbb{S}^{1, I}_{>0}\) follows
\ba\label{Philgamma2}
&{\tilde{\Phi}_{2, I}}: \overline{\mathbb{B}^{1}}\rightarrow \overline{ \mathbb{S}^{1, I}_{>0}}= \mathbb{S}^{1, I}_{>0}\sqcup (\mathbb{S}^{0, \gamma_1}_{>0}\sqcup \mathbb{S}^{0,\gamma_2}_{>0}), \hbox{ by } {\tilde{\Phi}_{2, \gamma}}(s)=\left(\cos(\frac{\pi}{4}s+\frac{\pi}{4}), \sin(\frac{\pi}{4}s+\frac{\pi}{4})\right). &
\ea Here, \({\tilde{\Phi}_{2, I}}\) is a homeomorphism.
\end{exm}

\begin{exm}[A CW-decomposition for \(\overline{\mathbb{S}^{2, I}_{>0}}\)]\label{exm3} Let \(k=n=3\) and consider \(\mathbb{S}^{2, I}_{>0}=\{C\in \mathbb{S}^{2}|\, c_{i}>0 \hbox{ for } i=1, 2, 3\}.\) Let \(S^{2, I}_3= \{\gamma_1, \gamma_2, \gamma_3\},\) where
\bas
&  (\gamma_1(1), \gamma_1(2), \gamma_1(3)):= (2, 3, 1), \quad (\gamma_2(1), \gamma_2(2), \gamma_2(3)):= (1, 3, 2), &
\eas  and \((\gamma_3(1), \gamma_3(2), \gamma_3(3)):= (1, 2, 3).\) Hence,
\(\partial_1\mathbb{S}^{2, I}_{>0}= (\mathbb{S}^{1, \gamma_1}_{>0}\sqcup \mathbb{S}^{1,\gamma_2}_{>0})\sqcup \mathbb{S}^{1,\gamma_3}_{>0},\) where
\bes \mathbb{S}^{1,\gamma_i}_{>0}= \{C\in \mathbb{S}^{2}| c_j>0 \hbox{ for } j\neq i, c_{i} =0\}.\ees
Further, \(S^{1, I}_n= \{\bar{\gamma}_1, \bar{\gamma}_2, \bar{\gamma}_3\}\) for \((\bar{\gamma}_1(1), \bar{\gamma}_1(2), \bar{\gamma}_1(3)):= (1, 2, 3),\) \((\bar{\gamma}_2(1), \bar{\gamma}_2(2), \bar{\gamma}_2(3)):= (2, 1, 3),\) and \((\bar{\gamma}_3(1), \bar{\gamma}_3(2), \bar{\gamma}_3(3)):= (3, 1, 2).\) Hence,
\bes \partial_2\mathbb{S}^{2, I}_{>0}= \{(1, 0, 0), (0, 1, 0), (0, 0, 1)\}, \; \hbox{ where } \; \mathbb{S}^{0, \bar{\gamma}_1}_{>0}= \{(1, 0, 0)\}, \mathbb{S}^{0, \bar{\gamma}_2}_{>0}= \{(0, 1, 0)\},\ees and \(\mathbb{S}^{0, \bar{\gamma}_3}_{>0}= \{(0, 0, 1)\}.\)
\begin{figure}[t]
\centering 
\subfloat[\(\overline{\mathbb{S}^{2}_{>0}}\) and its \(\Pi_1\)-projection on \((c_1, c_2)\)-plane \label{CWcom2}]{
\includegraphics[width=.53\linewidth,height=2in]{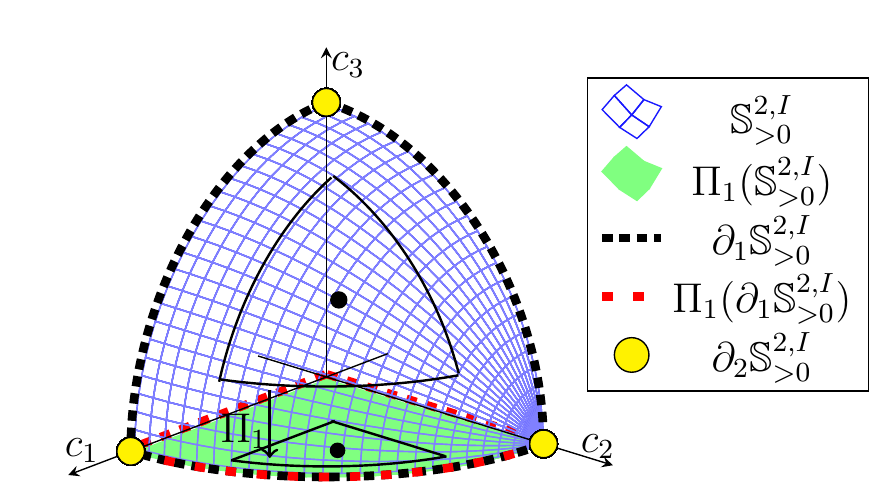}}\;
\subfloat[\(\Pi_1(\overline{\mathbb{S}^{2}_{>0}})\) circumscribed by a circle in \((c_1, c_2)\)-plane \label{CWcom1}]{
\includegraphics[width=.44\linewidth,height=1.7in]{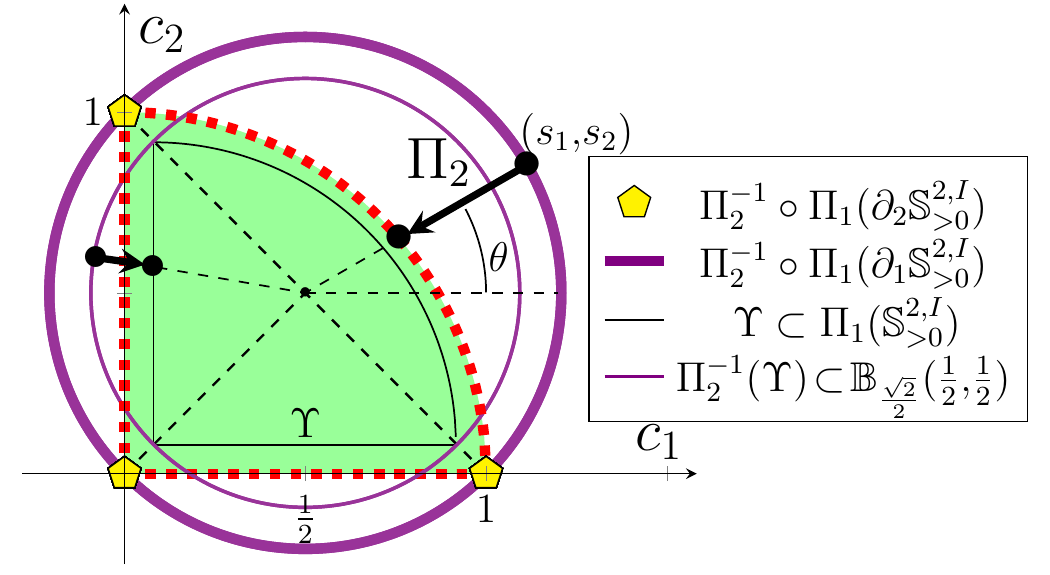}}
\caption{A CW complex decomposition for \(\overline{\mathbb{S}^{2}_{>0}}\).  }\label{CWcom}
\end{figure}
The space \(\overline{\mathbb{S}^{2, I}_{>0}}\) is homeomorphic with the sector in the \((x, y)\)-plane obtained by the projection map (see Figure \ref{CWcom2})
\begin{eqnarray*}
\Pi_1:D:=\{(c_1, c_2)|{c_1}^2+{c_2}^2\leq 1, c_1, c_2\geq 0\} \rightarrow \overline{\mathbb{S}^{2, I}_{>0}},\quad  \Pi_1(c_1, c_2)=\left(c_1, c_2, \sqrt{1-c_1^2-c_2^2}\right).
\end{eqnarray*} The projected (green) sector is circumscribed by a full circle as illustrated in Figure \ref{CWcom1}. Then, a uniform rescaling of segments of the circle's radius inside the (green) sector makes a homeomorphism between the sector and \(\overline{\mathbb{B}}^2\); see Figure \ref{CWcom1}. This is given by \(\Pi_2:\frac{\sqrt{2}}{2}\overline{\mathbb{B}}^2+(\frac{1}{2}, \frac{1}{2})\rightarrow D,\) \(\Pi_2(s_1, s_2)= (c_1, c_2),\) where
\begin{eqnarray*}
(c_1, c_2):=\left\{\!\begin{array}{lll}\!
(\frac{1}{2}, \frac{1}{2}) -\left({\frac {\left(\cos\theta+\sin\theta\right)-\sqrt{\sin 2\theta +3}}{\sqrt {2}}}\right) (s_1-\frac{1}{2},s_2 -\frac{1}{2}), &\qquad \theta\in [\frac{-\pi}{4},\frac{5\pi}{4}],\\ [2ex]
\frac{ r }{\sqrt{2}\max\{|2s_1-1|, |2s_2-1|\}} (2s_1-1, 2s_2-1)+ (\frac{1}{2}, \frac{1}{2}), &
\qquad \frac{5\pi}{4}\leq\theta\leq \frac{7\pi}{4},
\end{array}\right.
\end{eqnarray*}
\(\theta:=\tan^{-1}\frac{2s_2-1}{2s_1-1}.\) Further, \(\psi(0, 0)=(0, 0),\)
\begin{eqnarray*}
{\Pi_2}^{-1}(c_1, {c_2})=\left\{\!\begin{array}{lcc}\!
(\frac{1}{2}, \frac{1}{2})+\left(\frac{-1+c_1+{c_2}+\sqrt{3({c_1}^2+{c_2}^2)+2{c_1}{c_2}-4({c_1}+{c_2})+2}}{\sqrt{2}}\right)(\cos\theta, \sin\theta) & \quad \theta\in [\frac{-\pi}{4},\frac{5\pi}{4}], \\ [2ex]
\frac{\max\{|2{c_1}-1|, |2{c_2}-1|\}}{\sqrt{2}\sqrt{(2{c_1}-1)^2+(2{c_2}-1)^2}} (2{c_1}-1, 2{c_2}-1)+(\frac{1}{2}, \frac{1}{2}) & \quad \frac{5\pi}{4}\leq\theta\leq \frac{7\pi}{4}.
\end{array}\right.
\end{eqnarray*} We can transform \(\overline{\mathbb{B}}^2\) into \(\frac{\sqrt{2}}{2}\overline{\mathbb{B}}^2+(\frac{1}{2}, \frac{1}{2})\) using a shift and a rescaling. Then, the combination of this with \(\Pi_2\) and \(\Pi_1\) provides an attaching homeomorphism between \(\overline{\mathbb{B}}^2\) and \(\overline{\mathbb{S}^{2, I}_{>0}}.\)
\end{exm}

A toral CW complex \(X\) can be constructed by {\it attaching toral cells} associated with a cell decomposition of a regular CW complex \(\tilde{X}\). A {\it toral cell} is a smooth {\it toral bundle} and it here refers to a space homeomorphic to the Tychonoff product of an open CW-cell in a CW decomposition with a Clifford hypertorus. Thereby, a toral cell is a torus bundle over an open CW-cell whose fiber is a hypertorus. In this paper we encounter toral CW complexes as flow-invariant manifolds bifurcated from singular Eulerian cell-flows. Hence, we merely describe the toral CW complexes specifically to what appears in those cases. A toral cell here is homeomorphic to a Tychonoff product of an open CW \(k\)-cell with either a \(k+1\)-hypertorus or a \(k+2\)-hypertorus. Therefore, an even dimensional toral cell is always homeomorphic to a Tychonoff product of \(\mathbb{T}_{k+2}\) with a CW \(k\)-cell while odd dimensional toral cells are homeomorphic to the product of \(\mathbb{T}_{k+1}\) with an open CW \(k\)-cell. Hence, we denote a regular CW complex \(\tilde{X}\) as
\ba\label{CWC}
&\tilde{X}:=\sqcup_{m, i_{m}\in I_{m}}\tilde{\mathfrak{C}}^{i_{m}}_{\lfloor{\frac{m-1}{2}}\rfloor}\quad \hbox{ when } \tilde{\mathfrak{C}}^{i_{m}}_{\lfloor{\frac{m-1}{2}}\rfloor} \;\hbox{ is an open } \; \lfloor{\frac{m-1}{2}}\rfloor-\hbox{dimensional CW cell in } \tilde{X}&
\ea and the corresponding toral cell associated with \(\tilde{\mathfrak{C}}^{i_{m}}_{\lfloor{\frac{m-1}{2}}\rfloor}\) is homeomorphic to the hypertorus bundle \(\mathbb{T}_{m-\lfloor{\frac{m-1}{2}}\rfloor}\times \tilde{\mathfrak{C}}^{i_{m}}_{\lfloor{\frac{m-1}{2}}\rfloor}.\) In other words, \(m\) stands for the dimension of the toral cell associated with the CW cell \(\tilde{\mathfrak{C}}^{i_{m}}_{\lfloor{\frac{m-1}{2}}\rfloor}.\) This representation for the CW cell decomposition splits CW \({\lfloor{\frac{m-1}{2}}\rfloor}\)-cells into two categories based on their association with odd and even dimensional toral cells.

Let
\be\label{tildef} \tilde{\Phi}_{n, {i_n}}: \overline{\mathbb{B}^{n}}\rightarrow \overline{\tilde{\mathfrak{C}}^{i_n}_{\lfloor{\frac{n-1}{2}}\rfloor}}\subseteq \sqcup_{m\leq n, i_{m}\in I_{m}}\tilde{\mathfrak{C}}^{i_{m}}_{\lfloor{\frac{m-1}{2}}\rfloor}\subseteq \tilde{X}\ee
be the attaching homeomorphism associated with \(\tilde{\mathfrak{C}}^{i_n}_{\lfloor{\frac{n-1}{2}}\rfloor}.\) We shall correspond \(\tilde{\Phi}_{n, {i_n}}\) to an attaching map \(\Phi^{i_n}_{n}\) associated with \(\mathfrak{C}^{i_n}_{\lfloor{\frac{n-1}{2}}\rfloor}\) in the introduction of a {\it toral CW complex} \(X:=\sqcup_{i_n\in I_n, n} \mathfrak{C}^{i_n}_n\). Now we look for a regular CW decomposition on the closed disc \(\overline{\mathbb{B}^{n}}\) corresponding with toral cells of different odd and even dimensions. Denote \(\partial^{o, i_n}_{0} \overline{\mathbb{B}^{n}}:=\mathbb{B}^{n}\) and \(\partial^{e, i_n}_{0} \overline{\mathbb{B}^{n}}:=\emptyset\) while for \(1\leq k< n,\)
\be\label{6.8}
\partial^{o,i_n}_{n-k} \overline{\mathbb{B}^{n}}:={{\tilde{\Phi}_{n, {i_n}}}}^{-1}(\sqcup_{i_{2k+2}\in I_{2k+2}}\tilde{\mathfrak{C}}^{i_{2k+2}}_{k})\subseteq \partial \mathbb{B}^{n}\; \hbox{ and }\; \partial^{e,i_n}_{n-k} \overline{\mathbb{B}^{n}}:={{\tilde{\Phi}_{n, {i_n}}}}^{-1}(\sqcup_{i_{2k+1}\in I_{2k+1}}\tilde{\mathfrak{C}}^{i_{2k+1}}_{k})\subseteq \partial \mathbb{B}^{n}.
\ee Here, super-indices \(o\) and \(e\) stand for their associations with odd and even dimensional toral cell cases. For instance, the space \(\partial^{o,i_n}_{n-k} \overline{\mathbb{B}^{n}}\) will be associated with a \(2k+1\)-dimensional toral cell in \(X.\) Further, \(\overline{\mathbb{B}^{n}}=\cup^{n}_{k=0, \delta\in \{o, e\}}\partial^{\delta, i_n}_k \overline{\mathbb{B}^{n}}.\) We assume that the space
\bas
&\sqcup^{n}_{k=0}\sqcup_{x\in \partial^{o, i_n}_k \overline{\mathbb{B}^{n}}}\left(\mathbb{T}^x_{n-k+1}\times\{x\}
\right)\sqcup \sqcup^{n}_{k=1}\sqcup_{x\in \partial^{e, i_n}_k \overline{\mathbb{B}^{n}}}\left(\mathbb{T}^x_{n-k+2}\times\{x\}\right)&
\eas carries a metrizable topology so that
\begin{itemize}
  \item \(\sqcup_{x\in \mathbb{B}^n} \mathbb{T}^x_{n+1}\times \{x\}\) is homeomorphic to \(\mathbb{T}_{n+1}\times \mathbb{B}^n.\)
  \item \(\sqcup_{x\in \mathbb{B}^n} \mathbb{T}^x_{n+1}\times \{x\}\) is dense and a relatively compact open subset, \ie
\be\label{toralB}
\overline{\sqcup_{x\in \mathbb{B}^n} \mathbb{T}^x_{n+1}\times \{x\}}:= \sqcup^{n}_{k=0, x\in \partial^{o, i_n}_k \overline{\mathbb{B}^{n}}}\, (\mathbb{T}^x_{n-k+1}\times\{x\})\sqcup \sqcup^{n}_{k=1, x\in \partial^{e, i_n}_k \overline{\mathbb{B}^{n}}}\, (\mathbb{T}^x_{n-k+2}\times\{x\}).
\ee
\end{itemize}
 For brevity of notations, we shall denote the space \(\sqcup_{x\in \mathbb{B}^n} \mathbb{T}^x_{n+1}\times \{x\}\) by its homeomorphic space \(\mathbb{T}_{n+1}\times \mathbb{B}^n\). We now describe the assumed metrizable topology in more details as follows. Euclidian topology demonstrates the convergent sequences within individual open cells. More precisely, open \(2k+2\)-toral cell
\be
\sqcup_{x\in{{\tilde{\Phi}_{n, {i_n}}}}^{-1}(\tilde{\mathfrak{C}}^{i_{2k+2}}_{k})}(\mathbb{T}^x_{k+2}\times\{x\})\subset\overline{\mathbb{T}_{n+1}\times \mathbb{B}^n}\;\;\hbox{ is homeomorphic to } \;\;\mathbb{T}_{k+2}\times {{\tilde{\Phi}_{n, {i_n}}}}^{-1}(\tilde{\mathfrak{C}}^{i_{2k+2}}_{k})
\ee and open \(2k+1\)-toral cell \(\sqcup_{x\in{{\tilde{\Phi}_{n, {i_n}}}}^{-1}(\tilde{\mathfrak{C}}^{i_{2k+1}}_{k})}(\mathbb{T}^x_{k+1}\times\{x\})\) is homeomorphic to \(\mathbb{T}_{k+1}\times{{\tilde{\Phi}_{n, {i_n}}}}^{-1}(\tilde{\mathfrak{C}}^{i_{2k+1}}_{k})\) for \(k\leq n\). No sequence from an open cell in \(\overline{\mathbb{T}_{n+1}\times \mathbb{B}^n}\) converges to a point on another open cell with an equal or higher dimension, \eg
\bes
\partial\big(\sqcup_{x\in \partial^{e, i_n}_l \overline{\mathbb{B}^{n}}}\mathbb{T}^x_{n-l+2}\times\{x\}\big)\subseteq\sqcup^{n}_{k=l+1, x\in \partial^{o, i_n}_k \overline{\mathbb{B}^{n}}}\mathbb{T}^x_{n-k+1}\times\{x\}\sqcup \sqcup^{n}_{k=l+1, x\in \partial^{e, i_n}_k \overline{\mathbb{B}^{n}}}\mathbb{T}^x_{n-k+2}\times\{x\}.
\ees
Convergent sequences from a higher dimensional cell to a lower dimensional cell are as follows. For \(x_j\in \partial^{\delta_1, i_n}_k\overline{\mathbb{B}^{n}}\) and  \(1\leq j\leq \infty,\) a sequence
\bas
&(s_1^{x_j}, \ldots, s_{n-k+1}^{x_j}, x_j)\in \mathbb{T}^{x_j}_{n-k+1}\times \{x_j\}= (\prod^{n-k+1}_{j=1}\mathbb{S}^{1, x_j})\times \{x_j\},  &
\eas where \(\mathbb{S}^{1, x_j}\) is an \(x_j\)-dependent circle, approaches
\bas & (s_1^y, \ldots, s_{n-l+1}^y, y)\in \mathbb{T}^{y}_{n-l+1} \times \{y\}, \quad \hbox{for } \; k<l\leq n, \delta_1, \delta_2\in \{o, e\}, &\eas when \(x_j\in \partial^{\delta_1, i_n}_k \overline{\mathbb{B}^{n}}\subset \overline{\mathbb{B}^{n}}\) converges to \(y\in \partial^{\delta_2, i_n}_l \overline{\mathbb{B}^{n}}\subset \overline{\mathbb{B}^{n}}\) and \(n-l+1\)-number of \(s_i^{x_j}\)-components from \((s_1^{x_j}, \ldots, s_{n-k+1}^{x_j})\) correspond with and converge to \((s_1^y, \ldots, s_{n-l+1}^y)\) as \(j\) approaches infinity. However, the sequence of \(s_i^{x_j}\)-components from \((s_1^{x_j}, \ldots, s_{n-k+1}^{x_j})\) corresponding with the \(l-k\)-remaining indices collapse to a point as \(j\) converges to infinity. Roughly speaking, some \(\mathbb{S}^{1}\)-components of the toral fiber collapse to a point as points from a toral cell in \(\overline{\mathbb{T}_{n+1}\times \mathbb{B}^n}\) approaches a point on a neighboring lower dimensional toral cell. This naturally reduces the dimension of the corresponding torus.

\begin{defn}[Toral CW complexes]\label{TorCWDef} We refer to a Hausdorff space \(X\) with a finite partition \(\{\mathfrak{C}^{i_m}_m\},\) \ie \(X= \sqcup_{i_m\in I_m, m}\mathfrak{C}^{i_m}_m\) for finite number of finite index sets \(I_m\), as a {\it toral CW complex} and call each \(\mathfrak{C}^{i_{2k+1}}_{2k+1}\) and \(\mathfrak{C}^{i_{2k}}_{2k}\) by an open \(2k+1\)-toral cell and an open \(2k\)-toral cell in \(X\) when the following conditions hold.
\begin{itemize}
\item There exist a regular CW complex \(\tilde{X}\) given by \eqref{CWC} and homeomorphisms \(g^{i_{m}}_{m}\) so that \(\mathfrak{C}^{i_{2k+2}}_{2k+2}\) in \(X\) is \(g^{i_{2k+2}}_{2k+2}\)-homeomorphic to \(\mathbb{T}_{k+2}\times\tilde{\mathfrak{C}}^{i_{2k+2}}_{k}\) and each open \(2k+1\)-toral cell \(\mathfrak{C}^{i_{2k+1}}_{2k+1}\) is \(g^{i_{2k+1}}_{2k+1}\)-homeomorphic to \(\mathbb{T}_{k+1}\times\tilde{\mathfrak{C}}^{i_{2k+1}}_{k}\).
\item Consider the attaching map \(\tilde{\Phi}_{n, {i_n}}\) in \eqref{tildef} and the space
\(\overline{\mathbb{T}_{n+1}\times \mathbb{B}^n}\) described by \eqref{6.8}-\eqref{toralB}. Then,
there exists an attaching homeomorphism
\bes\Phi^{i_n}_n: \overline{\mathbb{T}_{n+1}\times \mathbb{B}^n}\twoheadrightarrow \overline{\mathfrak{C}^{i_n}_n}\subseteq {X}:=\sqcup_{m, i_m\in I_m}\mathfrak{C}^{i_m}_m\ees
for each toral cell \(\mathfrak{C}^{i_n}_n.\) 
Furthermore,
\be\label{Pg}
\Phi^{i_n}_n\big(\partial({\mathbb{T}_{n+1}\times \mathbb{B}^n})\big)\subseteq \sqcup_{k< n, i_k\in I_k} \mathfrak{C}^{i_k}_{k}\;\hbox{ and }\; g^{i_n}_n(\Phi^{i_n}_n(x, y))= (x, {\tilde{\Phi}_{n, {i_n}}}(y))
\ee for \((x, y)\in \mathbb{T}_{n+1}\times\mathbb{B}^{n}.\)
\end{itemize} We say that {\it toral CW complex} \(X\) is associated with the regular CW complex \(\tilde{X}\).
\end{defn}

\begin{rem}[An alternative description via toral fibers] Let \(\tilde{X}\) be a CW complex and \(P: X \rightarrow \tilde{X}\) be a continuous surjective map with the homotopy lifting property with respect to the closed interval \([0, 1]\). The toral CW complexes appearing in this paper admit such function \(P.\) The fiber of \(P\) over a point \(b\in \tilde{X}\) is \(F_b = P^{-1}(b)\). We here assume that fibers over all open \(k\)-cells is homeomorphic to a hypertorus of either \(k+1\) or \(k+2\)-dimension.
\end{rem}
\begin{lem}\label{TS1} There exists a toral CW complex associated with the space \(\overline{\mathbb{S}^{k-1, \sigma}_{>0}}\) and its CW-cell decomposition described by equations \eqref{ParI} and \eqref{DecParI} for \(l=k\) and \(\sigma=\gamma\). Elements of the partition associated with this toral CW complex are homomorphic to those in
\be\label{Partition}\left\{ \mathbb{T}_{l+1}\times \mathbb{S}^{l, {\gamma}}_{>0} \,|\, {\gamma}\in S^{l+1, \sigma}_n, 0\leq l\leq k-1\right\}.\ee
Here, there is no even dimensional toral cell.
\end{lem}
\bpr The main idea of the proof is to construct a toral CW complex via a quotient space \(Y/\sim\) of the space \(Y= \mathbb{T}_{k}\times \overline{\mathbb{S}^{k-1, \sigma}_{>0}}\) over an equivalence relation \(\sim,\) where \(\sim\) is generated by identifying the extra dimensions of the hypertori corresponding with the lower dimensional open cells in the boundary set \(\partial \overline{\mathbb{S}^{k-1, \sigma}_{>0}}\). By Lemma \ref{Splus}, for \(l\leq k\) and  \(\gamma\in S^{l, \sigma}_n\), we have
\bes\overline{\mathbb{S}^{k-1, \sigma}_{>0}}= \mathbb{S}^{k-1, \sigma}_{>0}\sqcup \sqcup^{k-1}_{i=1} \partial_i\mathbb{S}^{k-1,\sigma}_{>0}, \ees where
\bes\partial_i\mathbb{S}^{k-1,\sigma}_{>0}=\sqcup_{\gamma\in S^{k-i, \sigma}_{n}}\mathbb{S}^{k-i-1,\gamma}_{>0}\quad\hbox{ and }\quad \partial \mathbb{S}^{k-1, \gamma}_{>0}=\sqcup_{i=1}^{k-1}\partial_i\mathbb{S}^{k-1,\gamma}_{>0}.\ees We consider the Tychonoff product space
\bas
&Y:=\left(\prod_{j=1}^{k}\mathbb{S}_{\sigma(j)}^1\right)\times \overline{\mathbb{S}^{k-1, \sigma}_{>0}} \qquad\hbox{ for }\; \mathbb{S}_{\sigma(j)}^1= \mathbb{S}^1.&
\eas Given the \(\mathbb{S}^1\)-components collapsing criteria for the converging sequences to lower dimensional cells, the equivalence relation \(\sim\) is generated by identifying elements of
\bes
\widehat{Y}^{(s_{\gamma(j)})^{k-i}_{j=1}}_{x, i, \gamma} :=\left\{ \big((s_{\sigma(j)})^k_{j=1}, x\big)\in Y \,|\, s_{\gamma(l)}\in \mathbb{S}_{\gamma(l)}^1 \; \hbox{ for } l>k-i \right\}
\ees as an equivalent class for \(x\in\mathbb{S}^{k-i-1,\gamma}_{>0},\) \(1\leq i<k,\) and \(\gamma\in S^{k-i, \sigma}_{n}.\) This gives rise to the quotient space \(Y/\sim.\) Thus, the open cell \(\mathbb{S}^{k-i-1,\gamma}_{>0}\) in \(\partial \mathbb{S}^{k-1, \gamma}_{>0}\) corresponds with a space in the quotient space \(Y/\sim\) that is homeomorphic to \(\mathbb{T}_{k-i}\times\mathbb{S}^{k-i-1,\gamma}_{>0}.\) This introduces the homeomorphism \(g^{i_{2k-2i-1}}_{2k-2i-1}\) for \(i_{2k-2i-1}:= \gamma \in S^{k-i, \sigma}_{n}.\) The quotient space \(Y/\sim\) is the desired toral CW complex associated with CW-cell decomposition described by equations \eqref{ParI} and \eqref{DecParI}. The homeomorphism \({\tilde{\Phi}_{k, \gamma}}\) given in equation \eqref{Philgamma} induces the corresponding topology from \(Y/\sim\) onto \(\overline{\mathbb{T}_{k}\times \mathbb{B}^{k-1}}\) and the homeomorphism \(\Phi^{\gamma}_{2k-2i-1}\). The equivalence relation \(\sim\) is designed such that the homeomorphisms \(g^{\gamma}_{2k-2i-1}\) and \(\Phi^{\gamma}_{2k-2i-1}\) follow equations \eqref{Pg}.
\epr

\begin{thm}[A toral CW complex bifurcation associated with \(\overline{ \mathbb{S}^{k-1, \sigma}_{>0}}\)]\label{s1SingleHyp}
Consider \(k>1,\) \(\sigma\in S^k_n,\) \(s=1,\) \(\sign(a_{\mathbf{e}_{\sigma(j)}}a_{\mathbf{e}_{\sigma(i)}})=1\) for all \(1\leq i<j\leq k,\) and the closure of an open \(2k\)-cell \(\overline{\mathcal{M}_{k, \sigma}}.\) Then, there is a cell-bifurcation variety at
\be\label{Hopfs1} T_{Pch}:= \{\nu_0| \nu_0=0\}\ee
for the one-parametric Eulerian flow associated with (See the normal form in \cite[Theorem 4.7]{GazorShoghiEulNF})
\begin{eqnarray}\label{EulSysParGens1}
&\Theta+v:=\Theta+\sum_{i=1}^{n}\left(\nu_0+\sum_{j=1}^{n}a_{\mathbf{e}_{j}}({x_j}^2+{y_j}^2)\right)\big(\frac{x_{i}\partial}{\partial x_{i}}+\frac{{y_{i}}\partial}{\partial y_{i}}\big). &
\end{eqnarray} Here, a flow-invariant toral CW complex associated with the CW complex \(\overline{\mathbb{S}^{k-1, \sigma}_{>0}}= \sqcup^{k-1}_{l=0}\sqcup_{{\gamma}\in S^{l+1, \sigma}_n} \mathbb{S}^{l, {\gamma}}_{>0}\) bifurcates from the origin corresponding with the dynamics on \(\overline{\mathcal{M}_{k, \sigma}}.\) This toral CW complex and its partition are homeomorphic to the one given in Lemma \ref{TS1} and the partition \eqref{Partition}. This flow-invariant toral manifold exists when \(a_{\mathbf{e}_{\sigma(k)}}\nu_0<0.\) This is asymptotically stable when \(a_{\mathbf{e}_{\sigma(k)}}<0\) and otherwise, it is unstable. A trajectory associated with \(\rho_{\sigma(i)},\) for \(1\leq i\leq k,\) converges to/diverges from the stable/unstable toral CW complex at frequency \(\frac{\omega_{\sigma(i)}}{2\pi}\)[hz] and radial velocity \(\frac{c_i}{c_k}\rho_{\sigma(k)}\left(\nu_0+\rho_{\sigma(k)}\sum_{j=1}^{n} a_{\mathbf{e}_{j}} \frac{{c_j}^2}{{c_k}^2}\right)\)[m/s], where \(C\) is determined by the initial conditions and \({\rho_{\sigma(k)}}^2= {x_j}^2+{y_j}^2\).
\end{thm}
\bpr The dynamics of \eqref{EulSysParGens1} on \(\overline{\mathcal{M}_{k, \sigma}}\) follows the governing dynamics on open \(2l\)-cells \(\mathcal{M}_{l, \gamma}\) for \(\gamma\in S^{l, \sigma}_n\) and \(l\leq k.\) The hypertorus exists when \(a^{\sigma(l)}_1(C)\nu_0<0\) and it is stable only when \(a^{\sigma(l)}_1(C)<0\). Hence, a hypertoral flow-invariant manifold, say \(X,\) inside the invariant space \(\overline{\mathcal{M}_{k, \sigma}}\) (the closure of a \(2k\)-cell) bifurcates from the origin when the parameter \(\nu_0\) crosses the variety \(T_{Pch}\) given by \eqref{Hopfs1}.

Now we introduce a toral CW decomposition associated with \(\overline{ \mathbb{S}^{k-1, \sigma}_{>0}}\) for the flow-invariant manifold \(X.\) For \(\gamma\in S^{l, \sigma}_n\) and by equation \eqref{Philgamma}, the attaching map associated with a CW complex \(\mathbb{S}^{l-1, \gamma}_{>0}\) in \(\overline{\mathbb{S}^{k-1, \sigma}_{>0}}\) is the homeomorphism
\bes
{\tilde{\Phi}_{l, \gamma}}: \overline{\mathbb{B}^{l-1}}\twoheadrightarrow \overline{\mathbb{S}^{l-1, \gamma}_{>0}}= \sqcup^{l-1}_{j=0}\sqcup_{\bar{\gamma}\in S^{j+1, \gamma}_n} \mathbb{S}^{j, \bar{\gamma}}_{>0}\subseteq \sqcup^{k-1}_{j=0}\sqcup_{{\gamma}\in S^{j+1, \sigma}_n} \mathbb{S}^{j, {\gamma}}_{>0} \qquad \hbox{ for } 1\leq l\leq k-1.
\ees Since the index set \(I_{2l-2j+1}:= S^{l-j+1, \gamma}_n,\) we may replace the index
\(i_{2l-2j+1}\in I_{2l-2j+1}\) with \(\bar{\gamma}\in S^{l-j+1, \gamma}_n.\) The flow-invariant toral manifold \(X\) is determined by its sectional hypertorus in each \(\MKC\)-leaf. For each open toral cell
\bes
\mathfrak{C}^{\gamma}_{2l-1}:= \sqcup_{C\in\mathbb{S}^{l-1, \gamma}_{>0}}\mathbb{T}^{C, \gamma}_{l}\times \{C\} \quad\hbox{ for }\gamma\in I_{2l-1}= S^{l, \sigma}_n,
\ees
the hypertorus \(\mathbb{T}^{C, \gamma}_{l}\) is determined by the radius vector \(\rho_C=( \rho_1, \ldots \rho_n),\) where
\bes
{\rho_{\gamma(q)}}= \sqrt{\frac{-\nu_0}{a^{\sigma(q)}_1(C)}} \;\hbox{ for } q=1, \ldots l, \qquad\hbox{ and }\qquad {\rho_{\gamma(q)}}=0 \; \hbox{ for } q>l.
\ees
Hence, for any \(\gamma\in S^{l, \sigma}_n\) we have
\begin{eqnarray}\nonumber
&\lim_{c_{\gamma(q)}\to 0}\rho_{\gamma(q)}=0 \quad \hbox{ and } &\\&
\lim_{c_{\gamma(q)}\to 0, q=l+1, \ldots, k} {\rho_{\sigma(j)}}^2=\frac{-\nu_0{c_{\sigma(j)}}^2}{\sum_{i=1, i\notin \{\gamma(q)\}^k_{q=l+1}}^{k}{a_{\mathbf{e}_{\sigma(i)}}{c_{\sigma(i)}}^2}}\neq 0\; \hbox{ for } \sigma(j)\in \{\gamma(q)\}^l_{q=1}.&
\end{eqnarray} Thus, the squared radiuses corresponding with bifurcated \(k\)-hypertori within \(\MKC\) converge to \(\frac{-\nu_0{c_{\sigma(j)}}^2}{\sum_{i=1, i\neq \{\gamma(q)\}^k_{q=l+1}}^{k}{a_{\mathbf{e}_{\sigma(i)}}{c_{\sigma(i)}}^2}}\) when \(C\) approaches \(\mathbb{S}^{l-1, \gamma}_{>0}\subset \partial_{k-l}\mathbb{S}^{k-1, \sigma}_{>0}.\)

On the other hand,
\bes
\partial^{o, \gamma}_{j-1} \overline{\mathbb{B}^{l-1}}:= \sqcup_{\bar{\gamma}\in S^{l-j+1, \gamma}_n} {\tilde{\Phi}_{l, \gamma}}^{-1}\big(\mathbb{S}^{l-j, \bar{\gamma}}_{>0}\big)\; \hbox{ and }\; \partial^{e, i_l}_j \overline{\mathbb{B}^{l}}=\emptyset\quad \hbox{ for } 1\leq j\leq l \hbox{ and } 1\leq l\leq k.
\ees
The \(j\)-th nonzero squared radius of the bifurcated hypertorus for the \(\MKC\)-leaf normal form is \({\rho_{\sigma(j)}}^2:=\frac{-\nu_0}{a^{\sigma(j)}_1(C)}.\) Hence, we introduce \(\sqcup_{x\in {\tilde{\Phi}_{l, \gamma}}^{-1}\big(\mathbb{S}^{l-j, \bar{\gamma}}_{>0}\big)}\mathbb{T}^{x, \gamma, \bar{\gamma}}_{l-j+1}\times \{x\}\) and for notation simplicity denote it by \(\mathbb{T}^{\gamma, \bar{\gamma}}_{l-j+1}\times{\tilde{\Phi}_{l, \gamma}}^{-1}\big(\mathbb{S}^{l-j, \bar{\gamma}}_{>0}\big).\) Similarly, \(\mathbb{T}^\gamma_{l-j+1}\times \partial^{o, \gamma}_{j-1} \mathbb{B}^{l-1}\) stands for
\(\sqcup_{x\in \partial^{o, \gamma}_{j-1} \mathbb{B}^{l-1}}\mathbb{T}^{x, \gamma}_{l-j+1}\times \{x\}\) and \(\mathbb{T}^{x, \gamma}_{l-j+1}\) is a \(l-j+1\)-dimensional torus. Here,
\be\label{TPhi}
\mathbb{T}^{x, \gamma, \bar{\gamma}}_{l-j+1}\times \{x\}:=
\Big\{(\rho_x, \theta, x)\big|\, 0<\theta_{\gamma(i)}<2\pi, {\rho_{\bar{\gamma}(i)}}^2:=\frac{-\nu_0}{a^{\bar{\gamma}(i)}_1({\tilde{\Phi}_{l, \bar{\gamma}}}(x))} \hbox{ for } 0\leq i\leq l-j+1\Big\},
\ee where 
\(x\in \partial^{o, \gamma}_{j-1} \overline{\mathbb{B}^{l-1}},\) \((\rho_x, \theta):= (\rho_1, \rho_2, \ldots, \rho_n, \theta_1, \theta_2, \ldots, \theta_n)\) represents an action-angle coordinate system for the \(l-j+1\)-hypertorus \(\mathbb{T}^{x, \gamma, \bar{\gamma}}_{l-j+1},\) and \(\rho_{\bar{\gamma}(i)}=0\) for \(i> l-j+1.\) Thus, \(\mathbb{T}^{\gamma, \bar{\gamma}}_{l-j+1}\times{\tilde{\Phi}_{l, \gamma}}^{-1}\big(\mathbb{S}^{l-j, \bar{\gamma}}_{>0}\big)\) is homeomorphic to \(\mathbb{T}_{l-j+1}\times \mathbb{B}^{l-j}\) for any \(\gamma\in S^{l, \sigma}_n,\) \(\bar{\gamma}\in S^{l-j+1, \gamma}_n\) and \(l\leq k.\) Now we claim that
\be\label{FormalNot}
\overline{\mathbb{T}^\gamma_{l-j+1}\times \partial^{o, \gamma}_{j-1} \mathbb{B}^{l-1}}:=\sqcup_{\bar{\gamma}\in S^{l-j+1, \gamma}_n} \mathbb{T}^{\gamma, \bar{\gamma}}_{l-j+1}\times{\tilde{\Phi}_{l, \gamma}}^{-1}\big(\mathbb{S}^{l-j, \bar{\gamma}}_{>0}\big).
\ee Therefore,
\be
\overline{\mathbb{T}^\gamma_{l}\times \mathbb{B}^{l-1, \gamma}}=\sqcup^{l}_{j=1}\big(\mathbb{T}^\gamma_{l-j+1}\times\partial^{o, \gamma}_{j-1} \overline{\mathbb{B}^{l-1}}\big)= \sqcup^{l}_{j=1}\sqcup_{\bar{\gamma}\in S^{l-j+1, \gamma}_n} \mathbb{T}^{\gamma, \bar{\gamma}}_{l-j+1}\times{\tilde{\Phi}_{l, \gamma}}^{-1}\big(\mathbb{S}^{l-j, \bar{\gamma}}_{>0}\big).
\ee Furthermore, the space \(\overline{\mathbb{T}^\gamma_{l+1}\times \mathbb{B}^{l, \gamma}}\) is homeomorphic to the toral CW complex constructed in Lemma \ref{TS1}. Following Definition \ref{TorCWDef}, we define
\bes
\Phi^{\gamma}_{l+1}: \overline{\mathbb{T}^\gamma_{l}\times \mathbb{B}^{l-1, \gamma}}\twoheadrightarrow \overline{\mathfrak{C}^{\gamma}_{2l-1}}= \sqcup_{\bar{\gamma}\in S^{m, \gamma}_n, m\leq l}\mathfrak{C}^{\overline{\gamma}}_{2m+1} \; \hbox{ by } \Phi^{\gamma}_{l+1}(r, \theta, x):= (\rho_{{\tilde{\Phi}_{l+1, \gamma}}(x)}, \theta, {\tilde{\Phi}_{l+1, \gamma}}(x)).
\ees Here, \(\Phi^{\gamma}_{l+1}\) is a homeomorphism. This completes the proof.
\epr


\begin{thm}\label{TopEqu1} Let the hypotheses of Theorem \ref{s1SingleHyp} hold.
Then, the parametric \(2k\)-cell vector fields \(v_k(\nu_0^1)\) and \(v_k(\nu_0^2)\) for either \(\nu_0^1, \nu_0^2\in\{\nu_0| \nu_0>0\}\) or \(\nu_0^1, \nu_0^2\in\{\nu_0| \nu_0<0\}\) are orbitally equivalent. Hence, the variety given by \eqref{Hopfs1} is the only \(2k\)-cell bifurcation variety for the \(2k\)-cell truncated normal form system.
\end{thm}
\bpr Consider the following two differential equations
\begin{eqnarray}\label{TopEqivEqs.1}
&\frac{d}{dt}\mathbf{r}=\mathbf{r}\left(\nu^1_0+\sum_{i=1}^{k}a_{\mathbf{e}_{\sigma(i)}}{r_{{\sigma(i)}}}^2\right)\quad\hbox{ and }\quad
\frac{d}{dt}\mathbf{R}=\mathbf{R}\left(\nu_0^2+\sum_{i=1}^{k}a_{\mathbf{e}_{\sigma(i)}}{R_{{\sigma(i)}}}^2\right),&
\end{eqnarray} where \(\nu^1_0\nu_0^2>0,\) \(\mathbf{r}=(r_{\sigma(1)},\cdots, r_{\sigma(k)}),\) and \(\mathbf{R}=(R_{\sigma(1)}, \cdots, R_{\sigma(k)}).\) We show that these equations are orbitally equivalent. Consider the homeomorphism \(h:\mathbb{R}^k\rightarrow\mathbb{R}^k\) and the map \(\tau:\mathbb{R}^k\times\mathbb{R}\rightarrow\mathbb{R}\) defined by
\begin{eqnarray*}
&h(\mathbf{r})=(h_1(\mathbf{r}),\cdots,h_k(\mathbf{r})):=\sqrt{\frac{\nu^2_0}{\nu^1_0}}\mathbf{r}\; \hbox{ and }\;  \tau(\mathbf{r},t)=\frac{\nu^2_0}{\nu^1_0}t.&
\end{eqnarray*}
The flow associated with the first equation in \eqref{TopEqivEqs.1} follows
\begin{eqnarray*}
&\mathbf{r}(t,\mathbf{r}^0)=\left(\frac{-\nu^1_0{r_{{\sigma(k)}}^0}^2\exp(2\nu^1_0(t-t_0))}{\sum_{i=1}^{k}a_{\mathbf{e}_{{\sigma(i)}}}{r_{{\sigma(i)}}^0}^2
\left(\exp(2\nu^1_0(t-t_0))-1\right)-\nu^1_0}\right)
^\frac{1}{2}\frac{\mathbf{r}^0}{r_{{\sigma(k)}}^0} \quad \hbox{ for } \mathbf{r}(t_0, \mathbf{r}^0)= \mathbf{r}^0.&
\end{eqnarray*} Then, \(\mathbf{R}(t_0, h(\mathbf{r}^0))= h(\mathbf{r}^0)\) and
\begin{small}
\begin{eqnarray*}
h(\mathbf{r}(\tau(\mathbf{r}^0, t),\mathbf{r}^0))&=&\Big(\frac{\nu^2_0}{\nu^1_0}\Big)^{\frac{1}{2}}\mathbf{r}(\tau(\mathbf{r}^0, t),\mathbf{r}^0)
=\Big(\frac{-\nu^2_0{r_{{\sigma(k)}}^0}^2\exp(2\nu^2_0(t-t_0)}{\sum_{i=1}^{k}a_{\mathbf{e}_{{\sigma(i)}}}{r_{{\sigma(i)}}^0}^2\left(\exp(2\nu^2_0(t-t_0))-1\right)
-\nu^1_0}\Big)
^\frac{1}{2}\frac{\mathbf{r}^0}{r_{{\sigma(k)}}^0}\\
&=&\Big(\,\frac{-\nu^2_0h_k(\mathbf{r}^0)^2\exp(2\nu^2_0(t-t_0))}{\sum_{i=1}^{k}a_{\mathbf{e}_{{\sigma(i)}}}h_i(\mathbf{r}^0)^2
(\exp(2\nu^2_0(t-t_0))-1)-\nu^2_0}\,\Big)
^\frac{1}{2}\frac{h(\mathbf{r}^0)}{h_k(\mathbf{r}^0)}=\mathbf{R}(t, h(\mathbf{r}^0)).
\end{eqnarray*}\end{small} This completes the proof.
\epr

Now we consider a one-parameter 5-degree truncated \(2k\)-cell normal form (see \cite{GazorShoghiEulNF}) given by \(\Theta_k+v^{\sigma}_{k}\) for
\begin{eqnarray}\label{ClassicalNF2}
&v^{\sigma}_{k}:= \sum^{k}_{i=1}\left(\nu_0\big(\frac{x_{\sigma(i)}\partial}{\partial x_{\sigma(i)}}+\frac{y_{\sigma(i)}\partial}{\partial y_{\sigma(i)}}\big)+\sum^{2}_{|\m|=1} \prod_{j=1}^{k}({x_{\sigma(j)}}^2+{y_{\sigma(j)}}^2)^{m_{\sigma(j)}} \big(a_{\m}\frac{x_{\sigma(i)}\partial}{\partial x_{\sigma(i)}}+a_{\m}\frac{y_{\sigma(i)}\partial}{\partial y_{\sigma(i)}}\big)\right),\;&
\end{eqnarray} where \(a_{\0}=\nu_0, \m=(m_i)^n_{i=1}\in \mathbb{R}^n,\) \(m_{\sigma(i)}=0\) for \(i>k.\) Using a leaf-invariant \(\MKC\), we have
\begin{eqnarray}
\label{LeafNormalForm2}&\Theta_k+ v^{\sigma}_{k, C}:=\Theta_k+ \sum^2_{j= 0}\sum_{l=1}^{k}\big(a^{\sigma(l)}_j(C) \frac{\left({x_{\sigma(l)}}^2+{y_{\sigma(l)}}^2\right)^jx_{\sigma(l)}\partial}{\partial x_{\sigma(l)}}+a^{\sigma(l)}_j(C) \frac{\left({x_{{\sigma(l)}}}^2+{y_{{\sigma(l)}}}^2\right)^jy_{{\sigma(l)}}\partial}{\partial y_{{\sigma(l)}}}\big).&
\end{eqnarray} Here, \(c_{\sigma(j)}\neq 0\) for \(j\leq k\) and  \(c_{\sigma(j)}= 0\) for \(j>k.\) Then, for any \(l\leq k\) we have \(a^{\sigma(l)}_0(C)=\nu_0,\)
\ba\label{a12}
& a^{\sigma(l)}_1(C)= \frac{1}{{c_{\sigma(l)}}^2}\sum^k_{i=1} {c_{\sigma(i)}}^2a_{\mathbf{e}_{\sigma(i)}}, \; a^{\sigma(l)}_2(C)=\frac{1}{{c_{\sigma(l)}}^4}\sum_{1\leq i\leq j\leq k} {c_{\sigma(i)}}^2{c_{\sigma(j)}}^2a_{\mathbf{e}_{\sigma(i)}+\mathbf{e}_{\sigma(j)}}. &
\ea

Denote \(\diag(a_{\mathbf{e}_{\sigma(i)}})\) for a \(n\times n\) diagonal matrix where \(a_{\mathbf{e}_{\sigma(i)}}\) (\(1\leq i\leq k\)) is the \({\sigma(i)}\)-th diagonal entry and the rest of entries are zero. Further, denote
\bes L_{(a_{\mathbf{e}_{\sigma(i)}})^k_{i=1}}:=\big\{C\in \mathbb{R}^n\,|\, \big\langle\diag(a_{\mathbf{e}_{\sigma(i)}})C, C\big\rangle=0 \big\}\ees for a quadric \(k\)-hypersurface passing through the origin. For a \(\gamma\in S^{l, \sigma}_n\), let
\ba\label{Gammas}
&\Gamma^{l, \gamma}_{a_1}:=L_{(a_{\mathbf{e}_{\sigma(i)}})^k_{i=1}}\cap \mathbb{S}^{l-1, \gamma}_{>0} \hbox{ and } \Gamma^{l, \gamma, \pm}_{a_1}\!:= \! \left\{C\in \mathbb{S}_{>0}^{l-1, \gamma}| \sign\!\left(a_{2\mathbf{e}_{\sigma(1)}}\big\langle\diag(a_{\mathbf{e}_{\sigma(i)}})C, \; C\big\rangle\right)=\pm1\right\}.&
\ea Hence, \(\Gamma^{l, \gamma, -}_{a_1}:= \mathbb{S}_{>0}^{l-1, \gamma}\setminus (\Gamma^{l, \gamma}_{a_1}\sqcup\Gamma^{l, \gamma, +}_{a_1}).\)

\begin{lem}[CW complex structures for \(\overline{\Gamma^{k, \sigma, \pm}_{a_1}}\)]\label{TopLem} Assume that  \(k\geq 2,\) \(\sigma\in S^k_n,\) and for at least a pair of indices \((i, j),\) \(1\leq i<j\leq k,\)
\be\label{s1DoblTor}\sign(a_{\mathbf{e}_{\sigma(j)}}a_{\mathbf{e}_{\sigma(i)}})=-1, \quad \hbox{ while }\quad \sign(a_{\mathbf{e}_{\sigma(i_1)}+\mathbf{e}_{\sigma(i_2)}})\sign(a_{\mathbf{e}_{\sigma(j_1)}+\mathbf{e}_{\sigma(j_2)}})=1\ee
for all \(1\leq i_1\leq i_2\leq k\) and \(1\leq j_1\leq j_2\leq k.\) Then,
\(a^{\sigma(i)}_{1}(C_*)=0\) for any \(C_*\in \Gamma^{k, \sigma}_{a_1}=L_{(a_{\mathbf{e}_{\sigma(i)}})^k_{i=1}}\cap \mathbb{S}_{>0}^{k-1}\) and \(i\leq k.\) Besides, the \(C\)-parameter space \(\mathbb{S}_{>0}^{k-1}\) is partitioned into a union of disjoint three topological subspaces \(\Gamma^{k, \sigma, -}_{a_1},\) \(\Gamma^{k, \sigma, +}_{a_1},\) and \(\Gamma^{k, \sigma}_{a_1}\) given by \eqref{Gammas}. Each of the topological closures of \(\Gamma^{k, \sigma, +}_{a_1}\) and \(\Gamma^{k, \sigma, -}_{a_1}\) constitutes a CW complex.
\end{lem}
\bpr
There exists a unique natural number \(l<k\) and a \(\gamma_+\in S^{l, \sigma}_n\) so that \(a_{\mathbf{e}_{\sigma({\gamma_+(i)})}}a_{2\mathbf{e}_{\sigma(i_1)}}>0\) for all \(1\leq i\leq l\) and \(a_{\mathbf{e}_{\sigma(j)}}a_{2\mathbf{e}_{\sigma(i_1)}}<0\) for all \(j\notin \{\gamma_+(i)| 1\leq i\leq l\}.\) Similarly, there is a unique \(\gamma_-\in S^{k-l, \sigma}_n\) so that \(\{1, 2, \ldots, k\}\setminus\{\gamma_+(i)| 1\leq i\leq l\}= \{\gamma_-(i)| 1\leq i\leq k-l\}.\) Hence, the CW decomposition of \(\overline{\Gamma^{k, \sigma, +}_{a_1}}\) is given by the disjoint sets appearing in the equation
\be
\overline{\Gamma^{k, \sigma, +}_{a_1}}= (\sqcup^{l}_{j=1, \bar{\gamma}\in S^{j, \gamma_+}_n} \mathbb{S}^{j-1, \bar{\gamma}}_{>0})\sqcup( \sqcup^{k}_{j=1,\bar{\gamma}\in S^{j, \sigma}_n\setminus (S^{j, \gamma_+}_n\cup S^{j, \gamma_-}_n)}  (\Gamma^{j, \bar{\gamma}, +}_{a_1}\sqcup \Gamma^{j, \bar{\gamma}}_{a_1})).
\ee Note that \(S^{j, \gamma_+}_n=\emptyset\) for \(j>l\) while \(S^{j, \gamma_-}_n=\emptyset\) for \(j>k-l.\) Similarly, the CW-decompositions for \(\overline{\Gamma^{k, \sigma, -}_{a_1}}\) and \(\overline{\Gamma^{k, \sigma}_{a_1}}\) are derived by the disjoint subsets appearing in
\ba\nonumber
&\overline{\Gamma^{k, \sigma, -}_{a_1}}= \sqcup^{l}_{j=1, \bar{\gamma}\in S^{j, \gamma_-}_n} \mathbb{S}^{j-1, \bar{\gamma}}_{>0}\;\sqcup\sqcup^{k}_{j=1, \bar{\gamma}\in S^{j, \sigma}_n\setminus (S^{j, \gamma_+}_n\cup S^{j, \gamma_-}_n)}  (\Gamma^{j, \bar{\gamma}, -}_{a_1}\sqcup \Gamma^{j, \bar{\gamma}}_{a_1})&\\\label{gamma-}& \hbox{ and }\qquad \overline{\Gamma^{k, \sigma}_{a_1}}= \sqcup^{k}_{j=1, \gamma\in S^{j, \sigma}_n}  \Gamma^{j, \gamma}_{a_1}.&
\ea Remark that the spaces \(\Gamma^{k, \sigma, +}_{a_1}\) and \(\Gamma^{k, \sigma, -}_{a_1}\) are relatively compact connected open subsets of \(\mathbb{S}_{>0}^{k-1}.\) Further, the spaces \(\Gamma^{k, \sigma, +}_{a_1}\) and \(\Gamma^{k, \sigma, -}_{a_1}\) are homeomorphic to \(\mathbb{B}^{k-1}\) while \(\Gamma^{k, \sigma}_{a_1}\) is homeomorphic to \(\mathbb{B}^{k-2}.\) We need to introduce the attaching maps to complete the proof. The attaching map \({\tilde{\Phi}_{j, \bar{\gamma}}}\) given by \eqref{Philgamma} works fine in the cases of \(j-1\)-CW cells \(\mathbb{S}^{j-1, \bar{\gamma}}_{>0}\) for \(\bar{\gamma}\in S^{j, \gamma{\pm}}_n\) and \(1\leq j\leq l.\) Thus, we only refine the attaching maps to work for \(\Gamma^{j, \bar{\gamma}, -}_{a_1}\)-cells. The other cases are similar. We first introduce a homeomorphism \(h^{j, \gamma, -}_{a_1}: \overline{\mathbb{S}^{j-1, \gamma}_{>0}}\rightarrow \Gamma^{j, \gamma, -}_{a_1}.\) The idea is to choose a point, say \(P\), from the interior of \(\Gamma^{j, \gamma, -}_{a_1}\subset \mathbb{S}^{j-1, \gamma}_{>0}.\) Consider all two dimensional planes passing through the origin and \(P.\) The intersections of each of these planes with \(\mathbb{S}^{j-1, \gamma}_{>0}\) and \(\Gamma^{j, \gamma, -}_{a_1}\) give rise to two open arcs (an arc here refers to a one-manifold). The point \(P\) divides each of these two arcs into two connected arc-pieces and the homeomorphism \(h^{j, \gamma, -}_{a_1}\) is defined as identity on one piece while it compresses the other piece in \(\mathbb{S}^{j-1, \gamma}_{>0}\) to homeomorphically match it with the corresponding arc-piece in \(\Gamma^{j, \gamma, -}_{a_1}\). The homeomorphism \(h^{j, \gamma, -}_{a_1}\) is readily defined as a uniformly continuous map on \({\mathbb{S}^{j-1, \gamma}_{>0}}\). Thus, it can also be uniquely extended to \(h^{j, \gamma, -}_{a_1}: \overline{\mathbb{S}^{j-1, \gamma}_{>0}} \rightarrow\overline{\Gamma^{j, \gamma, -}_{a_1}}.\) Using this map and the attaching map \({\tilde{\Phi}_{j, \bar{\gamma}}}\) from \eqref{Philgamma}, we introduce an attaching homeomorphism for the space decomposition \eqref{gamma-} by
\be\label{AtachGamma-}
{\widetilde{\Phi_{j, \gamma}}}: \overline{\mathbb{B}^{j-1}}\rightarrow \overline{\Gamma^{j, \gamma, -}_{a_1}} \quad \hbox{ where }\; {\widetilde{\Phi_{j, \gamma}}}:= {h^{j, \gamma, -}_{a_1}}\circ {\tilde{\Phi}_{j, \bar{\gamma}}}.
\ee The CW-decomposition \eqref{gamma-} and the attaching map \eqref{AtachGamma-} provide a CW complex structure for \(\overline{\Gamma^{j, \gamma, -}_{a_1}}.\) Hence, the proof is complete.
\epr

\begin{thm}[Cell-bifurcation of a toral CW complex associated with the CW complex \(\overline{\Gamma^{k, \sigma, +}_{a_1}}\)]\label{LemS1Gamma+} Assume that the hypotheses described by \eqref{s1DoblTor} hold and \(\sigma\in S^k_n.\) Let \(\mathcal{M}_{k, \sigma}\) be an open cell and \(\overline{\mathcal{M}_{k, \sigma}}\) as its closure. Consider a one-parametric (normal form) vector field \(\Theta+v(\mathbf{r}, \theta, \nu_0)\) given by
\begin{eqnarray}\label{CellS1DeGen}
&\Theta\!+\!\sum_{i=1}^{n}\left(\nu_0+\sum_{j=1}^{n}a_{\mathbf{e}_{j}}({x_{j}}^2+{y_{j}}^2) +\sum_{1\leq j\leq l\leq n}a_{\mathbf{e}_{j}+\mathbf{e}_{l}}({x_{j}}^2+{y_{j}}^2)({x_{l}}^2+{y_{l}}^2)\right)(\frac{x_{i}\partial}{\partial {x_{i}}}+\frac{y_{i}\partial}{\partial {y_{i}}}).\; &
\end{eqnarray} Then, there is a primary cell-bifurcation variety given by
\be\label{T2Pch}
T_{2Pch}:= \{\nu_0| \nu_0=0\}.
\ee
\begin{enumerate}
\item When \(\nu_0a_{2\mathbf{e}_{\sigma(k)}}<0.\) A secondary flow-invariant toral CW complex associated with the CW complex \(\overline{\Gamma^{k,
\sigma, +}_{a_1}}\) bifurcates from the origin exactly when \(\nu_0a_{2\mathbf{e}_{{\sigma(k)}}}<0\). There exist a natural number \(l<k,\) \(\gamma_+\in S^{l, \sigma}_n\) and a \(\gamma_-\in S^{k-l, \sigma}_n\) so that its toral CW decomposition is homeomorphic to
\begin{small}\ba
&\{{\mathbb{T}_j}\times\mathbb{S}^{j-1, \bar{\gamma}}_{>0}| \bar{\gamma}\in S^{j, \gamma_+}_n\}^l_{j=1}\sqcup \{(\mathbb{T}_j\times\Gamma^{j, \bar{\gamma}, +}_{a_1})\sqcup (\mathbb{T}_j\times\Gamma^{j, \bar{\gamma}}_{a_1})| \bar{\gamma}\in S^{j, \sigma}_n\setminus (S^{j, \gamma_+}_n\cup S^{j, \gamma_-}_n)\}^k_{j=1}.&
\ea\end{small} This manifold is unbounded when \(|\nu_0|\) approaches to infinity and \(\nu_0a_{2\mathbf{e}_{{\sigma(k)}}}<0.\) This toral manifold, when it exists, is asymptotically stable for \(a_{2\mathbf{e}_{{\sigma(k)}}}<0\) and is unstable otherwise.
\item For \(\nu_0a_{2\mathbf{e}_{\sigma(k)}}\geq 0,\) there is no flow-invariant hypertorus for the system \eqref{S1DeGen} when \(C\in \overline{\Gamma^{k, \sigma, +}_{a_1}}.\)
\end{enumerate}
\end{thm}
\bpr The possible radiuses of flow-invariant tori are given by
\begin{small}\ba\label{rbar}
&{r_{\sigma(i)}}^2=\frac{-a^{\sigma(i)}_1(C)\pm \sqrt{{a^{\sigma(i)}_1(C)}^2-4\nu_0a^{\sigma(i)}_2(C)}}{2a^{\sigma(i)}_2(C)}, \hbox{ and }  {a^{\sigma(i)}_1(C)}^2-4\nu_0a^{\sigma(i)}_2(C)>0 \;\hbox{ when }\; \nu_0a^{\sigma(k)}_2(C)<0.\quad&
\ea\end{small} The assumption \eqref{s1DoblTor} implies that either \(a^{\sigma(i)}_2(C)>0\) for all \(C\) and \(i\leq k,\) or \(a^{\sigma(i)}_2(C)<0\) for all \(C\) and \(i\leq k.\) Further, \(\frac{{a^{\sigma(i)}_1(C)}^2}{4a^{\sigma(i)}_2(C)}= \frac{{a^{\sigma(j)}_1(C)}^2}{4a^{\sigma(j)}_2(C)}\) for all \(i\leq j\leq k.\) Thereby, there is always precisely one hypertorus on each leaf \(\MKC\) for \(C\in \Gamma^{k+}_{a_1}\cup\Gamma^{k}_{a_1}\) as long as \(\nu_0a^{\sigma(k)}_2(C)<0.\) A secondary toral CW complex parameterized by \(C\in\overline{\Gamma^{k, \sigma+}_{a_1}\sqcup\Gamma^{k, \sigma}_{a_1}}\) bifurcates from the origin at \(T_{2Pch}\). This secondary manifold exists when \(\nu_0a^{\sigma(k)}_2(C)<0\) and vanishes for \(\nu_0a^{\sigma(k)}_2(C)>0.\) Hence, the flow-invariant toral manifold bifurcates from the origin via a simultaneous \(\MKC\)-leaf bifurcation of hypertori at the variety \(T_{2Pch}\) for all \(C\in\overline{\Gamma^{k, \sigma, +}_{a_1}}\). The attaching maps for toral cells indexed with \(\bar{\gamma}\in S^{j, \gamma_+}_n\) is similar to the cases in the proof of Theorem \ref{s1SingleHyp} and we skip them here. We instead assume that \(\nu_0a^{\sigma(k)}_2(C)<0\) and \(\bar{\gamma}\in S^{j, \sigma}_n\setminus (S^{j, \gamma_+}_n\cup S^{j, \gamma_-}_n).\) Then, we introduce
\bes
\mathbb{T}^{\sigma, \bar{\gamma}}_j\times{\widetilde{\Phi}_{k, \sigma}}^{-1}(\Gamma^{j, \bar{\gamma}, +}_{a_1}\sqcup \Gamma^{j, \bar{\gamma}}_{a_1}) := \left\{(r_x, \theta, x)\, |\, 0<\theta_{\gamma(i)}<2\pi, \bar{r}_{\gamma(i)}:=0 \hbox{ for } i> j\right\},
\ees where \(x\in \partial^{o, \gamma}_j \overline{\mathbb{B}^{l}},\) \(r_x:= (\bar{r}_1, \bar{r}_2, \ldots, \bar{r}_n)\) and \(\bar{r}_{\gamma(i)}\) follows \eqref{rbar} for \(i\leq j\) and \(C= {\widetilde{\Phi}_{k, \sigma}}(x)\).

A sequence
\bes
((r_1^{x_p}, \theta_1^p), \ldots, (r_{l}^{x_p}, \theta_{l}^p), x_p)^\infty_{p=1}\subset \mathbb{T}^{\sigma, {\gamma}}_l \times {\widetilde{\Phi}_{k, \sigma}}^{-1}(\Gamma^{l, {\gamma}, +}_{a_1}\sqcup \Gamma^{l, {\gamma}}_{a_1}),
\ees
approaches
\bas & ((R_1^y, \vartheta_1), \ldots, (R_{j}^y, \vartheta_{j}), y)\in \mathbb{T}^{\sigma, \bar{\gamma}}_j \times {\widetilde{\Phi}_{k, \sigma}}^{-1}(\Gamma^{j, \bar{\gamma}, +}_{a_1}\sqcup \Gamma^{j, \bar{\gamma}}_{a_1}), \quad \hbox{for } \; j<l,  &\eas when \(x_p\) converges to \(y\) and \(j\)-number of angles from the sequence \((\theta_1^p, \ldots, \theta_{l}^p)\) correspond with and converge to \((\vartheta_1, \ldots, \vartheta_{j}).\) Further, the radiuses corresponding with the same \(j\)-number of indices from the sequence of \((r_1^{x_p}, \ldots, r_{l}^{x_p})\) converges to those in \((R_1^y, \ldots, R_{j}^y).\) However, the sequence of radiuses corresponding with the \(l-j\)-remaining indices either converges to zero. Hence, we have
\be
\overline{\mathbb{T}^\sigma_{k}\times \Gamma^{k, \sigma, +}_{a_1}}:= \sqcup^{k}_{j=1,\bar{\gamma}\in S^{j, \sigma}_n\setminus (S^{j, \gamma_+}_n\cup S^{j, \gamma_-}_n)}\mathbb{T}^{\sigma, \bar{\gamma}}_j\times{\widetilde{\Phi}_{k, \sigma}}^{-1}(\Gamma^{j, \bar{\gamma}, +}_{a_1}\sqcup \Gamma^{j, \bar{\gamma}}_{a_1})
\sqcup\sqcup^{l}_{j=1, \bar{\gamma}\in S^{j, \gamma_+}_n} \mathbb{T}^{\sigma, \bar{\gamma}}_j
\times{\widetilde{\Phi}_{k, \sigma}}^{-1}(\mathbb{S}^{j-1, \bar{\gamma}}_{>0}).
\ee
For \(C\in \Gamma^{k+}_{a_1}\cup\Gamma^{k}_{a_1},\) the vector radius of the hypertorus approaches the origin and then, the invariant hypertorus vanishes as \(\nu_0\) converges to and crosses the transit variety \(T_{2Pch}\). In other words, there is no invariant hypertorus on \(\MKC\)-leaves when \(\nu_0a^{\sigma(k)}_2(C)>0\) and \(a^{\sigma(k)}_1(C)a^{\sigma(k)}_2(C)\geq 0\). The radiuses of the hypertorus diverges to infinity, as \(\nu_0a^{\sigma(k)}_2(C)\) diverges to the negative infinity.
\epr

For any \(\gamma\in S^{l, \sigma}_n\) define
\ba\nonumber&D^{\gamma}_{\nu_0}:=\left\{C\in \Gamma^{l, \gamma, -}_{a_1}\,|\, 0<\nu_0< \frac{{a^{\gamma(l)}_1(C)}^2}{4a^{\gamma(l)}_2(C)} \hbox{ when }  a_{2\mathbf{e}_{\gamma(l)}}>0, \hbox{while } \frac{{a^{\gamma(l)}_1(C)}^2}{4a^{\gamma(l)}_2(C)}< \nu_0<0 \hbox{ for } a_{2\mathbf{e}_{\gamma(l)}}<0\right\},&\\\label{PartofGam}
&D^{\gamma, \partial}_{\nu_0}:=\left\{C\in \Gamma^{l, \gamma, -}_{a_1}\,|\, \nu_0=\frac{{a^{\gamma(l)}_1(C)}^2}{4a^{\gamma(l)}_2(C)}\right\}, \hbox{ and } N^{\gamma}_{\nu_0}:= \Gamma^{l, \gamma, -}_{a_1}\setminus (D^{\gamma}_{\nu_0}\sqcup D^{\gamma, \partial}_{\nu_0}).&
\ea

For an instance, we illustrate \(D^{\sigma}_{\nu_0}\) by assuming that \(\nu_0a^{\sigma(k)}_2(\mathbf{e}_{\sigma(1)})>0\) and \(a_{2\mathbf{e}_{\sigma(k)}}>0.\) (The case for the conditions \(\nu_0a^{\sigma(k)}_2(\mathbf{e}_{\sigma(1)})>0\) and \(a_{2\mathbf{e}_{\sigma(k)}}<0\) will be similar.) Then, there exist a unique \(0\leq l\leq k\) and a \(\widetilde{\gamma}\in S^{l, \sigma}_n\) such that
\be\label{gammatilde}
4\nu_0a^{\widetilde{\gamma}(l)}_2(\mathbf{e}_{\widetilde{\gamma}(i)})<{a^{\widetilde{\gamma}(l)}_1(\mathbf{e}_{\widetilde{\gamma}(i)})}^2 \hbox{ for all } i\leq l,\; \hbox{ and } 4\nu_0a^{\widetilde{\gamma}(l)}_2(\mathbf{e}_{\widetilde{\gamma}(i)})\geq {a^{\widetilde{\gamma}(l)}_1(\mathbf{e}_{\widetilde{\gamma}(i)})}^2
\hbox{ for } i> l.
\ee In this case, \(\mathbf{e}_{\widetilde{\sigma}(i)}\in D^{\sigma}_{\nu_0}\) for all \(i\leq l,\) and \(\mathbf{e}_{\widetilde{\gamma}(i)}\notin D^{\sigma}_{\nu_0}\) for all \(i> l.\) When \(l=0,\) \(D^{\sigma}_{\nu_0}=\emptyset.\) For \(l=k,\) \(D^{\sigma}_{\nu_0}=\Gamma^{k, \sigma, -}_{a_1}.\)
Let \(0< l<k.\) Since \(\Gamma^{k, \sigma, -}_{a_1},\) \({D^{\sigma}_{\nu_0}}\) and \(N^{\sigma}_{\nu_0}\) are three connected \(k-1\)-dimensional open manifolds and \(D^{\sigma, \partial}_{\nu_0}\) is a connected \(k-2\)-dimensional open manifold, the spaces \(\Gamma^{l, \sigma, -}_{a_1},\) \({D^{\sigma}_{\nu_0}}\) and \(N^{\sigma}_{\nu_0}\) are all homeomorphic to \(\mathbb{B}^{k-1}\), while \(D^{\sigma, \partial}_{\nu_0}\) is homeomorphic to \(\mathbb{B}^{k-2}.\) The closure of \(D^{\sigma}_{\nu_0}\) is a regular CW complex whose CW decomposition is given by the disjoint sets
\be
\overline{D^{\sigma}_{\nu_0}}:= (\sqcup^{l}_{j=1, \gamma\in S^{j, \widetilde{\gamma}}_n} \mathbb{S}^{j-1, \gamma}_{>0})\sqcup( \sqcup^{k}_{j=1,\gamma\in S^{j, \sigma}_n\setminus S^{j, \widetilde{\gamma}}_n}  (D^{\gamma}_{\nu_0}\sqcup D^{\gamma, \partial}_{\nu_0})).
\ee Here, \(\tilde{\gamma}\) follows the conditions \eqref{gammatilde}. The associated attaching maps is defined similar to what is given in the proof of Lemma \ref{TopLem}.

Consider the symmetric matrix
\begin{small}\ba\label{Agama}
&M_\gamma:=[M^1_\gamma, \ldots, M^{l}_\gamma], \; M_\gamma^j:=a_{2\mathbf{e}_{{\gamma(j)}}}\mathbf{e}_j+ \frac{1}{2}\sum_{i=1, i\neq j}^{l} a_{\mathbf{e}_{{\gamma(i)}}+\mathbf{e}_{{\gamma(l)}}}\mathbf{e}_i, \hbox{ and } \mathbf{a}_\gamma:=\left(a_{\gamma(1)}, \ldots, a_{\gamma(l) }\right)^t&
\ea\end{small}
for \(\gamma\in S^{l, \sigma}_n.\) Then by equation \eqref{a12}, \(a^{\gamma(l)}_2(C)= \frac{1}{{c_{\gamma(l)}}^4}\langle\mathcal{C}_\gamma, M_\gamma{\mathcal{C}_\gamma}\rangle\) for \(\mathcal{C}_\gamma:= \left({c_{\gamma(1)}}^2, \ldots ,{c_{\gamma(l)}}^2\right)^t.\)

\begin{thm}[Toral CW complex bifurcations associated with CW complex subspaces in \(\Gamma^{k, \sigma, -}_{a_1}\)]\label{Lem7.5}
Consider the closed cell \(\overline{\mathcal{M}_{k, \sigma}}\) and the vector field \eqref{CellS1DeGen} along with the assumptions in Lemma \ref{LemS1Gamma+}. Further, assume that \(M_\sigma\) given by \eqref{Agama} is either a positive definite or a negative definite matrix and
\ba\label{nuMinMax}
&\nu_{\min}:=\min\Big\{0,\frac{\langle{\mathbf{a}_\gamma}, {M_\gamma}^{-1}\mathbf{a}_\gamma\rangle}{4}| {\gamma\in S^{l, \sigma}_n, l\leq k}\Big\},\; \nu_{\max}:= \max\Big\{0,\frac{\langle{\mathbf{a}_\gamma}, {M_\gamma}^{-1}\mathbf{a}_\gamma\rangle}{4}| {\gamma\in S^{l, \sigma}_n, l\leq k}\Big\}.&
\ea
\begin{enumerate}
\item When \(\nu_0a_{2\mathbf{e}_{\sigma(k)}}>0\), there is a bifurcation variety given by
\ba\label{Tmin1}
&T_{SN}:= \{\nu_0| \nu_0=\nu_0^{\min} \;\hbox{ when }\; a_{2\mathbf{e}_k}<0, \,\hbox{ and }\, \nu_0=\nu_0^{\max}\; \hbox{ for } \;a_{2\mathbf{e}_k}>0\}. &
\ea Two flow-invariant toral CW complex manifolds \((\mathcal{T}^{int}_{\nu_0}\) and \(\mathcal{T}^{ext}_{\nu_0})\) associated with the topological closure of \(D^{\sigma}_{\nu_0}\subseteq\Gamma^{k, \sigma, -}_{a_1}\) simultaneously exists when \(\nu_0^{\min}<\nu_0<\nu_0^{\max}\) and \(\nu_0a_{2\mathbf{e}_{\sigma(k)}}>0.\) The hypertoral manifold \(\mathcal{T}^{int}_{\nu_0}\) lives inside \(\mathcal{T}^{ext}_{\nu_0}\). There is no flow-invariant hypertori corresponding with \(C\in N^{\sigma}_{\nu_0}\) for positive values of \(\nu_0a_{2\mathbf{e}_{\sigma(k)}}.\) The external toral CW complex \(\mathcal{T}^{int}_{\nu_0}\) is asymptotically stable when \( a_{2\mathbf{e}_k}<0\) while it is unstable for \( a_{2\mathbf{e}_k}>0\). The internal toral CW complex \(\mathcal{T}^{int}_{\nu_0}\) is asymptotically unstable/stable when \(\mathcal{T}^{ext}_{\nu_0}\) is asymptotically stable/unstable. As \(\nu_0\) approaches \(T_{SN}\) when \(\nu_0a_{2\mathbf{e}_{\sigma(k)}}>0,\) the space \(N^{\sigma}_{\nu_0}\) enlarges and converges to \(\Gamma^{k, \sigma, -}_{a_1}.\)

\item The toral manifolds \(\mathcal{T}^{int}_{\nu_0}\) and \(\mathcal{T}^{ext}_{\nu_0}\) collide (intersect) on \(D^{\sigma, \partial}_{\nu_0}\) and
construct a flow-invariant {\it  bi-stable} toral CW complex associated with the CW complex \(\overline{D^{\sigma, \partial}_{\nu_0}}\).

\item When \(\nu_0a_{2\mathbf{e}_{\sigma(k)}}>0\) and \(\nu_0\) is outside of the interval \([\nu_0^{\min}, \nu_0^{\max}],\) the vector field
\eqref{S1DeGen} does not admit any flow-invariant hypertorus.
\item For \(\nu_0a_{2\mathbf{e}_{\sigma(k)}}<0,\) there is precisely one flow-invariant toral CW complex associated with
\(\Gamma^{k, \sigma, -}_{a_1}.\) When sign of \(\nu_0\) changes, \ie \(\nu_0a_{2\mathbf{e}_{\sigma(k)}}>0,\) this toral CW complex turns to be \(\mathcal{T}^{ext}_{\nu_0}\). In other words, the CW complex associated with this toral CW complex shrinks to the CW complex \(D^{\sigma}_{\nu_0}\subsetneq \Gamma^{k, \sigma, -}_{a_1}.\) This toral CW complex coalesces with the secondary toral CW complex \(\mathcal{T}^{int}_{\nu_0}\) on \(D^{\sigma}_{\nu_0}\) and disappear when the parameter crosses the transition variety \(T_{SN}\) defined by \eqref{Tmin1}.
\end{enumerate}
\end{thm}
\bpr Let \(\nu_0a_{2\mathbf{e}_{\sigma(k)}}>0\). The squared radiuses \({r^{\pm}_{\sigma(i)}}^2\) are positive when \(\frac{{a^{\sigma(i)}_1(C)}^2}{4a^{\sigma(i)}_2(C)}<\nu_0\) for \(a^{\sigma(i)}_2(C)<0.\) Hence, for \(C\in D^{\sigma}_{\nu_0}\subseteq\Gamma^{k, \sigma, -}_{a_1},\) two Clifford hypertori bifurcate from the origin through a secondary saddle-node type leaf-bifurcation at
\ba
&T^C_{SN}:=\Big\{\nu_0\big|\, \nu_0=\frac{{a^{\sigma(k)}_1(C)}^2}{4a^{\sigma(k)}_2(C)}\Big\}.&
\ea Here, one of the hypertori live inside the other one. On the other hand,
\bas
&\frac{{a^{\sigma(i)}_1(C)}^2}{4a^{\sigma(i)}_2(C)}= \frac{\left(\sum^k_{l=1} a_{\mathbf{e}_{\sigma(l)}}{c_{\sigma(l)}}^2\right)^2}{4\sum_{1\leq l\leq j\leq k} a_{\mathbf{e}_{\sigma(l)}+\mathbf{e}_{{\sigma(j)}}}{{c_{\sigma(l)}}^2{c_{\sigma(j)}}^2}}\quad\hbox{ and } \quad\sum^k_{i=1}{c_{\sigma(i)}}^2=1.&
\eas In order to find possible critical values of the parameter \(\nu_0,\) we consider the Lagrange function
\ba
&L(C, \lambda):= \frac{\left(\sum^k_{l=1} a_{\mathbf{e}_{\sigma(l)}}{c_{\sigma(l)}}^2\right)^2}{4\sum_{1\leq j\leq l\leq k} a_{\mathbf{e}_{\sigma(l)}+\mathbf{e}_{{\sigma(j)}}}{{c_{\sigma(l)}}^2{c_{\sigma(j)}}^2}}+\lambda \left(\sum^k_{i=1}{c_{\sigma(i)} }^2-1\right),&
\ea where \(\lambda\) is the Lagrange multiplier. Let \(a_1:= \sum^k_{l=1} a_{\mathbf{e}_{\sigma(l)}}{c_{\sigma(l)}}^2\) and \(a_2:=\sum_{1\leq j\leq l\leq k} a_{\mathbf{e}_{\sigma(l)}+\mathbf{e}_{{\sigma(j)}}}{{c_{\sigma(l)}}^2{c_{\sigma(j)}}^2}.\) Then, we have
\begin{eqnarray*}
&\nabla_{C, \lambda}L=\frac{a_1}{2a_2}(\nabla_{C}\, a_1-\frac{a_1}{2a_2}\nabla_{C}\, a_2)+2\lambda C+ \big(\sum^k_{i=1}{c_{\sigma(i)} }^2-1\big){\textbf{e}_{k+1}}.&
\end{eqnarray*} We show that \(\nabla_{C, \lambda}L=0\) has no roots on the manifold \(\mathbb{S}^{k-1, \sigma}_{>0}\). Since \(\langle C,\nabla_{C}\, a_1\rangle=2a_1,\) \(\langle C,\nabla_{C}\, a_2\rangle=4a_2,\) and \(\langle C, C\rangle=\sum^k_{i=1}{c_{\sigma(i)} }^2=1,\) we compute
\begin{eqnarray*}
&\langle C, \nabla_{C, \lambda}L\rangle=\frac{a_1}{2a_2}\left(\langle C,\nabla_{C}\, a_1\rangle-\frac{a_1}{2a_2}\langle C, \nabla_{C}\, a_2\rangle\right)+2\lambda\langle C, C\rangle= 2\lambda\langle C, C\rangle=0.&
\end{eqnarray*} Hence, \(\lambda=0.\) Now assume that \(C\in \overline{\mathbb{S}^{k-1, \sigma}_{>0}}\) is a solution of \(\nabla_{C, \lambda}L=0=\frac{a_1}{2a_2}(\nabla_{C}\, a_1-\frac{a_1}{2a_2}\nabla_{C}\, a_2).\) Thus,
\begin{eqnarray*}
&a_{\mathbf{e}_{\sigma(i)}}c_{\sigma(i)}=\frac{c_{\sigma(i)}{a_1}\sum_{1\leq l\leq k}{a_{\mathbf{e}_{{\sigma(i)}}+\mathbf{e}_{{\sigma(l)}}}}{c_{\sigma(l)}}^2}{2{a_2}}+\frac{a_{2\mathbf{e}_{{\sigma(i)}}}{a_1} {c_{\sigma(i)}}^3}{2{a_2}} &
\\&\hbox{ and } \quad a_{\mathbf{e}_{\sigma(i)}}=\frac{{a_1}}{2{a_2}}(\sum_{1\leq l\leq k}{a_{\mathbf{e}_{{\sigma(i)}}+\mathbf{e}_{{\sigma(l)}}}}{c_{\sigma(l)}}^2+a_{2\mathbf{e}_{{\sigma(i)}}}{c_{\sigma(i)}}^2) \qquad \hbox{ if }\quad c_{\sigma(i)}\neq0.&
\end{eqnarray*} Hence, for \(c_{\sigma(i)}c_{\sigma(j)}\neq 0\) we have \(\sign(a_{\mathbf{e}_{\sigma(i)}}a_{\mathbf{e}_{{\sigma(j)}}})>0.\) Otherwise,
 \(c_{\sigma(i)}=0\) or \(c_{\sigma(j)}=0.\) Therefore, the critical values do not occur for \(C\)-values on \(\mathbb{S}^{k-1, \sigma}_{>0}\).

Let \(\gamma\in S^{l, \sigma}_n,\) \(c_{\gamma(i)}c_{\gamma(j)}\neq0\) for all \(i\neq j\leq l<k\) and \(c_{\gamma(i)}=0\) for \(i>l.\) Since
\bes \frac{a_1}{a_2}= \frac{{a^{\gamma(l)}_1(C)}^2}{a^{\gamma(l)}_2(C)},\; \nabla_{C}\, a_1=\frac{a_1}{2a_2}\nabla_{C}\, a_2,\;
\; a_2=\langle\mathcal{C}_\gamma, M_\gamma {\mathcal{C}_\gamma}\rangle, \; \hbox{ and } \; \mathbf{a} \hbox{ as given by \eqref{Agama}},
\ees we have \(\mathbf{a}_\gamma=\frac{a^{\gamma(l)}_1(C)}{a^{\gamma(l)}_2(C)} M_\gamma {\mathcal{C}_\gamma}\). Since \(M_\gamma\) is invertible,
\begin{eqnarray*}
&{M_\gamma}^{-1}\mathbf{a}_\gamma=\frac{a^{\gamma(l)}_1(C)}{a^{\gamma(l)}_2(C)}{\mathcal{C}_\gamma}\; \hbox{ and }\; \frac{1}{4}\langle{\mathbf{a}_\gamma}, {M_\gamma}^{-1}\mathbf{a}_\gamma\rangle= \frac{1}{4}\frac{a^{\gamma(l)}_1(C)}{a^{\gamma(l)}_2(C)}\langle{\mathbf{a}_\gamma}, {\mathcal{C}_\gamma} \rangle= \frac{1}{4}\frac{{a^{\gamma(l)}_1(C)}^2}{a^{\gamma(l)}_2(C)}.&
\end{eqnarray*} Therefore, the local extremum values of \(\frac{1}{4}\frac{{a^{\gamma(l)}_1(C)}^2}{a^{\gamma(l)}_2(C)}\) is given by \(\langle\frac{1}{4}{\mathbf{a}_\gamma}, {M_\gamma}^{-1}\mathbf{a}_\gamma\rangle.\) Thereby, the critical values of parameters are \(\nu_{\min}\) and \(\nu_{\max}\) given by equations \eqref{nuMinMax} and we always have \(\nu_{\min}\leq \frac{{a^{\sigma(i)}_1(C)}^2}{4a^{\sigma(i)}_2(C)}\leq
\nu_{\max}.\)

When \(\nu_0a_{2\mathbf{e}_{\sigma(k)}}>0,\) a secondary flow-invariant internal hypertoral CW complex manifold \(\mathcal{T}^{int}_{\nu_0}\) associated with
\(D^{\sigma, \circ}_{\nu_0}\subset\Gamma^{k, \sigma, -}_{a_1}\) defined in \eqref{PartofGam} bifurcates from the origin through an instant bifurcation at \(T_{2Pch}\) given by \eqref{T2Pch}. It shrinks through a continuous leaf-dependent family of saddle-node type bifurcation of hypertori.
In other words, the external hypertoral manifold \(\mathcal{T}^{ext}_{\nu_0}\) bifurcates from this leaf-dependent continuous hypertoral saddle-node type bifurcation when \(\nu_0a^{\sigma(k)}_2(C)=\frac{{a^{\sigma(k)}_1(C)}^2}{4}.\) Therefore, we call the internal manifold by a {\it secondary} hypertoral manifold while the external manifold is referred by a {\it tertiary hypertoral manifold}. Part of these flow-invariant manifolds associated with \(\mathbb{S}^{k-1, \sigma}_{>0}\) is homeomorphic to \(\mathbb{B}^{k-1}\times \mathbb{T}_k\) and is relatively compact. Their topological closure, however, represent the actual flow-invariant toral CW complex bifurcated compact manifolds \(\mathcal{T}^{int}_{\nu_0}\) and \(\mathcal{T}^{ext}_{\nu_0}\). The proof for the toral CW complex structure of \(\mathcal{T}^{int}_{\nu_0}\) and \(\mathcal{T}^{ext}_{\nu_0}\) is similar to Theorem \ref{LemS1Gamma+} and is thus omitted for briefness. For \(\nu_0a_2(C)<0,\) there is always a hypertorus corresponding with the CW complex \(\Gamma^{k, \sigma, -}_{a_1}\) in the \(2k\)-cell.

The radiuses of the tori converge to \(\frac{-a^{\sigma(i)}_1(C)}{a^{\sigma(i)}_2(C)}\) for \(a^{\sigma(k)}_1(C)a^{\sigma(k)}_2(C)<0\) and \(i\leq k,\) when \(\nu_0\) approaches zero. Hence, the existing hypertorus for \(\nu_0a_2(C)<0\) continues to live as the external hypertorus when \(\nu_0\) changes its sign and remains sufficiently small associated with \(C\in D^{\sigma}_{\nu_0}\subset\Gamma^{k, \sigma, -}_{a_1}.\) More precisely, when \(\nu_0a_2(C)>0\) and \(\nu_0\) lie in the interval \((\nu_{\min}, \nu_{\max})\), there are always two hypertori (one inside the other one) corresponding to the \(\MKC\)-leaf for \(C\in D^{\sigma}_{\nu_0}\subset\Gamma^{k, \sigma, -}_{a_1}\subset \mathbb{S}_{>0}^{k-1},\) while there is no hypertorus associated with \(C\in\Gamma^{k, \sigma, +}_{a_1}\subset \mathbb{S}_{>0}^{k-1}\). When \(\nu_0a_2(C)>0\) and \(\nu_0\) is outside the interval \((\nu_{\min}, \nu_{\max}),\) there is no flow-invariant hypertorus throughout the \(2k\)-cell.
\epr

\begin{exm}[Case \(n=k=2\)]\label{Ex2} Let \(n=k=2,\) \(a_{\mathbf{e}_1}=-a_{\mathbf{e}_2}=1,\) and \(a_{2\mathbf{e}_1}= a_{2\mathbf{e}_2}=a_{\mathbf{e}_1+\mathbf{e}_2}=1.\)

\begin{figure}[h]
\floatbox[{\capbeside\thisfloatsetup{capbesideposition={left,top},capbesidewidth=8.5cm}}]{figure}[\FBwidth]
{\caption{The space \(\mathbb{S}_{>0}^{1, I}\) and its partition given by the spaces \(\Gamma^{2, I}_{a_1}, \Gamma^{2, I, +}_{a_1}, \Gamma^{2, I, -}_{a_1},\) \(N^{I}_{\nu_0}\) and \(D^{I}_{\frac{1}{13}}.\) When \(\nu_0>0,\) there exist two invariant toral CW complex over \(\overline{D^{I}_{\nu_0}}\) and there is no hypertorus corresponding with \(C\in N^{I}_{\nu_0}\cup\Gamma^{2, I, +}_{a_1}.\) There is a toral CW complex over the entire CW complex \(\overline{\mathbb{S}_{>0}^{1, I}}\) for \(\nu_0<0.\) A bistable toral CW complex exists over the CW complex \(\overline{D^{I, \partial}_{\nu_0}}\) when \(0<\nu_0\leq \nu_0^{max}.\) }\label{DNGam-}}
{\;\includegraphics[width=7.5cm,height=5.2cm]{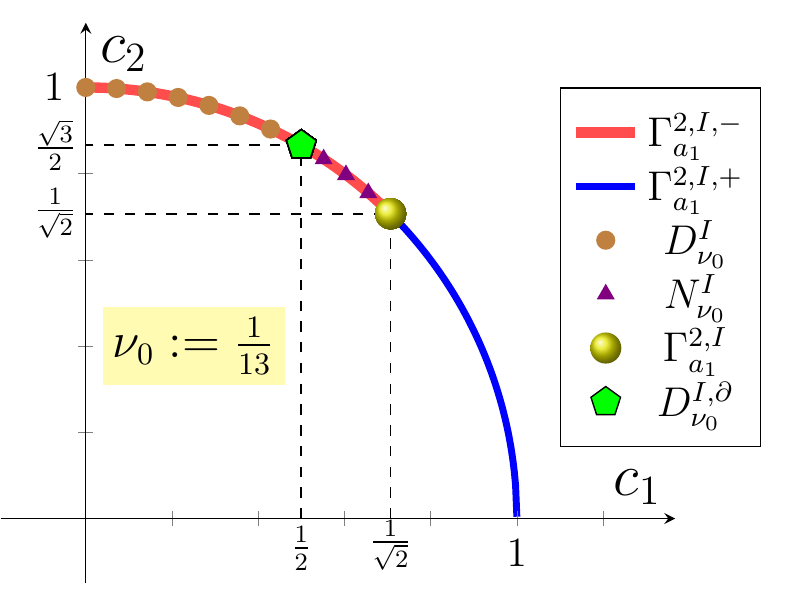}}
\end{figure}

\noindent Then, for \(C=(c_1, c_2)\in \mathbb{S}_{>0}^{1, I}, c_1\neq0,\) we have
\begin{eqnarray*}
&(a^1_1(C), a^1_2(C))= (\frac{{c_1}^2-{c_2}^2}{{c_1}^2}, \frac{{c_1}^4+{c_1}^2{c_2}^2+{c_2}^4}{{c_1}^4})\; \hbox{ and } \; (a^2_1(C), a^2_2(C))= (\frac{{c_1}^2-{c_2}^2}{{c_2}^2}, \frac{{c_1}^4+{c_1}^2{c_2}^2+{c_2}^4}{{c_2}^4}).&
\end{eqnarray*}

\begin{figure}[h!]
\centering
\subfloat[Red curves represents the \(\rho_1\)-radiuses of two tori for \(\nu_0>0.\)\label{Rho1} ]{
\includegraphics[width=3.2in,height=2in]{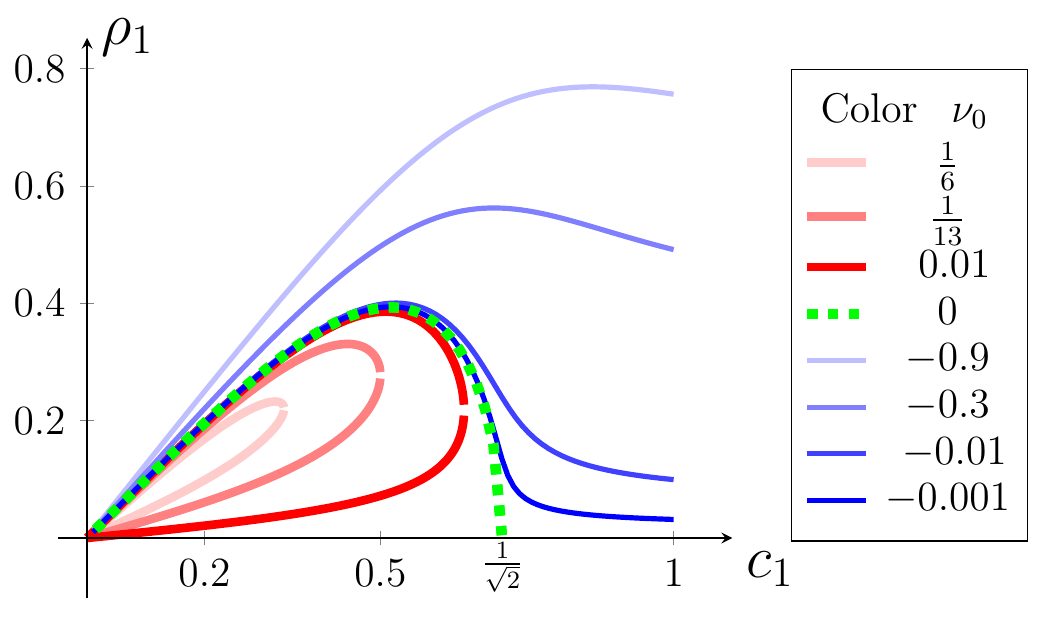}}\,
\subfloat[Blue curves represents the \(\rho_2\)-radiuses of a torus for \(\nu_0<0.\)\label{Rho2} ]{
\includegraphics[width=3.2in,height=2in]{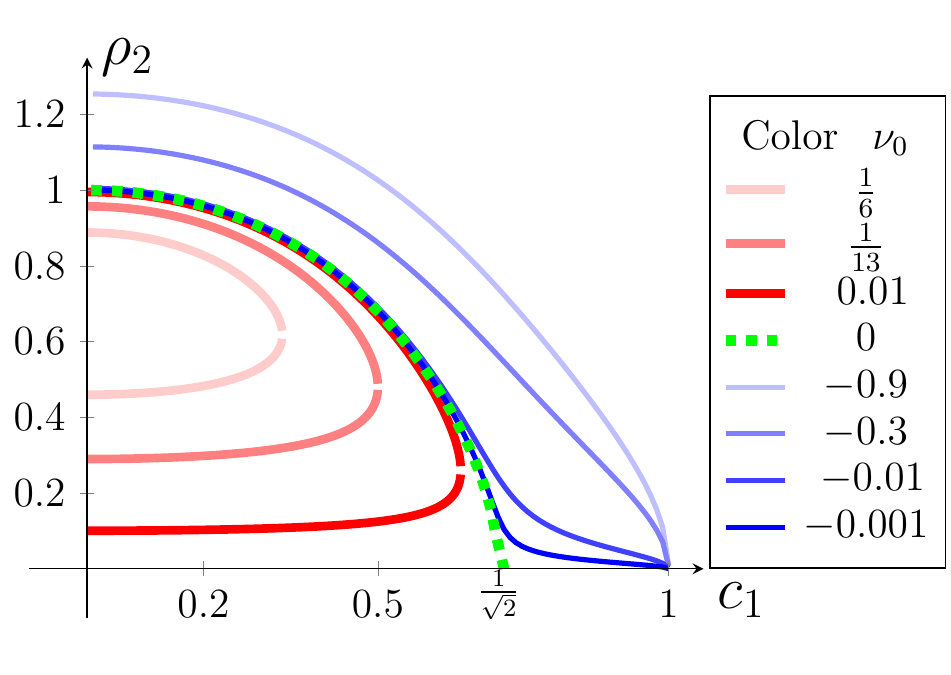}}\,
\caption{The radius curves of the invariant toral CW complexes versus \(c_1,\) where \((c_1, c_2)\in \overline{\mathbb{S}_{>0}^{1, I}}.\) }\label{Fig6}
\end{figure}
\noindent Thus, \(\Gamma^{2, I}_{a_1}= \{(\frac{\sqrt{2}}{2}, \frac{\sqrt{2}}{2})\}\) and \(\Gamma^{2, I, \pm}_{a_1}= \left\{C\in \mathbb{S}_{>0}^{1, I}|\; \sign({c_1}^2-{c_2}^2)=\pm1\right\}\). In particular, \(\Gamma^{2, I, +}_{a_1}= \left\{C\in \mathbb{S}_{>0}^{1, I}|\; \frac{1}{\sqrt{2}}<c_1\leq 1\right\}.\) Further,
\bas &D^{I}_{\frac{1}{13}}:=\left\{(c_1, c_2)\in \mathbb{S}_{>0}^{1, I}| 0\leq c_1< \frac{1}{2}\right\}, \quad D^{I, \partial}_{\frac{1}{13}}:=\left\{(\frac{1}{2}, \frac{\sqrt{3}}{2})\right\},\\&
 \hbox{ and }\quad  N^{I}_{\frac{1}{13}}:= \left\{C\in \Gamma^{2, I, -}_{a_1}| \frac{1}{2}< c_1< \frac{\sqrt{2}}{2}\right\}.&
\eas These are depicted in Figure \ref{DNGam-}. The \(\rho_1\)- and \(\rho_2\)-radiuses of tori corresponding with positive and negative values of \(\nu_0>0\) are depicted in Figure \ref{Rho1} and \ref{Rho2}. Red closed curves demonstrate two flow-invariant tori corresponding with the CW complex \(D^{I}_{\nu_0}\) for \(\nu_0= 0.01, \frac{1}{13}, \frac{1}{6}.\) The blue curves demonstrates an invariant toral CW complex on \(\mathbb{S}_{>0}^{1, I}\) for \(\nu_0= -0.9, -0.3, -0.01, -0.001.\) Part of the blue curves corresponding with \(\Gamma^{2, I, +}_{a_1}\) coalesces to the origin and disappear when \(\nu_0\) approaches to zero, \ie the green curve corresponds with \(\nu_0=0\). The family of tori collapse to an invariant limit cycle when \(c_1\) converges to either zero or \(1.\)

\end{exm}

\begin{exm} Let \(n=k=3,\) \(a_{\mathbf{e}_1}=a_{\mathbf{e}_3}=-a_{\mathbf{e}_2}=1,\) \(a_{\mathbf{e}_i+\mathbf{e}_j}=1\) for all \(1\leq i\leq j\leq 3,\) and \(\sigma=I\) as the identity permutation.

\begin{figure}[h]
\centering
\floatbox[{\capbeside\thisfloatsetup{capbesideposition={left,top},capbesidewidth=7cm}}]{figure}[\FBwidth]
{\caption{Case \(k=n=3,\) \(a_{\mathbf{e}_1}=a_{\mathbf{e}_3}=-a_{\mathbf{e}_2}=1,\) and \(a_{\mathbf{e}_i+\mathbf{e}_j}=1\) for all \(1\leq i\leq j\leq 3.\) The space \(\mathbb{S}_{>0}^{2, I}\) and \(\Gamma^{3, I}_{a_1}, \Gamma^{3, I, +}_{a_1}, \Gamma^{3, I, -}_{a_1},\) \(N^{I}_{\nu_0}\) and \(D^{I}_{\nu_0}.\) The space \(D^{I, \partial}_{\nu_0}\) is indexed with \(\nu_0=0.1\) and \(0.25,\) while
\(\Gamma^{3, I}_{a_1}= \{(c_1, c_2, c_3)| c_2=\frac{\sqrt{2}}{2}, {c_1}^2+{c_3}^2=0.5\}\). There are two toral CW complexes \(\mathcal{T}^{int}_{\nu_0}\) and \(\mathcal{T}^{ext}_{\nu_0}\) over \(\overline{D^{I, \partial}_{\nu_0}}\) and no other toral object elsewhere when \(\nu_0>0.\) For negative values of \(\nu_0,\) the toral CW complex is associated with the whole space \(\mathbb{S}_{>0}^{2, I}.\)
}\label{DNGam-}}
{\;\includegraphics[width=4.2in,height=2.7in]{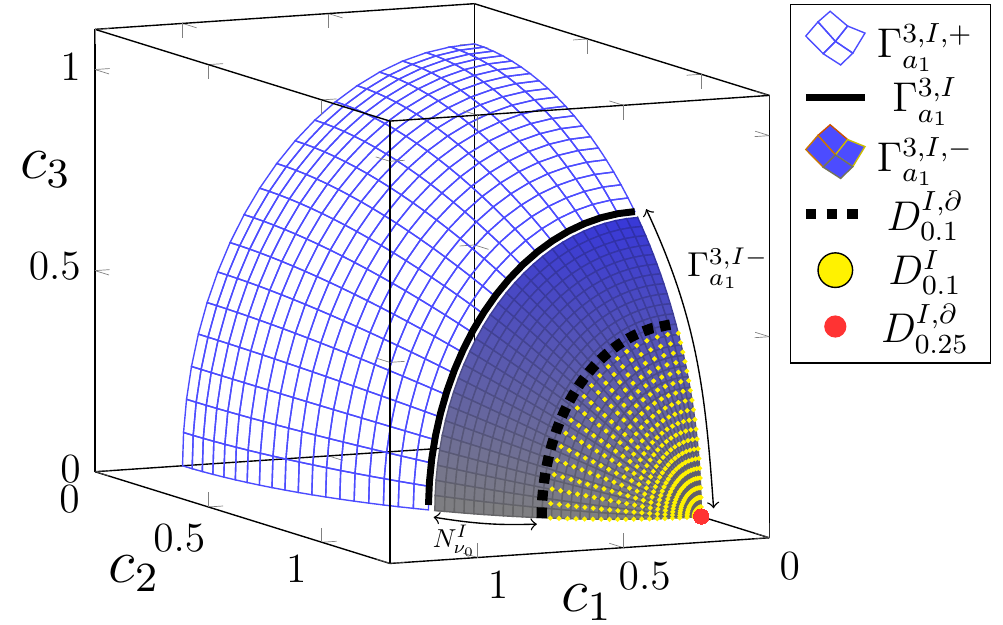}}
\end{figure}

\noindent Thus,  \(\Gamma^{3, I}_{a_1}=L_{(a_{\mathbf{e}_{i}})^3_{i=1}}\cap \mathbb{S}^{2, I}_{>0}= \{(c_1, c_2, c_3)\in \mathbb{S}_{>0}^{2, I}| c_2=\frac{\sqrt{2}}{2}, {c_1}^2+{c_3}^2=0.5\}\) and
\bas
&\Gamma^{3, I, \pm}_{a_1}= \left\{C\in \mathbb{S}_{>0}^{2, I}|\, \sign(1-2{c_2}^2)=\pm1\right\}.&
\eas Further,
\bas &D^{I}_{\nu_0}:=\left\{C\in \Gamma^{3, I, -}_{a_1}|\, \nu_0< \frac{{a^1_1(C)}^2}{4a^1_2(C)}\leq \frac{1}{4}\right\},&\\
&D^{I, \partial}_{\nu_0}:=\left\{C\in \Gamma^{3, I, -}_{a_1}|\, \nu_0=\frac{{a_1(C)}^2}{4a_2(C)}\right\}, \hbox{ and } N^{I}_{\nu_0}:= \Gamma^{3, I, -}_{a_1}\setminus (D^{I}_{\nu_0}\sqcup D^{I, \partial}_{\nu_0}). &
\eas where \(\nu_{\min}=0\) and \(\nu_{\max}=\frac{1}{4}\); see Figure \ref{DNGam-}. Two toral CW complex \(\mathcal{T}^{int}_{\nu_0}\) and \(\mathcal{T}^{ext}_{\nu_0}\) exists when \(\nu_0>0.\) These are associated with \(\overline{D^{I}_{\nu_0}},\) that is depicted in Figure \ref{DNGam-} with yellow bullets on part of the blue region; the blue region stands for \(\Gamma^{3, I, -}_{a_1}\). For positive values of \(\nu_0,\) there is a toral CW complex associated with \(\overline{\mathbb{S}_{>0}^{2, I}}.\) Part of this toral manifold associated with \(\Gamma^{3, I, +}_{a_1}\)
simultaneously collapses with the origin (\ie all radiuses of the tori converge to zero) as soon as the parameter \(\nu_0\) converges to zero. In this case, we only have a toral CW complex associated with \(\overline{\Gamma^{3, I, -}_{a_1}}.\) When \(\nu_0\) further increases from zero, this toral manifold shrinks to be only associated with \(\overline{D^{I}_{\nu_0}}\); this turns out to be \(\mathcal{T}^{ext}_{\nu_0}\). More precisely, both the internal and external toral CW complexes \(\mathcal{T}^{int}_{\nu_0}\) and \(\mathcal{T}^{ext}_{\nu_0}\) exist over \(\overline{D^{I}_{\nu_0}}.\) The intersection of the manifolds \(\mathcal{T}^{int}_{\nu_0}\) and \(\mathcal{T}^{ext}_{\nu_0}\) is a bistable toral CW complex on \(D^{I, \partial}_{\nu_0}\) for all \(0<\nu_0\leq \nu_0^{max}=0.25\). When \(\nu_0>0\) increases and approaches to \(0.25,\) \(D^{I, \partial}_{\nu_0}\) shrinks to the point \(D^{I, \partial}_{0.25}= \{(0, 1, 0)\}\); this is depicted by red bullet in Figure \ref{DNGam-}. Hence, the toral manifolds \(\mathcal{T}^{int}_{\nu_0}\) and \(\mathcal{T}^{ext}_{\nu_0}\) shrink and collapse to a bistable limit cycle. Then, the limit cycle disappears when \(\nu_0>0\) crosses over the transition variety \(T_{SN}= \{\nu_0=0.25\}\) given by \eqref{Tmin1}.
\end{exm}

\begin{thm}\label{TopEqu} Consider the parametric vector field \eqref{CellS1DeGen}, and assume that the condition \eqref{s1DoblTor} and hypotheses in Theorem \ref{Lem7.5} hold. Then, the varieties \eqref{T2Pch} and \eqref{Tmin1} are the only \(2k\)-cell bifurcation varieties for the differential system corresponding with \eqref{CellS1DeGen}. More precisely, the parametric vector fields \(v_\sigma(\mathbf{r}, \theta, \nu_0^1)\) and \(v_\sigma(\mathbf{r}, \theta, \nu_0^2)\) are topologically equivalent when one of the following holds.
\begin{enumerate}
\item For \(\delta=1, 2,\) \(\nu_0^{\min}<\nu^\delta_0<\nu_0^{\max}\) and \(\nu^\delta_0a_{2\mathbf{e}_{\sigma(k)}}>0.\)
\item \(\nu^\delta_0a_{2\mathbf{e}_{\sigma(k)}}<0\) for both \(\delta=1, 2.\)
\item When \(\nu^\delta_0a_{2\mathbf{e}_{\sigma(k)}}>0\) and \(\nu^\delta_0\) for \(\delta=1, 2\) is outside of the interval \([\nu_0^{\min}, \nu_0^{\max}].\)
\end{enumerate}

\end{thm}
\bpr The idea is to use a homeomorphism on the CW complex subspaces of \(\mathbb{S}_{>0}^{k-1, \sigma}\) to transform flow-invariant leaves associated with \(v_\sigma(\mathbf{r}, \theta, \nu_0^1)\) to those associated with \(v_\sigma(\mathbf{r}, \theta, \nu_0^2)\).

Let \(\nu_0^{\min}<\nu^\delta_0<\nu_0^{\max}\) and \(\nu^\delta_0a_{2\mathbf{e}_{\sigma(k)}}>0\) for both \(\delta=1, 2.\) The CW complexes \(\overline{D^{\sigma}_{\nu^1_0}}\) and \(\overline{D^{\sigma}_{\nu^2_0}}\) are homeomorphic. Further, the complement of these spaces are also homeomorphic; see the proof of Lemma \ref{TopLem} and argument above Theorem \ref{Lem7.5}. Assume that these are given by the homeomorphism
\bes
h: \overline{\mathbb{S}_{>0}^{k-1, \sigma}}\rightarrow \overline{\mathbb{S}_{>0}^{k-1, \sigma}},\hbox{ where }\;
h\big(\,\overline{D^{\sigma}_{\nu^1_0}}\,\big)=\overline{D^{\sigma}_{\nu^2_0}}\;\hbox{ and }\; h\big({\overline{D^{\sigma}_{\nu^1_0}}}^c\big)={\overline{D^{\sigma}_{\nu^2_0}}}^c.
\ees Given the radiuses \({r^{\pm}_{\sigma(i)}}^2(\nu_0, C)\) in equation \eqref{rbar}, the map
\(\tilde{h}: \mathcal{T}^{int}_{\nu^1_0}\rightarrow \mathcal{T}^{int}_{\nu^2_0}\) defined by \(\tilde{h}(\mathbf{r}(\nu^1_0, C), \theta):= (\mathbf{r}(\nu^2_0, h(C)), \theta)\) is a homeomorphism between the toral CW complex manifolds \(\mathcal{T}^{int}_{\nu^1_0}\) and \(\mathcal{T}^{int}_{\nu^2_0}.\) Similarly, we may assume that \(\mathcal{T}^{ext}_{\nu^1_0}\) and \(\mathcal{T}^{ext}_{\nu^2_0}\) are \(\tilde{h}\)-homeomorphic.

We shall extend the homeomorphism \(\tilde{h}\) to a flow-invariant homeomorphism \(\tilde{h}: \overline{\mathcal{M}_{k, \sigma}}\rightarrow\overline{\mathcal{M}_{k, \sigma}}.\) Thus, we merely need to consider the extension to the space \(\mathcal{M}_{l, \gamma}\) for any \(\gamma\in S^{l, \sigma}_n\) and \(l\leq k.\)

Let \(\mathbf{r}(t, \mathbf{r}^{\delta,\gamma}, \nu^\delta_0)\) stand for the trajectory of \(\mathbf{r}\) in action-angle \((\mathbf{r}, \theta)\)-coordinates corresponding with \(v_\sigma(\mathbf{r}, \theta, \nu^\delta_0)\) with the initial condition \(\mathbf{r}(0, \mathbf{r}^{\delta,\gamma}, \nu^\delta_0)= \mathbf{r}^{\delta,\gamma},\) \(\delta=1, 2\). Denote \(\mathbf{r}_{\gamma(l)}\) for \({\gamma(l)}\)-th component of \(\mathbf{r}\). Assume that \((\mathbf{r}_*, \theta_*)\in \mathcal{M}^C_{l, \gamma}\subset\mathcal{M}_{l, \gamma},\) \(C\in D^{\gamma}_{\nu^1_0},\) and \({\mathbf{r}_*}_{\gamma(l)}=r_*.\) Let \(r^{-}_{\gamma(l)}(\nu^\delta_0, h^{\delta-1}(C))\) and \(r^{+}_{\gamma(l)}(\nu^\delta_0, h^{\delta-1}(C))\) denote the \({\gamma(l)}\)-th radiuses of the flow-invariant internal and external tori associated with \(C, h(C)\in D^{\gamma}_{\nu^1_0},\) \(\nu^1_0\) and \(\nu^1_0,\) respectively.

Now consider the positive numbers \(r^{\delta, \gamma}_{0}:=\frac{{r^{-}_{\gamma(l)}}(\nu^\delta_0, h^{\delta-1}(C))}{2},\) \(r^{\delta, \gamma}_{1}:=\frac{{r^{-}_{\gamma(l)}}(\nu^\delta_0, h^{\delta-1}(C))+{r^{+}_{\gamma(l)}}(\nu^\delta_0, h^{\delta-1}(C))}{2},\) \(r^{\delta, \gamma}_{2}:=2{r^{+}_{\gamma(l)}}(\nu^\delta_0, h^{\delta-1}(C)),\) and \(\mathbf{r}^{\delta, \gamma}_{\iota}:= \frac{r^{\delta, \gamma}_{\iota}}{c_{\gamma(l)}} C\in \mathcal{M}^C_{l, \gamma},\) for \(\iota=0, 1, 2.\) Note that \(r^{\delta, \gamma}_{\iota}\) is the \(\gamma(l)\)-th component of \(\mathbf{r}^{\delta, \gamma}_{\iota}, \iota=0, 1, 2.\) Then, define \(\tilde{h}(\mathbf{r}^{1, \gamma}_{\iota}):=\mathbf{r}^{2, \gamma}_{\iota}.\) Note that we have \(r^{\delta, \gamma}_{0}< {r^{-}_{\gamma(l)}}(\nu^\delta_0, h^{\delta-1}(C))< r^{\delta, \gamma}_{1}< {r^{+}_{\gamma(l)}}(\nu^\delta_0, h^{\delta-1}(C))< r^{\delta, \gamma}_{2},\) where \(\delta=1, 2.\)

For any point \(r_*\) such that \(r_*< {r^{-}_{\gamma(l)}}(\nu^1_0, C),\) there is a unique time \(t_{r_*}\) (positive or negative) such that \(r_*= \mathbf{r}_{\gamma(l)}(t_{r_*}, \mathbf{r}^{1, \gamma}_0, \nu^1_0),\) where \(\mathbf{r}(0, \mathbf{r}^{1, \gamma}_0, \nu^1_0)=\mathbf{r}^{1,\gamma}_0.\) Then, we define  \(\tilde{h}((\mathbf{r}_*, \theta_*))= (\mathbf{r}(t_{r_*}, \mathbf{r}^{2, \gamma}_0, \nu^2_0), \theta^*).\) When either \({r^{-}_{\gamma(l)}}(\nu^1_0, C)<r_*<{r^{+}_{\gamma(l)}}(\nu^1_0, C),\) or \({r^{+}_{\gamma(l)}}(\nu^1_0, C)<r_*,\) we may similarly introduce \(\tilde{h}((\mathbf{r}_*, \theta_*))= (\mathbf{r}(t_{r_*}, \mathbf{r}^{2, \gamma}_1, \nu^2_0), \theta_*)\) and \(\tilde{h}((\mathbf{r}_*, \theta_*))= (\mathbf{r}(t_{r_*}, \mathbf{r}^{2, \gamma}_2, \nu^2_0), \theta_*),\) respectively. Here, \(t_{r_*}\) stands for the time required for \(\mathbf{r}_{\gamma(l)}(t_{r_*}, \mathbf{r}^{1,\gamma}_1, \nu^1_0)= r_*\) or \(\mathbf{r}_{\gamma(l)}(t_{r_*}, \mathbf{r}^{1,\gamma}_2, \nu^1_0)= r_*,\) accordingly. This yields a flow-invariant construction for \(\tilde{h}\) on the family of leaf-manifolds \(\MKC\) for all \(C\in D^{\gamma}_{\nu^1_0}.\) When \(C\) converges to \(D^{\gamma, \partial}_{\nu_0}\), the internal and external hypertori coalesce into a single hypertori and the dynamics associated with the radius interval \(({r^{-}_{\gamma(l)}}(\nu^1_0, C), {r^{+}_{\gamma(l)}}(\nu^1_0, C))\) is omitted. This justifies the bi-stability of the toral CW complex associated with \(D^{\gamma, \partial}_{\nu_0}\). Thus, the flow-invariant homeomorphism \(\tilde{h}\) is well-defined on \(\cup_{C\in \overline{D^{\gamma}_{\nu^1_0}}, \gamma\in S^{l, \sigma}_n, l\leq k}\mathcal{M}^C_{l, \gamma}=\overline{\cup_{C\in \overline{D^{\sigma}_{\nu^1_0}}} \MKC}.\)

Now we only need to introduce the homeomorphism \(\tilde{h}\) on \(\overline{\cup_{C\in {\overline{D^{\sigma}_{\nu^1_0}}}^c}\MKC}= \cup_{C\in \overline{D^{\gamma}_{\nu^1_0}}^c, \gamma\in S^{l, \sigma}_n, l\leq k}\mathcal{M}^C_{l, \gamma}.\) Consider \((\mathbf{r}^*, \theta^*)\in \mathcal{M}^C_{l, \gamma},\) \(C\in \overline{D^{\gamma}_{\nu^1_0}}^c,\) and \(\mathbf{r}^*_{\gamma(l)}=r^*.\) Then, there is a unique time \(t_{r^*}\) (positive or negative) such that \(\mathbf{r}^*:=r^*\frac{C}{c_{\gamma(l)}}= \mathbf{r}_{\gamma(l)}(t_{r^*}, \frac{C}{c_{\gamma(l)}}, \nu^1_0)\) where \(\mathbf{r}_{\gamma(l)}(0, \frac{C}{c_{\gamma(l)}}, \nu^1_0)=1.\) Then, \(\tilde{h}((\mathbf{r}^*, \theta^*))\) is defined by \((\mathbf{r}(t_{r^*}, \frac{C}{c_{\gamma(l)}}, \nu^2_0), \theta^*).\)

For the second part, we remark that there is only a flow-invariant hypertoral CW complex associated with \(\Gamma^{k, \sigma, \pm}_{a_1}\) for both \(\nu^1_0\) and \(\nu^2_0.\) Since the space \(\Gamma^{k, \sigma, \pm}_{a_1}\) is independent of \(\nu^1_0\) and \(\nu^2_0,\) we may simply consider the homeomorphism \(h: \overline{\mathbb{S}_{>0}^{k-1, \sigma}}\rightarrow \overline{\mathbb{S}_{>0}^{k-1, \sigma}}\) as the identity map. The rest of the proof is similar to the above. Let \(\nu^\delta_0a_{2\mathbf{e}_{\sigma(k)}}>0\) and \(\nu^\delta_0\) for \(\delta=1, 2\) is outside of the interval \([\nu_0^{\min}, \nu_0^{\max}].\) Then, there is no invariant hypertori associated with neither of the parameters. Hence, we again use \(h\) as the identity map on  \(\overline{\mathbb{S}_{>0}^{k-1, \sigma}}.\) For any \(C\in \mathbb{S}_{>0}^{l-1, \gamma}\) and \(r^*,\) there exists a unique time \(t_{r^*}\) (positive or negative) so that \(r^*\frac{C}{c_{\gamma(l)}}= \mathbf{r}_{\gamma(l)}(t_{r^*}, \frac{C}{c_{\gamma(l)}}, \nu^1_0)\) where \(\mathbf{r}(0, \frac{C}{c_{\gamma(l)}}, \nu^1_0)=\frac{C}{c_{\gamma(l)}}.\) Then, \(\tilde{h}((\frac{r^*}{c_{\gamma(l)}}C, \theta^*))\) is defined by \((\mathbf{r}(t_{r^*}, \frac{C}{c_{\gamma(l)}}, \nu^2_0), \theta^*).\)
\epr


\end{document}